\documentclass{article} % For LaTeX2e
\usepackage{iclr2027_conference,times}

\usepackage{amsmath,amsfonts,bm}

\def\eqref#1{equation~\ref{#1}}
\def\Eqref#1{Equation (\ref{#1})}
\def\1{\bm{1}}

\DeclareMathAlphabet{\mathsfit}{\encodingdefault}{\sfdefault}{m}{sl}
\SetMathAlphabet{\mathsfit}{bold}{\encodingdefault}{\sfdefault}{bx}{n}

\newcommand{\E}{\mathbb{E}}

\DeclareMathOperator{\Tr}{Tr}

\usepackage{hyperref}
\usepackage{url}

\usepackage{xcolor} 
\usepackage{amssymb}   
\usepackage{amsthm}    
\usepackage{amsmath}   
\usepackage{stmaryrd}  
\usepackage{mathrsfs}  
\usepackage{graphicx}
\usepackage{algorithm}
\usepackage{algpseudocode}
\usepackage{url}
\usepackage{float}
\title{SCMO: Stochastic Control for Optimization over Probability Measures on Infinite-Dimensional Spaces}
\usepackage{booktabs}
\usepackage{siunitx}
\usepackage{caption}
\usepackage{makecell}
\usepackage{tabularx}

\author{Hang Cheung\\
\texttt{hang.cheung7373@gmail.com} \\
\And
Jinniao Qiu \\
Department of Mathematics and Statistics \\
University of Calgary \\
Calgary, Alberta, Canada \\
\texttt{jinniao.qiu@ucalgary.ca}
}

\newtheorem{thm}{Theorem}[section]
\newtheorem{cor}[thm]{Corollary}
\newtheorem{lem}[thm]{Lemma}
\newtheorem{prop}[thm]{Proposition}
\newtheorem{exam}[thm]{Example}

\newtheorem{ass}{Assumption}[section]
\newtheorem{rmk}{Remark}[section]

\newcommand{\cP}{\mathcal{P}}

\newcommand{\cH}{\mathcal{H}}

\newcommand{\bP}{\mathbb{P}}
\newcommand{\bR}{\mathbb{R}}

\newcommand{\Law}{\mathcal L}
\newcommand{\op}{\mathrm{op}}

\iclrfinalcopy % Uncomment for camera-ready version, but NOT for submission.
\begin{document}
\maketitle
\lhead{}
\begin{abstract}
We study objective-only optimization of possibly nonconvex and nonsmooth
functionals over probability measures on a separable Hilbert space,
allowing the optimizer to be intrinsically non-Dirac. We introduce SCMO
(Stochastic Control Measure Optimizer), a gradient-free particle method
derived from entropy-regularized stochastic control. After finite-particle
and Galerkin approximations, a Cole--Hopf transform represents the optimal
feedback as a Gibbs-weighted terminal displacement. SCMO approximates this
feedback by sampling context clouds, replacing one particle at a time with
candidate draws, scoring the resulting empirical measures, and applying
exponential reweighting.

SCMO uses separate covariances for candidate proposals and particle updates,
termed matched when equal and nonmatched otherwise; our analysis covers both
settings. In the matched case, we establish
PDE-free qualitative convergence and a projection-first quantitative bound
separating particle, Galerkin-projection, and entropy-regularization errors.
For the practical
multi-context implementation with nonmatched covariance, we prove finite-candidate
consistency and show that the signed, curvature-dependent effect of
nonmatched covariance can reduce the resulting upper bound on the
approximation error. Finally, experiments on function-space, trajectory-law, and contact-rich
manipulation problems show that SCMO handles nonsmooth and nonconvex
objectives, escapes suboptimal local basins, and recovers prescribed
multimodal, non-Dirac law structure.

Code for reproducing our experimental results is available at
\url{https://github.com/HenryCHEUNG7373/SCMO}.
\end{abstract}
\section{Introduction}
\label{sec:introduction}
Many problems in machine learning and control seek a probability law over
high- or infinite-dimensional solutions rather than a single optimizer,
often under nonconvex or nonsmooth objectives; examples include multimodal
robot policies \citep{chi2023diffusionpolicy,ding2025riemannian},
function-space generative models
\citep{kerrigan2024functional,li2026functional}, and nonconvex optimization
over probability measures \citep{luu2024nongeodesic}. Motivated by these settings, we study problems of the form
\begin{equation}
    \mathcal V_0
    :=
    \inf_{\mu\in\mathcal P_2(\mathcal H)}
    G(\mu),
    \label{eq:intro-measure-problem}
\end{equation}
where \(\mathcal H\) is a separable Hilbert space and
\(G:\mathcal P_2(\mathcal H)\to\mathbb R\) may be nonconvex, nonsmooth, and
interacting. It combines three challenges: the
measure-valued decision variable is infinite-dimensional; function-space
implementations must control both particle and Galerkin-projection errors;
and gradients or Lions derivatives may be unavailable or uninformative for
simulated or discontinuous objectives. This motivates an objective-only
method that maintains a flexible empirical law and provides particle-wise
credit assignment.

This work is inspired by
\citet{qiu2026stochastic}, which connects optimization over
\(\mathcal P_2(\mathbb R^d)\) with entropy-regularized mean-field controls and \(N\)-particle approximations. 
However, that analysis
is finite-dimensional and uses the same Brownian covariance in the feedback representation and the particle dynamics. As the corresponding regularization bound scales with
the ambient dimension, the estimate is not uniform as $d\to\infty$. 
The present work nontrivially enlarges this framework to \(\mathcal P_2(\mathcal H)\) for a
separable Hilbert space \(\mathcal H\), using trace-class $Q$-Wiener noise and
Galerkin projection. It also introduces an \(R\)-context estimator that
recovers the unconditional full-cloud feedback as \(R\to\infty\) and separates
the proposal covariance used for feedback estimation from the execution
covariance used in the particle dynamics. We call it SCMO
(\emph{Stochastic Control Measure Optimizer}). Its methodological and
theoretical contributions are summarized below.

% Building on this framework, we develop SCMO
% (\emph{Stochastic Control Measure Optimizer}) for empirical laws on a
% separable, potentially infinite-dimensional Hilbert base space
% \(\mathcal H\). After Galerkin projection, SCMO approximates the feedback by
% evaluating particle-wise candidate replacements against multiple independently
% sampled context clouds and jointly normalizing the resulting scores. SCMO
% retains the stochastic-control representation and particle-wise replacement
% principle of \citet{qiu2026stochastic} and extends that framework in the
% directions summarized below.

\paragraph{Our Contributions.}
Our contributions are fivefold.
\textbf{(i) Hilbert-space formulation and Galerkin approximation:}
we address the optimization over the space of probability measures on infinite-dimensional separable Hilbert spaces, and obtain an objective-only Galerkin implementation
for function-valued laws. The resulting method requires neither particle
gradients nor Lions derivatives. The infinite-dimensionality refers to the base space on
which the measures are supported, rather than to the dimensionality of the
Wasserstein space itself.
\textbf{(ii) Matched-covariance convergence analysis:} we prove a novel PDE-free
qualitative convergence for a class of nonsmooth and discontinuous objectives.
Under stronger regularity, we derive a projection-first quantitative bound
that separates finite-particle, Galerkin-projection, and
regularization errors.
\textbf{(iii) Nonmatched-covariance convergence analysis:}
we distinguish the proposal covariance used to estimate the feedback from the
execution covariance used to update the particles. We derive a signed,
curvature-dependent nonmatched term and show that, under the stated sign
condition, it sharpens the matched-covariance upper approximation certificate.
\textbf{(iv) Multi-context finite-candidate estimation:}
we introduce an estimator using \(R\) independent complete context clouds and
\(S\) candidate replacements per context, with joint normalization over the
resulting \(RS\) scores. We establish consistency and a mean-square error
bound, and connect the practical finite-candidate implementation to the exact
Euler chain.
\textbf{(v) Empirical evaluation:}
experiments on function-space, trajectory-law, and contact-rich Push-T
problems show that SCMO handles nonsmooth, nonconvex, and discontinuous
objectives, escapes suboptimal local basins, and recovers prescribed
multimodal, non-Dirac law structure. Code for reproducing our experimental results is available at
\url{https://github.com/HenryCHEUNG7373/SCMO}.

\section{Related Work}

\paragraph{Closest precursor.}
The closest precursor is \citet{qiu2026stochastic}. Section~\ref{sec:introduction} states the precise
departures developed in the present work.

\paragraph{Entropy-regularized and Gibbs-weighted methods.}
The soft-constrained Schr\"odinger bridge of \citet{garg2024soft} replaces
exact terminal-law matching by a Kullback--Leibler penalty and yields a
geometric-mixture characterization of the optimal terminal law; this differs
from optimizing a general, potentially interacting law functional.
Path-integral methods similarly use Gibbs-weighted rollouts to update controls
\citep{theodorou2010pathintegral,williams2017information}, whereas SCMO uses
context-conditioned weights to update an empirical law. Entropy-SGD and related methods also share
SCMO's general Gibbs/log-partition structure, exponentially weighting nearby
parameter values according to their loss
\citep{chaudhari2017entropy,pittorino2021entropic}. They estimate the resulting
local-entropy gradient through an auxiliary noisy gradient process and update
a single parameter vector. SCMO instead assigns Gibbs weights to context-conditioned
one-particle replacements using the objective values of the complete empirical law
and executes objective-only particle dynamics.

\paragraph{Black-box, particle, and infinite-dimensional optimization.}
CEM, CMA-ES, and natural evolution strategies adapt a parameterized search
distribution to locate high-quality points
\citep{deboer2005crossentropy,hansen2001cma,wierstra2014nes}; when a complete
particle cloud is vectorized, each proposal receives one global score,
whereas SCMO assigns one-particle replacement scores and treats the resulting
law as the solution. SVGD, particle quantization, mean-field Langevin and
stochastic particle descent, entropic fictitious play, and Wasserstein or KL
flows instead use gradients, scores, or variational structure
\citep{liu2016svgd,xu2022quantization,chizat2018global,hu2021meanfield,
nitanda2017stochastic,chen2023entropic,yao2024klflow,
tankala2025meanfield,kook2024meanfield,luu2024nongeodesic}.
Consensus-based optimization is derivative-free but designed to concentrate
on a point minimizer \citep{pinnau2017cbo,carrillo2018cbo}; recent
Hilbert-space extensions and directional pre-basis methods likewise optimize
a point or function \citep{huang2026hilbert,khatab2026consensus,peixoto2026random}. SCMO instead
uses only evaluations of a functional on \(\mathcal P_2(\mathcal H)\) and may
return a multimodal, non-Dirac law.
\section{SCMO: Stochastic-control formulation and finite-candidate update}
Let $(\cH,\langle\cdot,\cdot\rangle,\|\cdot\|)$ be a separable Hilbert space, and let
$\mathcal P_2(\cH)$ denote the space of probability measures on $\cH$ with finite second moments,
equipped with the $2$-Wasserstein metric $\mathcal W_2$. We study the measure-valued optimization
problem $\mathcal V_0 :=
    \inf_{\mu\in\mathcal P_2(\cH)} G(\mu)$, where $G:\mathcal P_2(\cH)\to\bR$ is a possibly nonconvex and nonsmooth functional.

\subsection{Entropy-regularized measure control}
Let $Q$ be a symmetric, positive-definite, trace-class covariance operator on $\cH$. 
Write \(\cH_Q:=Q^{1/2}(\cH)\) for the associated Cameron--Martin space, i.e. the space of directions with finite \(Q\)-energy, equipped with the norm $\|h\|_Q:=\|Q^{-1/2}h\|,\,h\in\cH_Q$. Equivalently, if $\{e_j\}_{j\geq1}$ is an orthonormal eigenbasis of $Q$ with eigenvalues
$\{\lambda_j\}_{j\geq1}$, then $\cH_Q
    =
    \left\{
    h=\sum_{j\ge1}h_j e_j:
    \sum_{j\ge1}\frac{h_j^2}{\lambda_j}<\infty
    \right\}$, $\|h\|_Q^2
    =
    \sum_{j\ge1}\frac{h_j^2}{\lambda_j}$. Let $(W_t)_{t\geq0}$ be a $Q$-Wiener process on $\cH$ defined on a complete filtered probability
space $(\Omega,\mathcal F,(\mathcal F_t)_{t\geq0},\bP)$. For $t\in[0,1]$ and
$\mu\in\mathcal P_2(\cH)$, let $\xi$ be an $\mathcal F_t$-measurable $\cH$-valued random variable
with law $\mu$, independent of the future increments $(W_s-W_t)_{s\geq t}$. Let $\Theta_t$ denote the set of $\cH_Q$-valued progressively
measurable controls satisfying $\E\int_t^1 \|\theta_s\|_Q^2\,ds < \infty$. For $\theta\in\Theta_t$, define
    $X_s^{t,\xi;\theta}
    =
    \xi+\int_t^s \theta_r\,dr + W_s-W_t,
    \, s\in[t,1]$. For $\varepsilon \in (0,1)$, define the entropy-regularized value function
\begin{equation}
    \label{eq:value_function_MF_eps}
    V_\varepsilon(t,\mu)
    :=
    \inf_{\theta\in\Theta_t}
    \left\{
        G\bigl(\mathcal L(X_1^{t,\xi;\theta})\bigr)+\E\left[
            \frac{\varepsilon}{2}
            \int_t^1 \|\theta_s\|_Q^2\,ds
        \right]
    \right\},
    \qquad
    (t,\mu)\in[0,1]\times\mathcal P_2(\cH),
\end{equation}
where the quadratic term $\E\left[
            \frac{1}{2}
            \int_t^1 \|\theta_s\|_Q^2\,ds
        \right]$ by Girsanov's theorem is the relative entropy of the law of $X^{t,\xi;\theta}$ with respect to the law of the uncontrolled process $X^{t,\xi;0}$, providing the rationale for the terminology ``entropy-regularized".
In particular, for an initial law $\mu_0=\mathcal L(\xi_0)$,
\[
    V_\varepsilon(0,\mu_0)
    =
    \inf_{\theta\in\Theta_0}
    \left\{
         G\bigl(\mathcal L(X_1^{0,\xi_0;\theta})\bigr)
         +\E\left[
            \frac{\varepsilon}{2}
            \int_0^1 \|\theta_s\|_Q^2\,ds
        \right]
    \right\}.
\]

Formally applying the dynamic programming principle on $\mathcal P_2(\cH)$ yields the mean-field
Hamilton--Jacobi-Bellman (HJB) equation
\begin{equation}
    \label{eq:Master_eq_inf}
    \begin{cases}
    \displaystyle
    \partial_t V_\varepsilon(t,\mu)
    -
    \frac{1}{2\varepsilon}
    \int_{\cH}
        \|Q^{1/2}\partial_\mu V_\varepsilon(t,\mu)(x)\|^2
    \,\mu(dx)
    +
    \frac{1}{2}
    \int_{\cH}
        \operatorname{Tr}
        \left(
            Q\,\partial_x\partial_\mu V_\varepsilon(t,\mu)(x)
        \right)
    \,\mu(dx)
    =0,
    \\[1.2em]
    V_\varepsilon(1,\mu)=G(\mu).
    \end{cases}
\end{equation}
Here $\partial_\mu$ denotes the Lions derivative and $\partial_x\partial_\mu$ denotes its Fréchet
derivative in the spatial variable.
\subsection{Particle--Galerkin approximation and Cole--Hopf feedback}
\label{sec:scmo}
The formal measure-level \Eqref{eq:Master_eq_inf} is computationally intractable:
the state variable is a probability measure on an infinite-dimensional space. To tackle this issue, SCMO uses two
approximations. First, the measure is represented by an empirical particle law. Second, the Hilbert
space is projected to a finite-dimensional Galerkin subspace.

For $\mathbf{x}=(x_1,\ldots,x_N)\in\cH^N$, define the empirical measure $\mu^N_{\mathbf{x}} := \frac1N\sum_{i=1}^N \delta_{x_i}$. For $K\geq1$, set
$\cH_K := \operatorname{span}\{e_1,\ldots,e_K\}$, $Q_K := \Pi_K Q \Pi_K$, where $\Pi_K:\cH\to\cH_K$ is the orthogonal projection. Let $\iota_K:\cH_K\hookrightarrow\cH$ be the canonical embedding. For $\nu\in\cP_2(\cH_K)$, the projected terminal functional is the restriction $G_K(\nu):=G((\iota_K)_\#\nu)$. When no confusion can arise, we write simply $G(\nu)$ for $G_K(\nu)$. The Wasserstein distances used to compare with the original optimizer are computed after embedding laws into $\cH$.

For $\mathbf{x}\in(\cH_K)^N$, define the projected $N$-particle value function
\begin{equation}
    \label{eq:projected_particle_value}
    v_\varepsilon^{K,N}(t,\mathbf{x})
    :=
    \inf_{\boldsymbol{\theta}}
    \E\left[
        N G(\mu^N_{\mathbf{X}_1^K})
        +
        \frac{\varepsilon}{2}
        \sum_{i=1}^N
        \int_t^1 \|\theta_s^i\|_{Q_K}^2\,ds
    \right],
\end{equation}
where $\boldsymbol{\theta}=(\theta^1,\ldots,\theta^N)$ ranges over progressively measurable
$\cH_K$-valued controls with finite energy, and
$\mathbf{X}_s^K=(X_s^{K,1},\ldots,X_s^{K,N})$ is given by $X_s^{K,i}
    =
    x_i+\int_t^s \theta_r^i\,dr + W_s^{K,i}-W_t^{K,i},
    \,
    s\in[t,1],
    \,
    i=1,\ldots,N$.
Here $W^{K,1},\ldots,W^{K,N}$ are independent $Q_K$-Wiener processes on $\cH_K$.

When a smooth value function exists, $v_\varepsilon^{K,N}$ formally solves the finite-dimensional HJB
equation
\begin{equation}
    \label{eq:particle_hjb_galerkin}
    \begin{cases}
    \partial_t v_\varepsilon^{K,N}
    +
    \frac12
    \sum_{i=1}^N
    \operatorname{Tr}\!\left(Q_K D_{x_i}^2 v_\varepsilon^{K,N}\right)
    -
    \frac{1}{2\varepsilon}
    \sum_{i=1}^N
    \left\|Q_K^{1/2}D_{x_i}v_\varepsilon^{K,N}\right\|^2
    =
    0,\\[1.2em]
    v_\varepsilon^{K,N}(1,\mathbf{x})
    =
    N G(\mu^N_{\mathbf{x}}).
    \end{cases}
\end{equation}
This projected particle problem is the point at which the Cole--Hopf transform becomes applicable.
Note that the Cole--Hopf transform does not directly linearize the infinite-dimensional Master \Eqref{eq:Master_eq_inf}; the obstruction is explained in
Appendix~\ref{app:cole_hopf_obstruction}. Let $u_\varepsilon^{K,N}(t,\mathbf{x})
    :=
    \exp\!\left(-v_\varepsilon^{K,N}(t,\mathbf{x})/\varepsilon\right)$. Then $u_\varepsilon^{K,N}$ solves the linear heat equation associated with the $N$ independent
$Q_K$-Wiener particles, with terminal condition
$\exp(-N G(\mu^N_{\mathbf{x}})/\varepsilon)$. Equivalently, for $t<1$ and $\tau:=1-t$, Feynman--Kac formula gives
\begin{equation}
\label{Feynman Kac}
    u_\varepsilon^{K,N}(t,\mathbf{x})
    =
    \E\left[
        \exp\!\left(
            -\frac{N}{\varepsilon}
            G(\mu^N_{\mathbf{x}+\Delta\mathbf{W}_{\tau}^K})
        \right)
    \right],    
\end{equation}
where
\[
    \Delta\mathbf{W}_{\tau}^K
    =
    (\Delta W_{\tau}^{K,1},\ldots,\Delta W_{\tau}^{K,N}),
    \qquad
    \Delta W_{\tau}^{K,i}\sim\mathcal N(0,\tau Q_K)
\]
are independent. The optimal feedback drift for particle $i$ is therefore
\begin{equation}
    \label{eq:optimal_control_particle_clean}
    \theta_i^{*,K,N}(t,\mathbf{x})
    =
    -\frac1\varepsilon Q_K D_{x_i}v_\varepsilon^{K,N}(t,\mathbf{x})
    =
    \frac1{\tau}
    \frac{
        \E\left[
            \exp\!\left(
                -\frac{N}{\varepsilon}
                G(\mu^N_{\mathbf{x}+\Delta\mathbf{W}_{\tau}^K})
            \right)
            \Delta W_{\tau}^{K,i}
        \right]
    }{
        \E\left[
            \exp\!\left(
                -\frac{N}{\varepsilon}
                G(\mu^N_{\mathbf{x}+\Delta\mathbf{W}_{\tau}^K})
            \right)
        \right]
    }.
\end{equation}
Note that this expression is gradient-free: it only requires evaluations of $G$ on empirical measures. 
\subsection{Finite-candidate SCMO update}

Motivated by the projected-particle feedback
formula~\Eqref{eq:optimal_control_particle_clean}, we propose the following algorithm. Full pseudocode is given in
Appendix~\ref{app:full-algorithm}. SCMO evolves a Galerkin
coefficient cloud
\(\mathbf C=(C_1,\ldots,C_N)^\top\in\mathbb R^{N\times K}\), representing
\(u^{(i)}=\sum_{k=1}^K C_{i,k}e_k\), $i = 1,\ldots,N$. For
user-chosen scale parameters
$\sigma_{\mathrm{prop}}$, $\sigma_{\mathrm{dyn}}>0$, let
\(\Lambda=\operatorname{diag}(\lambda_1,\ldots,\lambda_K)\) and define
\(Q_{\mathrm{prop}}=\sigma_{\mathrm{prop}}^2\Lambda\) and
\(Q_{\mathrm{dyn}}=\sigma_{\mathrm{dyn}}^2\Lambda
=\alpha_0Q_{\mathrm{prop}}\), where
\(\alpha_0=\sigma_{\mathrm{dyn}}^2/\sigma_{\mathrm{prop}}^2\).
The proposal covariance $Q_{\mathrm{prop}}$ is used to estimate the feedback, whereas the
execution covariance $Q_{\mathrm{dyn}}$ is used in the Euler update. The cases
\(\alpha_0<1\) and \(\alpha_0>1\) correspond respectively to colder and
hotter execution. A nonmatched choice
is motivated empirically (see Appendix~\ref{app:unmatched-curvature-diagnostic}) and by the signed curvature contribution derived in
Appendix~\ref{app:implementation-consistency}, which can tighten the matched
approximation bound.

Initialize \(\mathbf C\leftarrow\mathbf C_{\mathrm{init}}\) and
\(\Delta t=T/M\). For each of the \(L\) outer loops and
\(m=0,\ldots,M-1\), set \(t_m=m\Delta t\) and
\(\tau_m=T-t_m\). Draw \(R\) independent complete context clouds and
\(S\) candidate replacements per context:
\begin{equation}
\label{eq:main-context-candidates}
\widetilde C_p^{(r)}
=
C_p+\sqrt{\tau_m}\,
\sigma_{\mathrm{prop}}Z_p^{(r)}\Lambda^{1/2},
\qquad
Y_{i,r,s}
=
C_i+\sqrt{\tau_m}\,
\sigma_{\mathrm{prop}}\xi_{i,r,s}\Lambda^{1/2},
\end{equation}
where $r = 1,\ldots,R$, $s = 1,\ldots, S$ and \(Z_p^{(r)},\xi_{i,r,s}\sim\mathcal N(0,I_K)\) are independent.
Here
\(\widetilde{\mathbf C}^{(r)}
=(\widetilde C_1^{(r)},\ldots,\widetilde C_N^{(r)})\)
is one Gaussian perturbation of the full current cloud. When estimating the
feedback of particle \(i\), the \(r\)-th context supplies
\(\{\widetilde C_p^{(r)}:p\neq i\}\), while \(Y_{i,r,s}\) replaces its
\(i\)-th component.

For each particle \(i\), define the replacement scores, globally normalized
Gibbs weights, and feedback estimate by
\begin{equation}
\label{eq:main-finite-candidate-update}
\begin{aligned}
\mathcal E_{i,r,s}
&=
N\,G_K\!\left(
\frac1N\delta_{Y_{i,r,s}}
+
\frac1N\sum_{p\neq i}\delta_{\widetilde C_p^{(r)}}
\right),
\qquad
\mathcal E_i^{\min}
=
\min_{r,s}\mathcal E_{i,r,s},
\\
w_{i,r,s}
&=
\frac{
\exp\!\left(
-(\mathcal E_{i,r,s}-\mathcal E_i^{\min})/\varepsilon
\right)
}{
\displaystyle
\sum_{q=1}^R\sum_{j=1}^S
\exp\!\left(
-(\mathcal E_{i,q,j}-\mathcal E_i^{\min})/\varepsilon
\right)
},
\qquad
\widehat\theta_i
=
\frac1{\tau_m}
\sum_{r=1}^R\sum_{s=1}^S
w_{i,r,s}(Y_{i,r,s}-C_i),
\end{aligned}
\end{equation}
where we are approximating the ratio $\frac{
        \E\left[\E\left[
            \exp\!\left(
                -\frac{N}{\varepsilon}
                G(\mu^N_{\mathbf{x}+\Delta\mathbf{W}_{\tau}^K})
            \right)
            \Delta W_{\tau}^{K,i}\mid\widetilde{\mathbf C}_{-i}^{(r)}
        \right]\right]
    }{
        \mathbb{E}\left[\E\left[
            \exp\!\left(
                -\frac{N}{\varepsilon}
                G(\mu^N_{\mathbf{x}+\Delta\mathbf{W}_{\tau}^K})
            \right)\mid\widetilde{\mathbf C}_{-i}^{(r)}
        \right]\right]} = \frac{
        \E\left[
            \exp\!\left(
                -\frac{N}{\varepsilon}
                G(\mu^N_{\mathbf{x}+\Delta\mathbf{W}_{\tau}^K})
            \right)
            \Delta W_{\tau}^{K,i}
        \right]
    }{
        \E\left[
            \exp\!\left(
                -\frac{N}{\varepsilon}
                G(\mu^N_{\mathbf{x}+\Delta\mathbf{W}_{\tau}^K})
            \right)
        \right]
    }$. The estimator \(\widehat\theta_i\) is the Gibbs-weighted terminal
displacement divided by the remaining time, while subtracting
\(\mathcal E_i^{\min}\) is only a numerical stabilization.

All particles are updated simultaneously. Define
\(\overline C_i=C_i+\Delta t\,\widehat\theta_i\) and
\(\overline{\mathbf C}
=(\overline C_1,\ldots,\overline C_N)^\top\), draw a fresh
\(\boldsymbol\eta_m\sim\mathcal N(0,I_{N\times K})\), independently of all
context and candidate variables, and set
\begin{equation}
\label{eq:main-execution-update}
\mathbf C
\leftarrow
\overline{\mathbf C}
+
\sqrt{\Delta t}\,
\sigma_{\mathrm{dyn}}
\boldsymbol\eta_m\Lambda^{1/2}.
\end{equation}
After all outer loops, SCMO returns
\(\mu^N=N^{-1}\sum_{i=1}^N\delta_{u^{(i)}}\). Finally, note that the frozen-background particle-wise approximation of
Qiu~\cite[Section~3.3]{qiu2026stochastic} corresponds to \(R=1\) in our
notation. This is exact for additive objectives as \(S\to\infty\), because
the background contribution cancels from the Gibbs ratio. For general
interacting objectives, \(R=1\) gives only a conditional feedback, whereas
the unconditional full-cloud feedback is recovered as \(R\to\infty\); see
Appendix~\ref{app:context-cloud-quadratic}.

% ============================================================
\section{Convergence Analysis}
\label{sec:convergence}
\raggedbottom

We first analyze the matched-covariance problem
\(Q_{\mathrm{prop}}=Q_{\mathrm{dyn}}\), because in this case the
Cole--Hopf representation yields a clean approximation theory. We give
two guarantees for the idealized projected particle problem: a
PDE-free qualitative result under a weak recovery condition and a
quantitative projection-first bound separating particle, projection,
and regularization errors. We then return to the practical
multi-context, finite-candidate implementation with nonmatched covariance. For one outer-loop rollout, we quantify the
context--candidate sampling error, the Euler discretization residual,
and the signed effect of replacing the proposal covariance by the
execution covariance.

\subsection{Qualitative convergence}

The first result uses only the finite-dimensional risk-sensitive representation. It does not require a
Master equation, Lions differentiability, or convexity.

\begin{ass}[Robust recovery]
\label{ass:robust-recovery_main}
Assume that \(G:\cP_2(\cH)\to\bR\) is Borel measurable and bounded from below, and that $-\infty<\mathcal V_0:=\inf_{\mu\in\cP_2(\cH)}G(\mu)<\infty$. For every \(\eta>0\) and every \(K_0,N_0\ge1\), assume that there exist
\(K\ge K_0\), \(N\ge N_0\), a cloud \(\mathbf z\in(\cH_K)^N\), and \(r>0\) such that $\sup_{\mathbf y\in B_r(\mathbf z)}G(\mu_{\mathbf y}^N)
    \le
    \mathcal V_0+\eta$.
\end{ass}

\begin{thm}[PDE-free convergence under robust recovery]
\label{thm:qualitative-robust-recovery}
Assume that \(Q\) is trace class with strictly positive eigenvalues and that
Assumption~\ref{ass:robust-recovery_main} holds. Let \(v_\varepsilon^{K,N}\) be the projected particle
value with terminal cost \(NG_K(\mu_{\mathbf x}^N)\). Then there exist sequences
\(K_m,N_m\to\infty\) and \(\varepsilon_m\downarrow0\) such that
\[
    \frac1{N_m}v_{\varepsilon_m}^{K_m,N_m}(0,\Pi_{K_m}\mathbf x_0)
    \longrightarrow
    \mathcal V_0 .
\]
\end{thm}

Upper semicontinuous objectives satisfy the recovery condition by density of empirical measures in
Galerkin subspaces; this includes indicator-type objectives with closed bad-event sets and open
good-event sets. The precise corollaries and proofs are given in
Appendix~\ref{app:proof-qualitative}.

\subsection{Projection-first quantitative bound}

Recall from the Feynman--Kac representation \Eqref{Feynman Kac} that $v_\varepsilon^{K,N}(t,\mathbf x)
    =
    -\varepsilon\log\E\left[
    \exp\left(
    -\frac{N}{\varepsilon}
    G_K(\mu^N_{\mathbf x+\Delta\mathbf W_{1-t}^K})
    \right)
    \right]$. This is the idealized projected particle value studied below. For \(\nu\in\cP_2(\cH_K)\), define the projected mean-field regularized value
\[
    V_\varepsilon^K(t,\nu)
    :=
    \inf_{\theta\in\Theta_K}
    \left\{
    G_K(\Law_{\cH_K}(X_1^{K,t,\xi;\theta}))
    +
    \E\left[
    \frac{\varepsilon}{2}
    \int_t^1\|\theta_s\|_{Q_K}^2\,ds
    \right]
    \right\},
\]
where \(\Law_{\cH_K}(\xi)=\nu\) and $X_s^{K,t,\xi;\theta}
    =
    \xi+\int_t^s\theta_r\,dr+W_s^K-W_t^K$.

The projected Master equation, verification statement, and regularity assumptions for
\(V_\varepsilon^K\) are given in Appendix~\ref{app:quantitative-assumptions}.

The proof first compares
\(N^{-1}v_\varepsilon^{K,N}\) with the projected mean-field value \(V_\varepsilon^K\), and then
compares \(V_\varepsilon^K(0,\delta_{\Pi_Kx_0})\) with the original optimum \(\mathcal V_0\), thus obtaining the difference between \(N^{-1}v_\varepsilon^{K,N}\) and $\mathcal{V}_0$.

Assume that the infimum $\mathcal V_0$ is attained by some $\mu^*\in\cP_2(\cH)$. For \(X\sim\mu^\star\), define $\Delta_K^{\mathrm{proj}}
    :=
    \E_{\mu^\star}\|(I-\Pi_K)X\|^2$, $D_{CM,K}^2(\mu^\star)
    :=
    \sum_{j=1}^K
    \frac{\E_{\mu^\star}[(X_j-x_{0,j})^2]}{\lambda_j}$. We assume the projected regularized value \(V_\varepsilon^K\) is a verified classical solution of
the projected Master equation with the derivative and residual bounds stated in
Appendix~\ref{app:quantitative-assumptions}. In particular, the residual constant
\(L_{\varepsilon,K}\) satisfies
\[
    \left|
    \int_{\cH_K}
    \Tr\!\left(
    Q_K\partial_{\mu\mu}^2V_\varepsilon^K(t,\nu)(y)(y)
    \right)\nu(dy)
    \right|
    \le L_{\varepsilon,K}.
\]
We also assume that \(G\) has a local H\"older continuity
condition of exponent \(\alpha\in(0,1]\) around a minimizer \(\mu^\star\), and the projected particle value \(v_\varepsilon^{K,N}\) is sufficiently regular
for the comparison argument: its spatial gradients are locally Lipschitz on compact time intervals
in \([0,1)\) and have at most linear growth. The precise statement is given in
Appendix~\ref{app:quantitative-assumptions}.

\begin{thm}[Projection-first convergence of the idealized particle approximation]
\label{thm:projection-first-global}
Assume $0<\varepsilon< 1$. Under the assumptions above, let \(v_\varepsilon^{K,N}\) be a classical solution of the projected
particle HJB. Then
\begin{align}
\label{eq:projection-first-final-bound}
0\!
\le\!
\frac1N v_\varepsilon^{K,N}(0,\Pi_K\mathbf x_0)-\mathcal V_0
&\le \!
\frac{L_{\varepsilon,K}}{2N}
+\!
C_*
\left(
\varepsilon^{2/\alpha}\Tr(Q)
+
\Delta_K^{\mathrm{proj}}
\right)^{\alpha/2}
\!\!\!\!
+
\frac{\varepsilon}{2}D_{CM,K}^2(\mu^\star)
+\!
\frac{\varepsilon K}{\alpha}\log\frac1\varepsilon .
\end{align}
Here \(\Pi_K\mathbf x_0=(\Pi_Kx_0,\ldots,\Pi_Kx_0)\), and \(C_*\) is independent of
\(K,N,\varepsilon\).
\end{thm}

The proof decomposes the error into a particle approximation term
\(L_{\varepsilon,K}/(2N)\) and a projected regularization error; details are given in
Appendix~\ref{app:proof-quantitative}. The bound tends to zero along
any joint choice \(K=K_m\), \(N=N_m\), and \(\varepsilon=\varepsilon_m\downarrow0\) such that
\[
    \frac{L_{\varepsilon_m,K_m}}{N_m}\to0,\qquad
    \Delta_{K_m}^{\mathrm{proj}}\to0,\qquad
    \varepsilon_mD_{CM,K_m}^2(\mu^\star)\to0,\qquad
    \varepsilon_mK_m\log\frac1{\varepsilon_m}\to0 .
\]
Appendix~\ref{app:uniform-L} gives one sufficient FBSDE-based condition under which
\(L_{\varepsilon,K}\) is bounded uniformly in \((\varepsilon,K)\).

\subsection{Finite-candidate implementation with nonmatched covariance}
\label{sec:finite-unmatched-convergence}

We now connect the practical SCMO update with nonmatched execution
covariance to the matched projected particle value. A single outer-loop rollout consists of \(M\) Euler updates. Since each
outer loop applies the same \(M\)-step transition to its current cloud, we
analyze one representative rollout (\(L=1\)). The results below are stated
on the normalized horizon \([0,1]\); a general horizon follows by time
rescaling. Let $Q_{\mathrm{prop}}
=
\sigma_{\mathrm{prop}}^2Q$,
$Q_{\mathrm{dyn}}
=
\sigma_{\mathrm{dyn}}^2Q$,
$\alpha_0
=
\frac{\sigma_{\mathrm{dyn}}^2}
{\sigma_{\mathrm{prop}}^2}$,
so that $Q_{\mathrm{dyn}}
=
\alpha_0Q_{\mathrm{prop}}$. Their projected \(N\)-particle extensions are
$
\mathbf Q_{K,\mathrm{prop}}
=
I_N\otimes\Pi_KQ_{\mathrm{prop}}\Pi_K$, $\mathbf Q_{K,\mathrm{dyn}}
=
\alpha_0\mathbf Q_{K,\mathrm{prop}}$. Throughout this subsection,
\(v_\varepsilon^{K,N}\) and \(b\) denote the matched projected
particle value and feedback associated with
\(\mathbf Q_{K,\mathrm{prop}}\).

For fixed \(t<1\) and \(\mathbf x\in(H_K)^N\), let
\(\widehat b_i^{R,S}(t,\mathbf x)\) be the globally normalized
\(R\)-context, \(S\)-candidate estimator used by SCMO (see \Eqref{eq:main-finite-candidate-update}). Under the
integrability conditions stated in
Appendix~\ref{app:implementation-consistency}, for every fixed
\(S\geq1\),
\[
\widehat b_i^{R,S}(t,\mathbf x)
\longrightarrow
b_i(t,\mathbf x)
\qquad
\text{almost surely as }R\to\infty.
\]
If the corresponding fourth moments and inverse-denominator moments
are uniformly bounded, then
\begin{equation}
\label{eq:main-multicontext-rate}
\mathbb E
\left[
\left\|
\widehat b_i^{R,S}(t,\mathbf x)
-
b_i(t,\mathbf x)
\right\|^2
\right]
\leq
C
\left(
\frac1R+\frac1{RS}
\right).
\end{equation}
A controlled interacting-quadratic diagnostic in
Appendix~A.4 shows that the empirical feedback mean-square error
decays at the predicted rate; see
Figure~\ref{fig:finite-candidate-rate}.

We next define the exact-feedback nonmatched-covariance chain to which the fully finite
implementation converges. Let $h=\frac1M$, $t_m=mh$,
and set
\begin{equation}
\label{eq:main-exact-fixed-chain}
\mathbf X_{m+1}^h
=
\mathbf X_m^h
+
h\,b(t_m,\mathbf X_m^h)
+
\sqrt h\,
\mathbf Q_{K,\mathrm{dyn}}^{1/2}Z_{m+1},
\qquad
\mathbf X_0^h=\Pi_K\mathbf x_0,
\end{equation}
where \(Z_{m+1}\sim\mathcal N(0,I_{NK})\) are independent.

For a deterministic state \(\mathbf x\), write $\overline{\mathbf x}_h
=
\mathbf x+h\,b(t,\mathbf x)$, and define the matched and execution continuation scores
\[
\Phi_h^{v,\bullet}(t,\mathbf x)
=
\mathbb E\left[
v_\varepsilon^{K,N}
\left(
t+h,
\overline{\mathbf x}_h
+
\sqrt h\,
\mathbf Q_{K,\bullet}^{1/2}Z
\right)
\right],
\qquad
\bullet\in\{\mathrm{prop},\mathrm{dyn}\}.
\]
The signed one-step covariance contribution is $
c_h^{K,N,\varepsilon}(t,\mathbf x)
=
\Phi_h^{v,\mathrm{dyn}}(t,\mathbf x)
-
\Phi_h^{v,\mathrm{prop}}(t,\mathbf x)$. Its favorable and unfavorable parts are $g_{h}^{K,N,\varepsilon}
=
\left(-c_h^{K,N,\varepsilon}\right)_+$, $p_{h}^{K,N,\varepsilon}
=
\left(c_h^{K,N,\varepsilon}\right)_+$. Let
$\ell_\varepsilon^{K,N}(t,\mathbf x)
=
\frac{\varepsilon}{2}
\|b(t,\mathbf x)\|_{\mathbf Q_{K,\mathrm{prop}}^{-1}}^2$, and define the matched weak Euler residual
\[
r_h^{K,N,\varepsilon}(t,\mathbf x)
=
\Phi_h^{v,\mathrm{prop}}(t,\mathbf x)
+
h\ell_\varepsilon^{K,N}(t,\mathbf x)
-
v_\varepsilon^{K,N}(t,\mathbf x).
\]
Along the chain~\Eqref{eq:main-exact-fixed-chain}, set 
\begin{align}
\label{Euler_residual}
    &\mathcal E_h^{K,N,\varepsilon} =
\frac1N
\sum_{m=0}^{M-1}
\mathbb E
\left[
\left(
r_h^{K,N,\varepsilon}
(t_m,\mathbf X_m^h)
\right)_+
\right],\quad\mathcal G_{h}^{K,N,\varepsilon}
=
\frac1N
\sum_{m=0}^{M-1}
\mathbb E
\left[
g_{h}^{K,N,\varepsilon}
(t_m,\mathbf X_m^h)
\right], \\
&\mathcal P_{h}^{K,N,\varepsilon}
=
\frac1N
\sum_{m=0}^{M-1}
\mathbb E
\left[
p_{h}^{K,N,\varepsilon}
(t_m,\mathbf X_m^h)
\right].
\end{align}
Finally, let
\(\mathcal R_{K,N,\varepsilon}^{\mathrm{prop}}\) denote the
right-hand side of
\Eqref{eq:projection-first-final-bound}, with \(Q\) replaced by
\(Q_{\mathrm{prop}}\) and the Cameron--Martin term computed using the
eigenvalues of \(Q_{\mathrm{prop}}\).

\begin{thm}[One-rollout implementation bound]
\label{thm:main-finite-unmatched}
Assume the hypotheses of
Theorem~\ref{thm:projection-first-global}, with
\(Q_{\mathrm{prop}}\) as the covariance of the matched control
problem, together with the stability and integrability conditions in
Appendix~\ref{app:implementation-consistency}. Then the exact
feedback nonmatched-covariance chain satisfies
\begin{equation}
\label{eq:main-fixed-covariance-bound}
\mathbb E
G_K(\mu_{\mathbf X_M^h}^N)
-
\mathcal V_0
\leq
\mathcal R_{K,N,\varepsilon}^{\mathrm{prop}}
+
\mathcal E_h^{K,N,\varepsilon}
-
\mathcal G_{h}^{K,N,\varepsilon}
+
\mathcal P_{h}^{K,N,\varepsilon}.
\end{equation}

Let \(\mathbf C_m^{R,S,h}\) be the fully finite SCMO chain using the
same execution noises as \(\mathbf X_m^h\), but replacing \(b\) by
\(\widehat b^{R,S}\). Under the quantitative stability assumptions,
there is a constant \(C_{\mathrm{stab}}<\infty\), independent of
\(R,S\), such that
\begin{equation}
\label{eq:main-finite-chain-rate}
\max_{0\leq m\leq M}
\mathbb E
\left[
\left\|
\mathbf C_m^{R,S,h}
-
\mathbf X_m^h
\right\|^2
\right]
\leq
C_{\mathrm{stab}}
\left(
\frac1R+\frac1{RS}
\right).
\end{equation}
Consequently,
\begin{align}
\label{eq:main-practical-limsup}
\limsup_{\substack{R\to\infty\\RS\to\infty}}
\left[
\mathbb E
G_K(\mu_{\mathbf C_M^{R,S,h}}^N)
-
\mathcal V_0
\right]
\leq{}&
\mathcal R_{K,N,\varepsilon}^{\mathrm{prop}}
+
\mathcal E_h^{K,N,\varepsilon}-
\mathcal G_{h}^{K,N,\varepsilon}
+
\mathcal P_{h}^{K,N,\varepsilon}.
\end{align}
\end{thm}

The term \(\mathcal E_h^{K,N,\varepsilon}\) is the time-discretization
contribution. Appendix~\ref{app:implementation-consistency} gives the
identity
\[
r_h^{K,N,\varepsilon}(t,\mathbf x)
=
\frac{\varepsilon}{2}
\mathbb E
\int_0^h
\left\|
b(t,\mathbf x)
-
b(t+s,\mathbf Y_s^{t,\mathbf x})
\right\|_{\mathbf Q_{K,\mathrm{prop}}^{-1}}^2\,ds,
\]
so it is nonnegative and vanishes as \(h\downarrow0\) under suitable
H\"older--Lipschitz regularity of the feedback.

The covariance contribution is signed rather than being treated
solely as an error. In the smooth case,
\[
c_h^{K,N,\varepsilon}(t,\mathbf x)
=
\frac h2
\operatorname{Tr}\left(
\left[
\mathbf Q_{K,\mathrm{dyn}}
-
\mathbf Q_{K,\mathrm{prop}}
\right]
D^2v_\varepsilon^{K,N}
(t+h,\overline{\mathbf x}_h)
\right)
+
o(h).
\]
For proportional covariances, define $
\kappa_v(t,\mathbf x)
=
\operatorname{Tr}\left(
\mathbf Q_{K,\mathrm{prop}}
D^2v_\varepsilon^{K,N}(t,\mathbf x)
\right)$. Then the leading contribution is
$
\frac h2(\alpha_0-1)\kappa_v$. Thus a colder execution covariance, \(\alpha_0<1\), is favorable in
regions of positive covariance-weighted continuation-value curvature,
whereas a hotter covariance, \(\alpha_0>1\), is favorable in regions
of negative weighted curvature. Accordingly, the signed covariance contribution in
\Eqref{eq:main-fixed-covariance-bound} is favorable whenever $\mathcal G_{h}^{K,N,\varepsilon}
>
\mathcal P_{h}^{K,N,\varepsilon}$. If, more strongly,
$\mathcal G_{h}^{K,N,\varepsilon}
>
\mathcal E_h^{K,N,\varepsilon}
+
\mathcal P_{h}^{K,N,\varepsilon}$, then the complete one-rollout upper bound is strictly smaller than the
matched particle--projection--regularization certificate
\(\mathcal R_{K,N,\varepsilon}^{\mathrm{prop}}\). Full proofs,
continuous-time analogues, and checkable bounded and polynomial-growth
conditions are given in
Appendix~\ref{app:implementation-consistency}.
\section{Experiments}
\label{sec:experiments}
We evaluate SCMO on two structured benchmark suites and one external simulator
environment. Across these experiments, we test multimodal-law recovery, escape
from suboptimal basins, and optimization under nonsmooth or discontinuous
law-level objectives. The main study comprises function-space optimization
with PDE energies, open-loop trajectory-law optimization, and a fixed-reset
Push-T simulator experiment. Complementary low-dimensional checks isolate
these capabilities in controlled settings, including progress from a flat
nonsmooth plateau. Complete results and figures for these checks are provided
in Appendix~\ref{app:experiments}; additional studies examine the effect of
\(R\), proposal--execution covariance nonmatch, and latent image generation in
Appendices~\ref{app:context-cloud-quadratic},
\ref{app:unmatched-curvature-diagnostic}, and
\ref{app:latent-matching}, respectively. Code for reproducing our experimental results is available at
\url{https://github.com/HenryCHEUNG7373/SCMO}.

\paragraph{Protocol and baselines.}
Within each task, empirical-law methods share the same initial cloud, and all
methods use the corresponding common initialization and paired evaluation
seeds. Method-specific hyperparameters are selected on separate development
runs and frozen before evaluation. Derivative-free methods are matched by
maximum primitive objective work: particle-energy evaluations in the PDE
suite, trajectory simulations in the passage suite, and simulator calls in Push-T. More details are
reported in Appendix~\ref{app:experiments}.

We use benchmark-specific subsets of complementary controls: Point Adam and
Particle Adam~\citep{kingma2015adam}; a diagonal-Gaussian family optimized
using reparameterized Monte Carlo gradients~\citep{kingma2014autoencoding};
whole-cloud CEM~\citep{deboer2005crossentropy} and
CMA-ES~\citep{hansen2001cma}; greedy whole-cloud random
search~\citep{solis1981randomsearch}; and an MPPI-style open-loop
optimizer~\citep{williams2017mppi}. Point Adam, Particle Adam, CEM, and CMA-ES
use the standard update rules of their cited optimizers on the stated decision
representations. The diagonal-Gaussian, whole-cloud random-search, and
MPPI-style empirical-law controls are defined in
Appendix~\ref{app:experiments}.

\paragraph{Function-space PDE laws.}
Each empirical law comprises \(N=256\) functions in \(L^2([0,1])\),
represented using \(K=32\) Neumann cosine coefficients. The five tasks test
smooth balanced two-phase Allen--Cahn recovery (P1); a nonsmooth, nonconvex
TV/double-well objective with an absolute-moment penalty (P2); discontinuous
interacting three-phase quotas (P3); escape from the shallow phase of a tilted
double well to its deeper global phase (P4); and escape from a metastable
central phase while balancing mass between two outer global phases (P5).

As shown in Table~\ref{tab:main-results}, SCMO attains the lowest mean
objective on P1--P3 and is the only method to satisfy the P3 quotas in all
\(20\) runs. SCMO and the diagonal-Gaussian baseline both escape the P4 local
phase in every run, with the Gaussian baseline attaining the lower objective.
On P5, only SCMO consistently recovers the balanced outer-phase law. Although
Particle Adam attains a lower P5 objective, it fails the balanced-law criterion
in every run, showing that objective value alone does not certify recovery of
the intended law structure.

\paragraph{Trajectory laws.}
Each empirical law comprises \(N=128\) open-loop, 42-step trajectories. The
five tasks test broad single-passage feasibility (T1); narrow local feasibility
(T2); continuous two-route mass recovery (T3); discontinuous, interacting
route quotas (T4); and escape from a feasible but inferior lower-route basin
to a better upper route under the ordinary trajectory objective (T5).

As shown in Table~\ref{tab:main-results}, all methods solve T1, whereas SCMO
remains effective on the narrow T2 task. SCMO is the only method to satisfy
the T3 full-success and two-route criterion in all \(20\) runs. On T4, SCMO
succeeds in \(20/20\) runs, compared with \(14/20\) for Mixture MPPI and none
for the remaining baselines. For T5, SCMO and
Mixture MPPI discover the better upper route in every run, whereas random
search and CEM remain in the lower-route basin; SCMO attains the lowest mean
objective.

\paragraph{Fixed-reset Push-T simulator check.}
We optimize empirical laws of \(N=8\) open-loop action sequences under the
standard contact-rich Push-T dynamics~\citep{chi2023diffusionpolicy}. The
physical task is to control the pushing agent so that its contacts move and
rotate the T-shaped block into a prescribed target pose. This is a fixed-reset
open-loop optimization study, not policy learning or the canonical
state-distribution generalization protocol. We evaluate ten held-out
optimization seeds under five deterministic scale perturbations
\(\gamma\in\{.85,.90,1,1.05,1.10\}\) of one frozen reset, using exactly
\(4{,}408\) simulator calls per method.

The objective combines mean rollout loss with route-balance and
minimum-route-mass penalties. Here, \(p_L\) and \(p_R\) are the fractions of
all eight particles that succeed through the left and right routes; failed or
unclassified rollouts contribute to neither. The quota \(p_L,p_R\ge .25\)
therefore requires at least two successful trajectories on each route. As
shown in Table~\ref{tab:main-results}, SCMO attains the lowest seed-averaged
objective, highest success mass, and closest route balance. SCMO and CEM
satisfy the quota in all \(50\) condition--seed cells and under all five
conditions for every seed, but SCMO achieves the lower objective and higher
success mass; the remaining baselines satisfy the quota less consistently.

\begin{table}[t]
\centering
\scriptsize
\setlength{\tabcolsep}{4.1pt}
\caption{Main results. Panels A--B report mean $\pm$ standard deviation and
structural successes over 20 paired seeds; Panel C reports ten seed-level
averages across five Push-T conditions, with pooled counts shown descriptively.}
\label{tab:main-results}

\begin{tabular}{@{}lrrrr@{}}
\toprule
\multicolumn{5}{l}{\textbf{A. Function-space PDE laws:
objective [successes/20]}}\\
Task & SCMO & Particle Adam & CEM cloud & Gaussian \\
\midrule
P1 Smooth Allen--Cahn
& \(\bm{0.128\pm.003}\) & \(\underline{0.184\pm.017}\)
& \(1.734\pm.042\) & \(0.996\pm.211\) \\
P2 TV--Allen--Cahn
& \(\bm{0.984\pm.028}\) & \(1.747\pm.031\)
& \(\underline{1.069\pm.020}\) & \(1.118\pm.279\) \\
P3 Three-phase quotas
& \(\bm{-2.541\pm.635}\,[20]\) & \(2.174\pm.034\,[0]\)
& \(2.835\pm.083\,[0]\) & \(\underline{2.009\pm.001}\,[0]\) \\
P4 Tilted local/global
& \(\underline{0.00883\pm.00008}\,[20]\) & \(3.846\pm.157\,[0]\)
& \(5.654\pm.007\,[0]\) & \(\bm{0.00109\pm.00010}\,[20]\) \\
P5 Metastable balanced law
& \(\underline{0.342\pm.004}\,[20]\) & \(\bm{0.264\pm.041}\,[0]\)
& \(5.263\pm.131\,[0]\) & \(5.271\pm.098\,[0]\) \\
\midrule
\multicolumn{5}{l}{\textbf{B. Trajectory laws:
objective [successes/20]}}\\
Task & SCMO & Random search & CEM cloud & Mixture MPPI \\
\midrule
T1 Easy passage
& \(0.886\pm.002\,[20]\) & \(\bm{0.846\pm.000}\,[20]\)
& \(0.886\pm.002\,[20]\) & \(\underline{0.848\pm.000}\,[20]\) \\
T2 Narrow passage
& \(\bm{1.097\pm.028}\,[20]\) & \(7.900\pm.002\,[0]\)
& \(7.853\pm.000\,[0]\) & \(\underline{3.188\pm.416}\,[20]\) \\
T3 Two-route multimodal
& \(\bm{0.080\pm.006}\,[20]\) & \(8.005\pm.252\,[0]\)
& \(3.930\pm.322\,[0]\) & \(\underline{1.360\pm.392}\,[0]\) \\
T4 Discontinuous quotas
& \(\bm{-0.482\pm.008}\,[20]\) & \(8.448\pm.128\,[0]\)
& \(6.628\pm.149\,[0]\) & \(\underline{1.073\pm.777}\,[14]\) \\
T5 Local/global routes
& \(\bm{2.715\pm.000}\,[20]\)
& \(2.857\pm.001\,[0]\)
& \(\underline{2.829\pm.001}\,[0]\)
& \(2.840\pm.004\,[20]\) \\
\bottomrule
\end{tabular}

\vspace{1.5mm}

\begin{tabular}{@{}lrrrrr@{}}
\toprule
\multicolumn{6}{l}{\textbf{C. Push-T external simulator anchor}}\\
Method & Objective \(\downarrow\) & \(p_s\uparrow\)
& Quota runs & All-8 runs & Quota seeds (all 5)\\
\midrule
SCMO
& \(\bm{0.021\pm.004}\) & \(\bm{0.965\pm.027}\)
& 50/50 & 36/50 & 10/10 \\
CEM cloud
& \(\underline{0.062\pm.014}\) & \(\underline{0.920\pm.059}\)
& 50/50 & 27/50 & 10/10 \\
Random-search cloud
& \(0.172\pm.043\) & \(0.585\pm.058\)
& 45/50 & 0/50 & 6/10 \\
MPPI-style cloud
& \(0.227\pm.026\) & \(0.520\pm.047\)
& 28/50 & 0/50 & 0/10 \\
Full CMA-ES cloud
& \(0.318\pm.038\) & \(0.420\pm.037\)
& 6/50 & 0/50 & 0/10 \\
\bottomrule
\end{tabular}
\end{table}

\section{Conclusion}
We introduced SCMO, an objective-only optimizer for empirical probability
laws on separable Hilbert spaces. Combining particle--Galerkin approximation
with a Cole--Hopf representation, SCMO estimates feedback through
context-conditioned one-particle replacement scores. We established
matched-covariance convergence, multi-context finite-candidate consistency,
and a signed, curvature-dependent analysis of proposal--execution
covariance nonmatch. Experiments on function-space PDE energies, open-loop
trajectory laws, and fixed-reset Push-T manipulation show that SCMO handles
nonsmooth, nonconvex, and discontinuous law-level objectives, escapes
suboptimal basins, and recovers prescribed multimodal, non-Dirac law
structure.

\section*{Funding}
This work was supported in part by the Natural Sciences and Engineering
Research Council of Canada (NSERC).

\bibliography{iclr2027_conference}

\begin{thebibliography}{32}
\providecommand{\natexlab}[1]{#1}
\providecommand{\url}[1]{\texttt{#1}}
\expandafter\ifx\csname urlstyle\endcsname\relax
  \providecommand{\doi}[1]{doi: #1}\else
  \providecommand{\doi}{doi: \begingroup \urlstyle{rm}\Url}\fi

\bibitem[Carrillo et~al.(2018)Carrillo, Choi, Totzeck, and Tse]{carrillo2018cbo}
Jos{\'e}~A. Carrillo, Young-Pil Choi, Claudia Totzeck, and Oliver Tse.
\newblock {An Analytical Framework for a Consensus-Based Global Optimization Method}.
\newblock \emph{Mathematical Models and Methods in Applied Sciences}, 28\penalty0 (6):\penalty0 1037--1066, 2018.
\newblock \doi{10.1142/S0218202518500276}.

\bibitem[Chaudhari et~al.(2017)Chaudhari, Choromanska, Soatto, LeCun, Baldassi, Borgs, Chayes, Sagun, and Zecchina]{chaudhari2017entropy}
Pratik Chaudhari, Anna Choromanska, Stefano Soatto, Yann LeCun, Carlo Baldassi, Christian Borgs, Jennifer~T. Chayes, Levent Sagun, and Riccardo Zecchina.
\newblock {{Entropy-SGD}: Biasing Gradient Descent Into Wide Valleys}.
\newblock In \emph{International Conference on Learning Representations}, 2017.
\newblock URL \url{https://openreview.net/forum?id=B1YfAfcgl}.

\bibitem[Chen et~al.(2023)Chen, Ren, and Wang]{chen2023entropic}
Fan Chen, Zhenjie Ren, and Songbo Wang.
\newblock {Entropic fictitious play for mean field optimization problem}.
\newblock \emph{Journal of Machine Learning Research}, 24\penalty0 (211):\penalty0 1--36, 2023.

\bibitem[Chi et~al.(2023)Chi, Feng, Du, Xu, Cousineau, Burchfiel, and Song]{chi2023diffusionpolicy}
Cheng Chi, Siyuan Feng, Yilun Du, Zhenjia Xu, Eric Cousineau, Benjamin C.~M. Burchfiel, and Shuran Song.
\newblock {Diffusion Policy: Visuomotor Policy Learning via Action Diffusion}.
\newblock In \emph{Proceedings of Robotics: Science and Systems}, Daegu, Republic of Korea, July 2023.
\newblock \doi{10.15607/RSS.2023.XIX.026}.

\bibitem[Chizat \& Bach(2018)Chizat and Bach]{chizat2018global}
L{\'e}na{\"i}c Chizat and Francis Bach.
\newblock {On the Global Convergence of Gradient Descent for Over-parameterized Models Using Optimal Transport}.
\newblock In \emph{Advances in Neural Information Processing Systems}, volume~31, 2018.

\bibitem[de~Boer et~al.(2005)de~Boer, Kroese, Mannor, and Rubinstein]{deboer2005crossentropy}
Pieter-Tjerk de~Boer, Dirk~P. Kroese, Shie Mannor, and Reuven~Y. Rubinstein.
\newblock {A Tutorial on the Cross-Entropy Method}.
\newblock \emph{Annals of Operations Research}, 134\penalty0 (1):\penalty0 19--67, 2005.
\newblock \doi{10.1007/s10479-005-5724-z}.

\bibitem[Ding et~al.(2025)Ding, Jaquier, Peters, and Rozo]{ding2025riemannian}
Haoran Ding, No{\'e}mie Jaquier, Jan Peters, and Leonel Rozo.
\newblock {Fast and Robust Visuomotor Riemannian Flow Matching Policy}.
\newblock \emph{IEEE Transactions on Robotics}, 41:\penalty0 5327--5343, 2025.
\newblock \doi{10.1109/TRO.2025.3601293}.

\bibitem[Garg et~al.(2024)Garg, Zhang, and Zhou]{garg2024soft}
Jhanvi Garg, Xianyang Zhang, and Quan Zhou.
\newblock {Soft-constrained Schr{\"o}dinger Bridge: a stochastic control approach}.
\newblock In \emph{International Conference on Artificial Intelligence and Statistics}, pp.\  4429--4437. PMLR, 2024.

\bibitem[Hansen \& Ostermeier(2001)Hansen and Ostermeier]{hansen2001cma}
Nikolaus Hansen and Andreas Ostermeier.
\newblock Completely derandomized self-adaptation in evolution strategies.
\newblock \emph{Evolutionary Computation}, 9\penalty0 (2):\penalty0 159--195, 2001.
\newblock \doi{10.1162/106365601750190398}.

\bibitem[Hu et~al.(2021)Hu, Ren, {\v{S}}i{\v{s}}ka, and Szpruch]{hu2021meanfield}
Kaitong Hu, Zhenjie Ren, David {\v{S}}i{\v{s}}ka, and {\L}ukasz Szpruch.
\newblock {Mean-Field Langevin Dynamics and Energy Landscape of Neural Networks}.
\newblock \emph{Annales de l'Institut Henri Poincar{\'e}, Probabilit{\'e}s et Statistiques}, 57\penalty0 (4):\penalty0 2043--2065, 2021.
\newblock \doi{10.1214/20-AIHP1140}.

\bibitem[Huang \& Kouhkouh(2026)Huang and Kouhkouh]{huang2026hilbert}
Hui Huang and Hicham Kouhkouh.
\newblock {A Derivative-Free Particle Method for Optimization in Hilbert Spaces}.
\newblock \emph{arXiv preprint arXiv:2605.31565}, 2026.

\bibitem[Kerrigan et~al.(2024)Kerrigan, Migliorini, and Smyth]{kerrigan2024functional}
Gavin Kerrigan, Giosue Migliorini, and Padhraic Smyth.
\newblock {Functional Flow Matching}.
\newblock In \emph{Proceedings of the 27th International Conference on Artificial Intelligence and Statistics}, 2024.

\bibitem[Khatab \& Totzeck(2026)Khatab and Totzeck]{khatab2026consensus}
Mahmoud Khatab and Claudia Totzeck.
\newblock A consensus-based optimization algorithm using gaussian processes for global optimization problems in {Sobolev} spaces.
\newblock \emph{arXiv preprint arXiv:2603.15337}, 2026.

\bibitem[Kingma \& Ba(2015)Kingma and Ba]{kingma2015adam}
Diederik~P. Kingma and Jimmy Ba.
\newblock {Adam: A Method for Stochastic Optimization}.
\newblock In \emph{International Conference on Learning Representations}, 2015.
\newblock URL \url{https://arxiv.org/abs/1412.6980}.

\bibitem[Kingma \& Welling(2014)Kingma and Welling]{kingma2014autoencoding}
Diederik~P. Kingma and Max Welling.
\newblock {Auto-Encoding Variational Bayes}.
\newblock In \emph{International Conference on Learning Representations}, 2014.
\newblock URL \url{https://arxiv.org/abs/1312.6114}.

\bibitem[Kook et~al.(2024)Kook, Zhang, Chewi, Erdogdu, and Li]{kook2024meanfield}
Yunbum Kook, Matthew~S. Zhang, Sinho Chewi, Murat~A. Erdogdu, and Mufan Li.
\newblock {Sampling from the Mean-Field Stationary Distribution}.
\newblock In \emph{Proceedings of the 37th Conference on Learning Theory}, volume 247 of \emph{Proceedings of Machine Learning Research}, pp.\  3099--3136. PMLR, 2024.

\bibitem[Li et~al.(2026)Li, Sun, Turk, and Zhu]{li2026functional}
Zhiqi Li, Yuchen Sun, Greg Turk, and Bo~Zhu.
\newblock {Functional Mean Flow in Hilbert Space}.
\newblock In \emph{Proceedings of the IEEE/CVF Conference on Computer Vision and Pattern Recognition}, pp.\  1928--1938, June 2026.

\bibitem[Liu \& Wang(2016)Liu and Wang]{liu2016svgd}
Qiang Liu and Dilin Wang.
\newblock {Stein Variational Gradient Descent: A General Purpose Bayesian Inference Algorithm}.
\newblock In \emph{Advances in Neural Information Processing Systems}, volume~29, 2016.

\bibitem[Luu et~al.(2024)Luu, Yu, Williams, Mikkola, Hartmann, Puolam{\"a}ki, and Klami]{luu2024nongeodesic}
Hoang Phuc~Hau Luu, Hanlin Yu, Bernardo Williams, Petrus Mikkola, Marcelo Hartmann, Kai Puolam{\"a}ki, and Arto Klami.
\newblock {Non-geodesically-convex Optimization in the Wasserstein Space}.
\newblock In \emph{Advances in Neural Information Processing Systems}, volume~37, 2024.

\bibitem[Nitanda \& Suzuki(2017)Nitanda and Suzuki]{nitanda2017stochastic}
Atsushi Nitanda and Taiji Suzuki.
\newblock {Stochastic Particle Gradient Descent for Infinite Ensembles}.
\newblock \emph{arXiv preprint arXiv:1712.05438}, 2017.

\bibitem[Peixoto et~al.(2025)Peixoto, Csillag, da~Costa, and Saporito]{peixoto2026random}
Caio Peixoto, Daniel Csillag, Bernardo F.~P. da~Costa, and Yuri~F. Saporito.
\newblock {Random Gradient-Free Optimization in Infinite Dimensional Spaces}.
\newblock \emph{arXiv preprint arXiv:2512.20566}, 2025.

\bibitem[Pinnau et~al.(2017)Pinnau, Totzeck, Tse, and Martin]{pinnau2017cbo}
Ren{\'e} Pinnau, Claudia Totzeck, Oliver Tse, and Stephan Martin.
\newblock {A Consensus-Based Model for Global Optimization and Its Mean-Field Limit}.
\newblock \emph{Mathematical Models and Methods in Applied Sciences}, 27\penalty0 (1):\penalty0 183--204, 2017.
\newblock \doi{10.1142/S0218202517400061}.

\bibitem[Pittorino et~al.(2021)Pittorino, Lucibello, Feinauer, Perugini, Baldassi, Demyanenko, and Zecchina]{pittorino2021entropic}
Fabrizio Pittorino, Carlo Lucibello, Christoph Feinauer, Gabriele Perugini, Carlo Baldassi, Elizaveta Demyanenko, and Riccardo Zecchina.
\newblock {Entropic Gradient Descent Algorithms and Wide Flat Minima}.
\newblock In \emph{International Conference on Learning Representations}, 2021.
\newblock URL \url{https://openreview.net/forum?id=xjXg0bnoDmS}.

\bibitem[Qiu(2026)]{qiu2026stochastic}
Jinniao Qiu.
\newblock {Stochastic Control Methods for Optimization}.
\newblock \emph{arXiv preprint arXiv:2601.01248}, 2026.

\bibitem[Solis \& Wets(1981)Solis and Wets]{solis1981randomsearch}
Francisco~J. Solis and Roger J.-B. Wets.
\newblock {Minimization by Random Search Techniques}.
\newblock \emph{Mathematics of Operations Research}, 6\penalty0 (1):\penalty0 19--30, 1981.
\newblock \doi{10.1287/moor.6.1.19}.

\bibitem[Tankala et~al.(2025)Tankala, Nagaraj, and Raj]{tankala2025meanfield}
Chandan Tankala, Dheeraj~M. Nagaraj, and Anant Raj.
\newblock {Beyond Propagation of Chaos: A Stochastic Algorithm for Mean Field Optimization}.
\newblock In \emph{Proceedings of the 38th Conference on Learning Theory}, volume 291 of \emph{Proceedings of Machine Learning Research}, pp.\  5410--5440. PMLR, 2025.

\bibitem[Theodorou et~al.(2010)Theodorou, Buchli, and Schaal]{theodorou2010pathintegral}
Evangelos~A. Theodorou, Jonas Buchli, and Stefan Schaal.
\newblock {A Generalized Path Integral Control Approach to Reinforcement Learning}.
\newblock \emph{Journal of Machine Learning Research}, 11\penalty0 (104):\penalty0 3137--3181, 2010.

\bibitem[Wierstra et~al.(2014)Wierstra, Schaul, Glasmachers, Sun, Peters, and Schmidhuber]{wierstra2014nes}
Daan Wierstra, Tom Schaul, Tobias Glasmachers, Yi~Sun, Jan Peters, and J{\"u}rgen Schmidhuber.
\newblock Natural evolution strategies.
\newblock \emph{Journal of Machine Learning Research}, 15\penalty0 (27):\penalty0 949--980, 2014.

\bibitem[Williams et~al.(2017{\natexlab{a}})Williams, Aldrich, and Theodorou]{williams2017mppi}
Grady Williams, Andrew Aldrich, and Evangelos~A. Theodorou.
\newblock {Model Predictive Path Integral Control: From Theory to Parallel Computation}.
\newblock \emph{Journal of Guidance, Control, and Dynamics}, 40\penalty0 (2):\penalty0 344--357, 2017{\natexlab{a}}.
\newblock \doi{10.2514/1.G001921}.

\bibitem[Williams et~al.(2017{\natexlab{b}})Williams, Wagener, Goldfain, Drews, Rehg, Boots, and Theodorou]{williams2017information}
Grady Williams, Nolan Wagener, Brian Goldfain, Paul Drews, James~M. Rehg, Byron Boots, and Evangelos~A. Theodorou.
\newblock {Information Theoretic {MPC} for Model-Based Reinforcement Learning}.
\newblock In \emph{Proceedings of the IEEE International Conference on Robotics and Automation}, pp.\  1714--1721, 2017{\natexlab{b}}.
\newblock \doi{10.1109/ICRA.2017.7989202}.

\bibitem[Xu et~al.(2022)Xu, Korba, and Slep{\v{c}}ev]{xu2022quantization}
Lantian Xu, Anna Korba, and Dejan Slep{\v{c}}ev.
\newblock {Accurate Quantization of Measures via Interacting Particle-based Optimization}.
\newblock In \emph{Proceedings of the 39th International Conference on Machine Learning}, volume 162 of \emph{Proceedings of Machine Learning Research}, pp.\  24576--24595. PMLR, 2022.

\bibitem[Yao et~al.(2024)Yao, Huang, and Yang]{yao2024klflow}
Rentian Yao, Linjun Huang, and Yun Yang.
\newblock {Minimizing Convex Functionals over Space of Probability Measures via {KL} Divergence Gradient Flow}.
\newblock In \emph{Proceedings of the 27th International Conference on Artificial Intelligence and Statistics}, volume 238 of \emph{Proceedings of Machine Learning Research}, pp.\  2530--2538. PMLR, 2024.

\end{thebibliography}
\bibliographystyle{iclr2027_conference}
\newpage
\appendix
\raggedbottom
\section{Appendix}

\subsection{Why Cole--Hopf is applied after particle approximation}
\label{app:cole_hopf_obstruction}

We explain why the Cole--Hopf transform is applied to the projected \(N\)-particle HJB, rather
than directly to the mean-field Master equation. The formal infinite-dimensional Master equation
has the form
\[
\partial_t V_\varepsilon(t,\mu)
+
\frac12\int_{\cH}
\Tr\!\left(Q\partial_x\partial_\mu V_\varepsilon(t,\mu)(x)\right)\mu(dx)
-
\frac1{2\varepsilon}
\int_{\cH}
\|Q^{1/2}\partial_\mu V_\varepsilon(t,\mu)(x)\|^2\mu(dx)
=0 .
\]
In finite-dimensional HJB equations, the Cole--Hopf transform works because the Laplacian of
\(u=\exp(-V/\varepsilon)\) produces a quadratic gradient term that cancels the Hamiltonian.
This cancellation does not occur at the level of the idiosyncratic-noise Master equation.

Indeed, set $u(t,\mu):=\exp(-V_\varepsilon(t,\mu)/\varepsilon)$. Then $\partial_\mu u(t,\mu)(x)
    =
    -\frac1\varepsilon u(t,\mu)\partial_\mu V_\varepsilon(t,\mu)(x)$. Taking the spatial derivative in \(x\) gives $\partial_x\partial_\mu u(t,\mu)(x)
    =
    -\frac1\varepsilon u(t,\mu)\partial_x\partial_\mu V_\varepsilon(t,\mu)(x)$. There is no additional quadratic term, because \(u(t,\mu)\) depends on the measure variable
\(\mu\), but not on the microscopic spatial variable \(x\). Consequently,
\[
\partial_t u(t,\mu)
+
\frac12\int_{\cH}
\Tr\!\left(Q\partial_x\partial_\mu u(t,\mu)(x)\right)\mu(dx)
=
-\frac{u(t,\mu)}{2\varepsilon^2}
\int_{\cH}
\|Q^{1/2}\partial_\mu V_\varepsilon(t,\mu)(x)\|^2\mu(dx),
\]
so the nonlinear control penalty remains. Thus the idiosyncratic Master equation is not linearized
by the Cole--Hopf transform.

This is the reason for passing first to the projected \(N\)-particle problem. In the particle HJB,
the state variable is the finite-dimensional vector
\(\mathbf x=(x_1,\ldots,x_N)\), and the second-order operator is the ordinary particle Laplacian $\frac12\sum_{i=1}^N\Tr(Q_KD_{x_i}^2)$. For
$u_\varepsilon^{K,N}
    =
    \exp(-v_\varepsilon^{K,N}/\varepsilon)$, the finite-dimensional chain rule produces the missing quadratic gradient terms, which cancel the
particle Hamiltonian. This yields a linear heat equation and hence the risk-sensitive
Feynman--Kac representation used in Section~\ref{sec:scmo}.
\subsection{A curvature diagnostic for nonmatched execution covariance}
\label{app:unmatched-curvature-diagnostic}

We give a controlled experiment that isolates the signed covariance term in
Theorem~\ref{thm:fixed-covariance-bound}. Consider the one-dimensional
additive measure objective $G(\mu)
=
\int_{\mathbb R} f(x)\,\mu(dx)$, $f(x)=\frac14(x^2-1)^2$. For an empirical measure, the extensive terminal cost is $F_N(\mathbf x)
=
\sum_{i=1}^N f(x_i)$. Consequently, the matched particle value factorizes as
\[
v_\varepsilon^N(t,\mathbf x)
=
\sum_{i=1}^N v_\varepsilon(t,x_i),
\]
where
\[
v_\varepsilon(t,x)
=
-\varepsilon
\log
\mathbb E
\left[
\exp\left(
-\frac{f(x+\sqrt{(1-t)q_{\mathrm{prop}}}Z)}{\varepsilon}
\right)
\right],
\qquad
Z\sim\mathcal N(0,1).
\]
The exact matched feedback and covariance-weighted curvature are therefore
\[
b(t,x)
=
-\frac{q_{\mathrm{prop}}}{\varepsilon}
\partial_xv_\varepsilon(t,x),
\qquad
\kappa_v(t,x)
=
q_{\mathrm{prop}}\partial_{xx}v_\varepsilon(t,x).
\]
We evaluate these quantities by Gauss--Hermite quadrature. In particular, if
\(\pi_{t,x}\) denotes the Gibbs-tilted Gaussian proposal, i.e., $\pi_{t,x}(dy):=\dfrac{e^{-f(y)/\varepsilon}\,\mathcal N(x,(1-t)q_{\mathrm{prop}})(dy)}{\int_{\mathbb R}e^{-f(z)/\varepsilon}\,\mathcal N(x,(1-t)q_{\mathrm{prop}})(dz)}$, then $\partial_{xx}v_\varepsilon(t,x)
=
\mathbb E_{\pi_{t,x}}[f''(Y)]
-
\frac1\varepsilon
\operatorname{Var}_{\pi_{t,x}}(f'(Y))$, so the continuation-value curvature is obtained without finite differences.

For a fixed variance ratio \(\alpha_0>0\), define the exact one-step effect
\[
\begin{aligned}
c_h(\alpha_0;t,x)
={}&
\mathbb E\left[
 v_\varepsilon\left(
 t+h,
 \overline x_h+
 \sqrt{h\alpha_0q_{\mathrm{prop}}}Z
 \right)
\right]
\\
&-
\mathbb E\left[
 v_\varepsilon\left(
 t+h,
 \overline x_h+
 \sqrt{hq_{\mathrm{prop}}}Z
 \right)
\right],
\qquad
\overline x_h=x+hb(t,x).
\end{aligned}
\]
The smooth theory in Appendix \ref{app:convergence} predicts
\[
c_h(\alpha_0;t,x)
=
\frac h2(\alpha_0-1)
\kappa_v(t+h,\overline x_h)
+O(h^2).
\]
Because
\[
f''(x)=3x^2-1,
\]
the terminal objective is locally concave near the central barrier and
locally convex near the two wells. The continuation value preserves this
sign separation over the parameter regime used below. Hence hot execution is
predicted to be favorable near the barrier, whereas cold execution is
predicted to be favorable near a well.

We use $\varepsilon=0.2$, $q_{\mathrm{prop}}=0.25$, $N=256$, $M=48$, and compare $\alpha_0\in\{0.25,0.5,1,1.5,2\}$. For each value of \(\alpha_0\), we simulate the exact discrete chain
\[
X_{m+1}^i
=
X_m^i
+h b(t_m,X_m^i)
+
\sqrt{h\alpha_0q_{\mathrm{prop}}}\,Z_{m+1}^i.
\]
We consider two initial clouds: a collapsed cloud at the central barrier,
\(X_0^i=0\), and a collapsed cloud at the right well, \(X_0^i=1\). The
experiment uses 50 independent runs and common Gaussian execution noises
across the different values of \(\alpha_0\) within each run.

At every step, we evaluate the exact one-step covariance effect and form
\[
\mathcal G_{h}
=
\frac1N\sum_m\mathbb E[(-c_h)_+],
\qquad
\mathcal P_{h}
=
\frac1N\sum_m\mathbb E[(c_h)_+],
\]
as well as the matched Euler residual \(\mathcal E_h\) defined in~\Eqref{Euler_residual}. We report the signed
covariance contribution $\mathcal C_{h}
=-\mathcal G_{h}
+\mathcal P_{h}$ and the complete implementation contribution $\mathcal E_h
-\mathcal G_{h}
+\mathcal P_{h}$.

Thus the experiment separately tests whether curvature makes the covariance
term favorable and whether this gain is large enough to dominate the Euler
residual.

\begin{table}[H]
\centering
\small
\setlength{\tabcolsep}{3.5pt}
\caption{Curvature diagnostic for nonmatched execution covariance over 50 independent runs. Negative $\mathcal C=-\mathcal G+\mathcal P$ is favorable; negative $\mathcal E+\mathcal C$ means that the gain also dominates the Euler residual.}
\label{tab:unmatched-curvature-diagnostic}
\begin{tabular}{lrrrrrrr}
\toprule
Initialization & $\alpha_0$ & Final $G$ & $\mathcal G$ & $\mathcal{P}$ & $\mathcal C$ & $\mathcal E+\mathcal C$ & Sign agr. \\
\midrule
barrier & 0.25 & 0.1896 $\pm$ 0.003749 & 0.001428 & 0.05052 & 0.04909 & 0.04967 & 0.999 \\
 & 0.50 & 0.1619 $\pm$ 0.004897 & 0.00355 & 0.02788 & 0.02433 & 0.02485 & 0.999 \\
 & 1.00 & 0.1359 $\pm$ 0.005757 & 0 & 0 & 0 & 0.0005928 & -- \\
 & 1.50 & 0.1269 $\pm$ 0.006398 & 0.0192 & 0.01319 & -0.006011 & -0.005289 & 0.997 \\
 & 2.00 & 0.1267 $\pm$ 0.007468 & 0.03442 & 0.03374 & -0.00068 & 0.0001851 & 0.996 \\
\addlinespace
well & 0.25 & 0.02055 $\pm$ 0.001712 & 0.08137 & 3.023e-05 & -0.08134 & -0.07982 & 1.000 \\
 & 0.50 & 0.03843 $\pm$ 0.003048 & 0.05316 & 0.0002451 & -0.05291 & -0.05134 & 1.000 \\
 & 1.00 & 0.0705 $\pm$ 0.005405 & 0 & 0 & 0 & 0.001691 & -- \\
 & 1.50 & 0.09788 $\pm$ 0.008226 & 0.002241 & 0.05093 & 0.04869 & 0.0505 & 0.999 \\
 & 2.00 & 0.1221 $\pm$ 0.01152 & 0.006136 & 0.1011 & 0.09495 & 0.09687 & 0.999 \\
\addlinespace
\bottomrule
\end{tabular}
\end{table}
Table~\ref{tab:unmatched-curvature-diagnostic} confirms the predicted curvature-dependent reversal. From the barrier initialization, where continuation-value curvature is predominantly negative, hot execution (\(\alpha_0>1\)) produces a favorable negative covariance contribution; at \(\alpha_0=1.5\), this gain also dominates the Euler residual. From the well initialization, where curvature is predominantly positive, cold execution (\(\alpha_0<1\)) is favorable and substantially lowers the final objective, whereas hot execution is unfavorable. The sign prediction agrees with the exact one-step effect in at least \(99.6\%\) of evaluated cases. 
\subsection{Full SCMO pseudocode}
% ============================================================
% ============================================================

% ============================================================

\label{app:full-algorithm}
\begin{algorithm}[H]
\caption{Multi-context SCMO with nonmatched covariance}
\label{alg:scmo_compact}
\begin{algorithmic}[1]
\footnotesize
\setlength{\abovedisplayskip}{2pt}
\setlength{\belowdisplayskip}{2pt}
\setlength{\abovedisplayshortskip}{1pt}
\setlength{\belowdisplayshortskip}{1pt}

\State \textbf{Input:}
Galerkin dimension \(K\); Galerkin objective \(G_K\); particle number \(N\); 
horizon \(T\); inner steps \(M\); outer loops \(L\);
context clouds \(R\); candidates per context \(S\);
temperature \(\varepsilon>0\);
proposal scale \(\sigma_{\mathrm{prop}}>0\); execution scale \(\sigma_{\mathrm{dyn}}>0\); and
\[
\Lambda
=
\operatorname{diag}(\lambda_1,\ldots,\lambda_K),
\]
with
\[
Q_{\mathrm{prop}}
=
\sigma_{\mathrm{prop}}^2\Lambda,
\qquad
Q_{\mathrm{dyn}}
=
\sigma_{\mathrm{dyn}}^2\Lambda.
\]

\State Initialize
\[
\mathbf C
=
(C_1,\ldots,C_N)^\top
\in\mathbb R^{N\times K},
\qquad
\Delta t=T/M.
\]

\For{\(\ell=1,\ldots,L\)}
    \For{\(m=0,\ldots,M-1\)}
        \State Set
        \[
        t_m=m\Delta t,
        \qquad
        \tau_m=T-t_m.
        \]

        \State Draw
        \[
        \mathbf Z^{(r)}
        \sim
        \mathcal N(0,I_{N\times K}),
        \qquad
        r=1,\ldots,R,
        \]
        \qquad\quad and form the complete context clouds
        \[
        \widetilde{\mathbf C}^{(r)}
        =
        \mathbf C
        +
        \sqrt{\tau_m}\,
        \sigma_{\mathrm{prop}}
        \mathbf Z^{(r)}\Lambda^{1/2}.
        \]

        \For{\(i=1,\ldots,N\)}
            \State Draw
            \[
            \xi_{i,r,s}
            \sim
            \mathcal N(0,I_K),
            \qquad
            r=1,\ldots,R,
            \quad
            s=1,\ldots,S,
            \]
            \qquad\quad\quad\,\, and set
            \[
            Y_{i,r,s}
            =
            C_i
            +
            \sqrt{\tau_m}\,
            \sigma_{\mathrm{prop}}
            \xi_{i,r,s}\Lambda^{1/2}.
            \]

            \State Compute the complete-cloud replacement scores
            \[
            \mathcal E_{i,r,s}
            =
            N\,G_K\!\left(
            \frac1N\delta_{Y_{i,r,s}}
            +
            \frac1N\sum_{p\neq i}
            \delta_{\widetilde C_p^{(r)}}
            \right).
            \]

            \State Let
            \[
            \mathcal E_i^{\min}
            =
            \min_{\substack{1\leq r\leq R\\1\leq s\leq S}}
            \mathcal E_{i,r,s},
            \]
            \qquad\quad\quad\,\, and define the globally normalized Gibbs weights
            \[
            w_{i,r,s}
            =
            \frac{
            \exp\!\left(
            -(\mathcal E_{i,r,s}-\mathcal E_i^{\min})/\varepsilon
            \right)
            }{
            \displaystyle
            \sum_{q=1}^{R}\sum_{j=1}^{S}
            \exp\!\left(
            -(\mathcal E_{i,q,j}-\mathcal E_i^{\min})/\varepsilon
            \right)
            }.
            \]

            \State Estimate the feedback drift
            \[
            \widehat\theta_i
            =
            \frac1{\tau_m}
            \sum_{r=1}^{R}\sum_{s=1}^{S}
            w_{i,r,s}
            (Y_{i,r,s}-C_i).
            \]
        \EndFor

        \State Form the simultaneous drift update
        \[
        \overline C_i
        =
        C_i+\Delta t\,\widehat\theta_i,
        \qquad
        i=1,\ldots,N,
        \]
        \qquad\quad and write
        \[
        \overline{\mathbf C}
        =
        (\overline C_1,\ldots,\overline C_N)^\top.
        \]

        \State Draw an execution noise
        \[
        \boldsymbol\eta_m
        \sim
        \mathcal N(0,I_{N\times K}),
        \]
        \qquad\quad independently of all context and candidate noises, and update
        \[
        \mathbf C
        \leftarrow
        \overline{\mathbf C}
        +
        \sqrt{\Delta t}\,
        \sigma_{\mathrm{dyn}}
        \boldsymbol\eta_m\Lambda^{1/2}.
        \]
    \EndFor
\EndFor

\State \textbf{Output:}
\[
\mu^N
=
\frac1N\sum_{i=1}^{N}\delta_{u^{(i)}},
\qquad
u^{(i)}
=
\sum_{k=1}^{K}C_{i,k}e_k.
\]

\end{algorithmic}
\end{algorithm}

\subsection{A diagnostic for the number of context clouds}
\label{app:context-cloud-quadratic}

We illustrate the error induced by a single complete context cloud using an
interacting quadratic measure objective. Let $\kappa > 0$, $G(\mu_{\mathbf x}^N)
=
\frac{\kappa}{2}
(
\frac1N\sum_{i=1}^N x_i-a
)^2$, and define the scaled terminal cost
$
F_N(\mathbf x)
=
N G(\mu_{\mathbf x}^N)
=
\frac{\kappa}{2N}
(
\sum_{i=1}^N x_i-Na
)^2$. Unlike an additive empirical average, this objective couples every particle
through the empirical mean.

Let $\tau,q >0$, under the Gaussian reference proposal
$Y_i = x_i+\sqrt{\tau q}\,Z_i$, $Z_i\overset{\mathrm{i.i.d.}}{\sim}\mathcal N(0,1)$,
the exact full-cloud Gibbs feedback is
\[
b_i(\mathbf x)
=
\frac1{\tau}
\frac{
\mathbb E[
(Y_i-x_i)e^{-F_N(\mathbf Y)/\varepsilon}
]
}{
\mathbb E[
e^{-F_N(\mathbf Y)/\varepsilon}
]
}
=
\frac{\kappa q}
{\varepsilon+\kappa q\tau}
\left(
a-\overline x
\right),
\]
where $\overline x
=
\frac1N\sum_{i=1}^N x_i$. Consider now the population limit of the one-context estimator. Fix one
Gaussian realization
$
\widetilde Y_j
=
x_j+\sqrt{\tau q}\,Z_j$, $j\neq i$, for the remaining particles and let the number of candidates tend to infinity.
The resulting conditional feedback is
$b_i^{(1)}
\left(
\mathbf x;
\widetilde{\mathbf Y}_{-i}
\right)
=
\frac{\kappa q}
{N\varepsilon+\kappa q\tau}
\left(
Na-x_i-\sum_{j\neq i}\widetilde Y_j
\right)$. Thus, even when the candidate expectation is evaluated exactly, the feedback
remains dependent on one random context. Taking expectation over this context
gives
$
\mathbb E[b_i^{(1)}]
=
\frac{N\kappa q}
{N\varepsilon+\kappa q\tau}
\left(
a-\overline x
\right)$. Consequently,
$
\frac{\mathbb E[b_i^{(1)}]}{b_i}
=
\frac{
N(\varepsilon+\kappa q\tau)
}{
N\varepsilon+\kappa q\tau
}$, which approaches \(N\) as \(\varepsilon\downarrow0\). Increasing the number
of candidates \(S\) therefore does not remove the one-context error; it only
approximates the random conditional feedback more accurately.

For \(R\) complete contexts and \(S\) candidates per context, define $L_{i,r,s}
=
\exp\left(
-\frac{\mathcal E_{i,r,s}}{\varepsilon}
\right)$, $
U_{i,r,s}
=
Y_{i,r,s}-x_i$, with $\mathcal E_{i,r,s}$ as defined in the~\Eqref{eq:main-finite-candidate-update}. The globally normalized estimator can be written as
$
\widehat b_i^{R,S}
=
\frac1{\tau}
\frac{
\sum_{r=1}^R\sum_{s=1}^S
L_{i,r,s}U_{i,r,s}
}{
\sum_{r=1}^R\sum_{s=1}^S
L_{i,r,s}
}$. For each context, let
$
A_{i,r}^{S}
=
\frac1S\sum_{s=1}^S L_{i,r,s}U_{i,r,s}$, $B_{i,r}^{S}
=
\frac1S\sum_{s=1}^S L_{i,r,s}$. The pairs
$(A_{i,r}^{S},B_{i,r}^{S})$ are independent and identically distributed over \(r\), with expectations
equal to the numerator and denominator of the full-cloud feedback.
Consequently, for every fixed \(S\geq1\),
$
\widehat b_i^{R,S}
\longrightarrow
b_i,\text{ almost surely as }R\to\infty$. This consistency requires the normalization to be performed jointly over all
\(RS\) context--candidate pairs.

We test this effect using $N=8,\,
\kappa=q=\tau=1,\,
\varepsilon=0.1,\,
a=0$, with initial empirical mean \(\overline x_0=0.5\). The exact initial feedback
is $b_i(\mathbf x_0)=-0.454545$, whereas the expected population one-context feedback is $\mathbb E[b_i^{(1)}(\mathbf x_0)]
=
-2.222222$. The total score budget is fixed at $RS=256$, and the experiment is repeated over \(50\) independent runs for
\(20\) updates. The initial objective is \(0.125\).

\begin{table}[H]
\centering
\small
\setlength{\tabcolsep}{4pt}
\caption{
Interacting-quadratic diagnostic with fixed context--candidate budget
\(RS=256\). Final objectives are mean \(\pm\) standard deviation.
}
\label{tab:context-cloud-quadratic}
\begin{tabular}{rrrrcc}
\toprule
\(R\)
&
\(S\)
&
\makecell{Mean feedback\\estimate}
&
\makecell{Feedback\\RMSE}
&
\makecell{Final\\objective}
&
\makecell{\(\Pr(G\leq10^{-2})\)}
\\
\midrule
1  & 256 & \(-1.7684\) & \(1.6139\) &
\(0.18222\pm0.17646\) & \(0.16\)
\\
2  & 128 & \(-1.3169\) & \(1.2310\) &
\(0.10094\pm0.12844\) & \(0.22\)
\\
4  & 64  & \(-0.8812\) & \(0.7917\) &
\(0.02242\pm0.03610\) & \(0.60\)
\\
8  & 32  & \(-0.6242\) & \(0.4493\) &
\(0.00522\pm0.00774\) & \(0.86\)
\\
16 & 16  & \(-0.5148\) & \(0.2442\) &
\(0.00209\pm0.00295\) & \(0.96\)
\\
32 & 8   & \(-0.4816\) & \(0.1525\) &
\(0.00141\pm0.00198\) & \(1.00\)
\\
\bottomrule
\end{tabular}
\end{table}

The one-context method is worse on average than the initial cloud, despite
using all \(256\) candidates in that single context. Reallocating the same
score budget over several independent contexts substantially improves both
the feedback approximation and the final optimization result. The choices
\(R=4\) and \(R=8\) already remove most of the one-context instability while
retaining respectively \(64\) and \(32\) candidates per context. This
provides a computational rationale for using several complete context clouds
rather than one context with a very large candidate pool.
\paragraph{Finite-candidate rate diagnostic.}
To complement the fixed-budget allocation study in Table~3, we
directly examine the mean-square rate in Proposition~C.1 using the
same interacting-quadratic objective, for which the exact full-cloud
feedback is available in closed form. We set
\[
N=8,\qquad
\kappa=0.02,\qquad
q=\tau=1,\qquad
\varepsilon=0.1,\qquad
a=0,\qquad
\bar x=0.5.
\]
Thus,
\[
b_i(x)
=
\frac{\kappa q}{\varepsilon+\kappa q\tau}(a-\bar x)
=
-\frac{1}{12}.
\]
For each
\(S\in\{1,4,16,64\}\), we vary
\(R\in\{8,16,32,64,128,256,512\}\) and estimate
\[
\mathbb{E}\!\left[
  \left|\widehat b_i^{R,S}(x)-b_i(x)\right|^2
\right]
\]
over \(10{,}000\) independent repetitions. As shown in
Figure~\ref{fig:finite-candidate-rate}, the empirical MSE decreases
approximately proportionally to \(R^{-1}\) for every fixed \(S\).
The fitted log--log slopes over \(R\geq128\) are
\(-0.991\), \(-0.996\), \(-1.005\), and \(-0.997\) for
\(S=1,4,16,64\), respectively, consistent with the
\(O(R^{-1}+(RS)^{-1})\) upper rate.

\begin{figure}[H]
  \centering
  \includegraphics[width=0.72\linewidth]
  {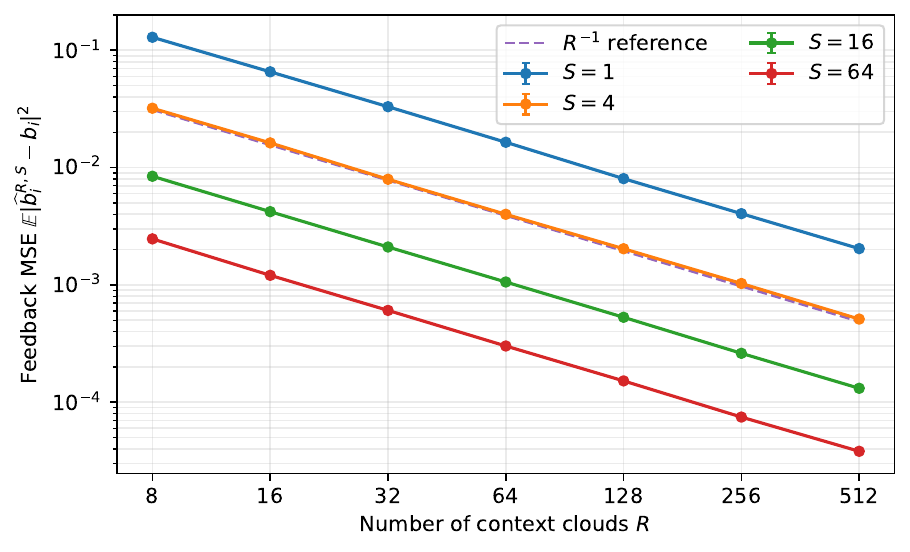}
  \caption{
  Finite-candidate rate diagnostic for the interacting-quadratic
  objective. Points show the empirical feedback MSE over
  \(10{,}000\) independent repetitions, with \(95\%\) Monte Carlo
  error bars. For each fixed candidate count \(S\), the MSE follows
  the predicted \(R^{-1}\) decay. Increasing \(S\) lowers the
  within-context candidate contribution, while increasing \(R\)
  controls the context-cloud contribution.
  }
  \label{fig:finite-candidate-rate}
\end{figure}
\section{Convergence Analysis for matched covariance problems}
\label{matched_problem}
In this section, we give detailed convergence analysis on matched covariance problems, i.e., $Q_{prop}=Q_{dyn}$. Recall the settings in Section \ref{sec:scmo}.
\paragraph{Algorithmic approximation.}
We provide two levels of convergence analysis. The first is a PDE-free qualitative result under weak recovery assumptions. It does not use the Master equation and can cover nonsmooth or discontinuous objectives. The second is a quantitative error decomposition for the idealized continuous-time projected particle value. This sharper result requires a verified classical regularized value function and analytic bounds. Neither result is a convergence theorem for the finite-candidate, time-discretized SCMO implementation.

\subsection{Convergence at the qualitative level}
\label{app:proof-qualitative}
The qualitative result uses only the finite-dimensional risk-sensitive representation and a robust recovery condition.  It does not require a Master equation, Lions differentiability, or convexity.

\begin{ass}[Robust recovery for possibly discontinuous objectives]
\label{ass:robust-recovery}
Assume that $G:\cP_2(\cH)\to\bR$ is Borel measurable and bounded from below, and that
\[
    -\infty<\mathcal V_0:=\inf_{\mu\in\cP_2(\cH)}G(\mu)<\infty .
\]
For every $\eta>0$ and every $K_0,N_0\ge1$, assume that there exist integers $K\ge K_0$, $N\ge N_0$, a cloud $\mathbf z\in(\cH_K)^N$, and a radius $r>0$ such that
\[
    \sup_{\mathbf y\in B_r(\mathbf z)}G(\mu_{\mathbf y}^N)
    \le
    \mathcal V_0+\eta.
\]
\end{ass}

\begin{thm}[PDE-free convergence under robust recovery]
\label{thm:qualitative-robust-recovery}
Assume that $Q$ is trace class with strictly positive eigenvalues and that Assumption~\ref{ass:robust-recovery} holds.  Let $v_\varepsilon^{K,N}$ be the projected particle value with terminal cost $NG_K(\mu_{\mathbf x}^N)$.  Then there exist sequences $K_m,N_m\to\infty$ and $\varepsilon_m\downarrow0$ such that
\[
    \frac1{N_m}v_{\varepsilon_m}^{K_m,N_m}(0,\Pi_{K_m}\mathbf x_0)
    \longrightarrow
    \mathcal V_0.
\]
\end{thm}

\begin{proof}
For fixed $K,N$, let $\gamma_{K,N}$ be the law of $\Pi_K\mathbf x_0+\Delta\mathbf W_1^K$ on $(\cH_K)^N$, and define
\[
    f_{K,N}(\mathbf z):=G_K(\mu_{\mathbf z}^N).
\]
The finite-dimensional risk-sensitive representation gives
\[
    \frac1N v_\varepsilon^{K,N}(0,\Pi_K\mathbf x_0)
    =
    -\frac\varepsilon N
    \log\int
    \exp\left(-\frac N\varepsilon f_{K,N}(\mathbf z)\right)\gamma_{K,N}(d\mathbf z).
\]
Because $G\ge \mathcal V_0$, the right-hand side is bounded below by $\mathcal V_0$.

Fix $\eta>0$ and choose $K_0,N_0$.  By robust recovery, choose $K\ge K_0$, $N\ge N_0$, $\mathbf z\in(\cH_K)^N$, and $r>0$ such that
\[
    G_K(\mu_{\mathbf y}^N)\le \mathcal V_0+\eta,
    \qquad \mathbf y\in B_r(\mathbf z).
\]
Since $\gamma_{K,N}$ has full support on $(\cH_K)^N$,
\[
    p:=\gamma_{K,N}(B_r(\mathbf z))>0.
\]
Therefore
\[
    \int
    \exp\left(-\frac N\varepsilon f_{K,N}(\mathbf y)\right)\gamma_{K,N}(d\mathbf y)
    \ge
    p\exp\left(-\frac N\varepsilon(\mathcal V_0+\eta)\right).
\]
It follows that
\[
    \frac1N v_\varepsilon^{K,N}(0,\Pi_K\mathbf x_0)
    \le
    \mathcal V_0+\eta-\frac\varepsilon N\log p.
\]
Choose $\varepsilon>0$ small enough so that $-(\varepsilon/N)\log p\le\eta$.  Then
\[
    \mathcal V_0
    \le
    \frac1N v_\varepsilon^{K,N}(0,\Pi_K\mathbf x_0)
    \le
    \mathcal V_0+2\eta.
\]
Taking $\eta=\eta_m\downarrow0$ and $K_0=N_0=m$ at the $m$-th step gives the desired diagonal sequence.
\end{proof}

\begin{cor}[Upper semicontinuous objectives]
\label{cor:usc}
Assume that $G:\cP_2(\cH)\to\bR$ is bounded from below and upper semicontinuous with respect to $\mathcal{W}_2$.  Suppose $-\infty<\mathcal V_0<\infty$ and that for every $\eta>0$ there exists $\mu_\eta\in\cP_2(\cH)$ with
\[
    G(\mu_\eta)\le \mathcal V_0+\eta.
\]
Then Assumption~\ref{ass:robust-recovery} holds, and hence the conclusion of Theorem~\ref{thm:qualitative-robust-recovery} follows.
\end{cor}

\begin{proof}
Fix $\eta>0$ and $K_0,N_0\ge1$.  Choose $\mu_\eta$ with $G(\mu_\eta)\le \mathcal V_0+\eta/3$.  Empirical measures with atoms in finite-dimensional Galerkin subspaces are dense in $\cP_2(\cH)$ under $\mathcal{W}_2$.  Hence there exist $K\ge K_0$, $N\ge N_0$, and $\mathbf z\in(\cH_K)^N$ such that $\mathcal{W}_2((\iota_K)_\#\mu_{\mathbf z}^N,\mu_\eta)$ is sufficiently small.  By upper semicontinuity,
\[
    G_K(\mu_{\mathbf z}^N)=G((\iota_K)_\#\mu_{\mathbf z}^N)
    \le
    G(\mu_\eta)+\eta/3
    \le
    \mathcal V_0+2\eta/3.
\]
The map $\mathbf y\mapsto (\iota_K)_\#\mu_{\mathbf y}^N$ is continuous from $(\cH_K)^N$ to $\cP_2(\cH)$.  Applying upper semicontinuity once more, there exists $r>0$ such that
\[
    \sup_{\mathbf y\in B_r(\mathbf z)}G_K(\mu_{\mathbf y}^N)
    \le
    \mathcal V_0+\eta.
\]
This is robust recovery.
\end{proof}

\begin{cor}[Indicator objectives with robust margins]
\label{cor:indicator}
Let $H:\cP_2(\cH)\to\bR$ be bounded below and upper semicontinuous.  Suppose
\[
    G(\mu)=H(\mu)+\sum_{\ell=1}^{L_b}a_\ell\1_{\mathcal B_\ell}(\mu)-\sum_{m=1}^{L_g}b_m\1_{\mathcal A_m}(\mu),
    \qquad a_\ell,b_m\ge0,
\]
where each bad-event set $\mathcal B_\ell$ is closed and each good-event set $\mathcal A_m$ is open in $\mathcal{W}_2$.  Then $G$ is upper semicontinuous. 
\end{cor}

\subsection{Convergence at the quantitative level}
\label{app:quantitative-assumptions}
For any probability measures $\mu$, we write the second moment of $\mu$ as $M_2(\mu):=\int \|x\|^2\mu(dx)$. For $\nu\in\cP_2(\cH_K)$, define the projected regularized value
\begin{equation}
\label{eq:projected-mf-value}
    V_\varepsilon^K(t,\nu)
    :=
    \inf_{\theta\in\Theta_K}
    \left\{
    G_K(\Law_{\cH_K}(X_1^{K,t,\xi;\theta}))+\E\left[\frac\varepsilon2\int_t^1\|\theta_s\|_{Q_K}^2\,ds\right]
    \right\},
\end{equation}
where $\Law_{\cH_K}(\xi)=\nu$, the controls are adapted and $\cH_K$-valued, and
\[
    X_s^{K,t,\xi;\theta}
    =
    \xi+\int_t^s\theta_r\,dr+W_s^K-W_t^K.
\]
Here $W^K:=\Pi_KW$ is a $Q_K$-Wiener process on $\cH_K$, and
\[
    \|h\|_{Q_K}:=\|Q_K^{-1/2}h\|_{\cH_K}.
\]
Formally, $V_\varepsilon^K$ solves the projected Master equation on $\cP_2(\cH_K)$:
\begin{equation}
\label{eq:projected-master}
\partial_tV_\varepsilon^K(t,\nu)
+
\frac12\int_{\cH_K}\Tr\left(Q_K\partial_x\partial_\mu V_\varepsilon^K(t,\nu)(x)\right)\nu(dx)
-
\frac1{2\varepsilon}\int_{\cH_K}\|Q_K^{1/2}\partial_\mu V_\varepsilon^K(t,\nu)(x)\|^2\nu(dx)=0,
\end{equation}
with terminal condition $V_\varepsilon^K(1,\nu)=G_K(\nu)$.

For $\mathbf x\in(\cH_K)^N$, define
\begin{equation}
\label{eq:projected-particle-value}
    v_\varepsilon^{K,N}(t,\mathbf x)
    :=
    \inf_{\theta^1,\ldots,\theta^N}
    \E\left[
        NG_K(\mu^N_{\mathbf X_1^K})
        +
        \frac\varepsilon2\sum_{i=1}^N\int_t^1\|\theta_s^i\|_{Q_K}^2\,ds
    \right],
\end{equation}
where
\[
    X_s^{K,i}=x_i+\int_t^s\theta_r^i\,dr+W_s^{K,i}-W_t^{K,i},
\]
with independent $Q_K$-Wiener processes $W^{K,1},\ldots,W^{K,N}$.  The projected particle HJB is
\begin{equation}
\label{eq:projected-particle-hjb}
\partial_t v_\varepsilon^{K,N}
+
\frac12\sum_{i=1}^N\Tr\left(Q_KD_{x_i}^2v_\varepsilon^{K,N}\right)
-
\frac1{2\varepsilon}\sum_{i=1}^N\|Q_K^{1/2}D_{x_i}v_\varepsilon^{K,N}\|^2
=0,
\end{equation}
with terminal condition $v_\varepsilon^{K,N}(1,\mathbf x)=NG_K(\mu_{\mathbf x}^N)$.
The Cole--Hopf transform gives the risk-sensitive representation
\[
    v_\varepsilon^{K,N}(t,\mathbf x)
    =
    -\varepsilon\log\E\left[
    \exp\left(-\frac{N}{\varepsilon}G_K(\mu^N_{\mathbf x+\Delta\mathbf W_{1-t}^K})\right)
    \right].
\]

Our proof has two layers. First, $N^{-1}v_\varepsilon^{K,N}$ approximates $V_\varepsilon^K$ as $N\to\infty$ for fixed $K$ and $\varepsilon$.  Second, $V_\varepsilon^K(0,\delta_{\Pi_Kx_0})$ is compared directly with $\mathcal V_0$.

\begin{ass}[Projected classical solution, growth, and particle residual]
\label{ass:projected-classical}
For the quantitative results in this section, assume that the terminal objective $G$ is bounded below, continuous and has at most quadratic growth:
\begin{equation}
\label{eq:G-quadratic-growth}
    |G(\mu)|
    \le
    C_G\left(1+M_2(\mu)\right),
    \qquad \mu\in\cP_2(\cH).
\end{equation}
For every $\varepsilon>0$ and $K\ge1$, we assume that the projected Master equation (\ref{eq:projected-master}) admits a solution $V_\varepsilon^K$ on $[0,1]\times\cP_2(\cH_K)$, continuous up to $t=1$ with terminal value $G_K$.  On $[0,1)\times\cP_2(\cH_K)$, the derivatives
\[
    \partial_tV_\varepsilon^K,
    \qquad
    \partial_\mu V_\varepsilon^K,
    \qquad
    \partial_x\partial_\mu V_\varepsilon^K,
    \qquad
    \partial_{\mu\mu}^2V_\varepsilon^K
\]
exist in the classical Lions sense needed below. On top of that, we assume the Carmona--Delarue chain-rule regularity on $\cP_2(\cH_K)$: the versions of $\partial_\mu V_\varepsilon^K$ and $\partial_x\partial_\mu V_\varepsilon^K$ are continuous at support points of the measure argument and locally bounded on bounded subsets of $[0,1)\times\cP_2(\cH_K)\times\cH_K$.

We also assume the following growth and local Lipschitz bounds: There exists $C_{\varepsilon,K}<\infty$ such that, for all $t\in[0,1)$,
\begin{equation}
\label{eq:V-derivative-linear-growth}
    \|\partial_\mu V_\varepsilon^K(t,\nu)(x)\|
    +
    \|\partial_x\partial_\mu V_\varepsilon^K(t,\nu)(x)\|_{\op}
    \le
    C_{\varepsilon,K}\left(1+\|x\|+M_2(\nu)^{1/2}\right).
\end{equation}
Moreover, we suppose that for every $R<\infty$ and every $\tau<1$, there is a constant
$C_{\varepsilon,K,R,\tau}<\infty$ such that, for all $t\in[0,\tau]$, all $x,x'\in\cH_K$ with
$\|x\|\vee\|x'\|\le R$, and all $\nu,\nu'\in\cP_2(\cH_K)$ with
$M_2(\nu)\vee M_2(\nu')\le R$,
\begin{equation}
\label{eq:partial-mu-local-lipschitz}
    \|\partial_\mu V_\varepsilon^K(t,\nu)(x)
    -\partial_\mu V_\varepsilon^K(t,\nu')(x')\|
    \le
    C_{\varepsilon,K,R,\tau}
    \left(\|x-x'\|+\mathcal{W}_{2}(\nu,\nu')\right).
\end{equation}
Finally, we assume that there exists $L_{\varepsilon,K}<\infty$ such that
\begin{equation}
\label{eq:L-eps-K}
    \left|
    \int_{\cH_K}\Tr\left(Q_K\partial_{\mu\mu}^2V_\varepsilon^K(t,\nu)(y)(y)\right)\nu(dy)
    \right|
    \le L_{\varepsilon,K},
    \qquad (t,\nu)\in[0,1)\times\cP_2(\cH_K).
\end{equation}
\end{ass}

\begin{lem}[Projected verification]
\label{lem:projected-verification}
Under Assumptions~\ref{ass:projected-classical}, the classical solution $V_\varepsilon^K$ coincides with the projected value function defined in \Eqref{eq:projected-mf-value}.  In particular, it is unique within this verified classical class.
\end{lem}

\begin{proof}
Fix an admissible projected control $\theta$, and set
\[
    X_s:=X_s^{K,t,\xi;\theta},
    \qquad
    \nu_s:=\Law_{\cH_K}(X_s).
\]
By Assumptions~\ref{ass:projected-classical}, the Carmona--Delarue chain rule applies to $s\mapsto V_\varepsilon^K(s,\nu_s)$ in expectation.  Using the projected Master equation (\ref{eq:projected-master}), we obtain, for $r<1$,
\begin{align*}
V_\varepsilon^K(r,\nu_r)-V_\varepsilon^K(t,\nu)
={}&
\int_t^r
\Bigg\{
\frac1{2\varepsilon}
\E\left[
\left\|Q_K^{1/2}\partial_\mu V_\varepsilon^K(s,\nu_s)(X_s)\right\|^2
\right]
\nonumber\\
&\hspace{2.5cm}
+
\E\left[
\left\langle
\partial_\mu V_\varepsilon^K(s,\nu_s)(X_s),
\theta_s
\right\rangle
\right]
\Bigg\}ds.
\end{align*}
Letting $r\uparrow1$ and using $V_\varepsilon^K(1,\nu_1)=G_K(\nu_1)$ gives
\begin{align*}
&G_K(\nu_1)
+
\frac\varepsilon2\E\int_t^1\|\theta_s\|_{Q_K}^2\,ds
-
V_\varepsilon^K(t,\nu)
\\
&\qquad=
\frac\varepsilon2\E\int_t^1
\left\|
\theta_s
+
\frac1\varepsilon Q_K\partial_\mu V_\varepsilon^K(s,\nu_s)(X_s)
\right\|_{Q_K}^2ds
\ge0.
\end{align*}
Therefore every admissible control has cost at least $V_\varepsilon^K(t,\nu)$.
It remains to check that the minimizing feedback is admissible.  By
 \Eqref{eq:V-derivative-linear-growth} and \Eqref{eq:partial-mu-local-lipschitz}, the finite-dimensional closed-loop McKean--Vlasov equation
with drift
\[
    -\frac1\varepsilon Q_K\partial_\mu V_\varepsilon^K(s,\nu_s)(X_s^K)
\]
is well posed.  Moreover, \Eqref{eq:V-derivative-linear-growth} and the standard second-moment
estimate for this closed-loop equation imply
\[
    \E\int_t^1
    \left\|
    -\frac1\varepsilon Q_K\partial_\mu V_\varepsilon^K(s,\nu_s)(X_s^K)
    \right\|_{Q_K}^2ds<\infty.
\]
For this feedback, the square term vanishes.  Hence equality is attained, and $V_\varepsilon^K$ is the
projected value function.  The same verification argument applied to any other function in this
class gives uniqueness.
\end{proof}

\begin{ass}
\label{ass:particle-fk}
Assume that the spatial gradients of $v_\varepsilon^{K,N}$ are locally Lipschitz in $\mathbf x$ on compact time intervals in $[0,1)$ and have at most linear growth: there is $C_{\varepsilon,K,N}<\infty$ such that
\begin{equation}
\label{eq:vw-gradient-linear-growth}
    \sum_{i=1}^N
    \|D_{x_i}v_\varepsilon^{K,N}(t,\mathbf x)\|
    \le
    C_{\varepsilon,K,N}(1+\|\mathbf x\|),
    \qquad t\in[0,1).
\end{equation}

\end{ass}

\begin{thm}[Particle approximation after Galerkin projection]
\label{thm:particle-after-projection}
Let Assumptions~\ref{ass:projected-classical} and~\ref{ass:particle-fk} hold, and let $v_\varepsilon^{K,N}$ be a classical solution of \Eqref{eq:projected-particle-hjb}.  Then, for every $t\in[0,1]$ and $\mathbf x\in(\cH_K)^N$,
\begin{equation}
\label{eq:particle-after-projection-bound}
    \left|
    \frac1N v_\varepsilon^{K,N}(t,\mathbf x)
    -
    V_\varepsilon^K(t,\mu_{\mathbf x}^N)
    \right|
    \le
    \frac{L_{\varepsilon,K}}{2N}(1-t).
\end{equation}
\end{thm}

\begin{proof}
Define $w_\varepsilon^{K,N}(t,\mathbf x):=NV_\varepsilon^K(t,\mu_{\mathbf x}^N)$.  The empirical Lions chain rule gives
\[
    D_{x_i}w_\varepsilon^{K,N}=\partial_\mu V_\varepsilon^K(t,\mu_{\mathbf x}^N)(x_i),
\]
\[
    D_{x_i}^2w_\varepsilon^{K,N}
    =
    \partial_x\partial_\mu V_\varepsilon^K(t,\mu_{\mathbf x}^N)(x_i)
    +
    \frac1N\partial_{\mu\mu}^2V_\varepsilon^K(t,\mu_{\mathbf x}^N)(x_i)(x_i).
\]
Applying the projected particle HJB operator to $w_\varepsilon^{K,N}$ gives
\begin{align*}
\mathcal L^{K,N}w_\varepsilon^{K,N}
={}&
N\Bigg[
\partial_tV_\varepsilon^K
+
\frac12\int_{\cH_K}\Tr(Q_K\partial_x\partial_\mu V_\varepsilon^K)d\mu_{\mathbf x}^N
-
\frac1{2\varepsilon}\int_{\cH_K}\|Q_K^{1/2}\partial_\mu V_\varepsilon^K\|^2d\mu_{\mathbf x}^N
\Bigg]
\\
&+
\frac12\int_{\cH_K}\Tr\left(Q_K\partial_{\mu\mu}^2V_\varepsilon^K(t,\mu_{\mathbf x}^N)(y)(y)\right)\mu_{\mathbf x}^N(dy).
\end{align*}
The bracketed term vanishes by \Eqref{eq:projected-master}.  Thus
\[
    \mathcal L^{K,N}w_\varepsilon^{K,N}=R^{K,N},
\]
where
\[
    R^{K,N}(t,\mathbf x)
    =
    \frac12\int_{\cH_K}\Tr\left(Q_K\partial_{\mu\mu}^2V_\varepsilon^K(t,\mu_{\mathbf x}^N)(y)(y)\right)\mu_{\mathbf x}^N(dy).
\]
By Assumption~\ref{ass:projected-classical}, $|R^{K,N}|\le L_{\varepsilon,K}/2$.

Let $e_\varepsilon^{K,N}:=v_\varepsilon^{K,N}-w_\varepsilon^{K,N}$.  Subtracting the HJB equations and using
$\|a\|^2-\|b\|^2=\langle a+b,a-b\rangle$ gives
\[
\partial_t e_\varepsilon^{K,N}
+
\frac12\sum_{i=1}^N\Tr(Q_KD_{x_i}^2e_\varepsilon^{K,N})
+
\sum_{i=1}^N\langle b_i(t,\mathbf x),D_{x_i}e_\varepsilon^{K,N}\rangle
=
-R^{K,N}(t,\mathbf x),
\]
with terminal $e_\varepsilon^{K,N}(1,\cdot)=0$.

Define the comparison drift
\[
    b_i(t,\mathbf x):=-\frac1{2\varepsilon}Q_K\left(D_{x_i}v_\varepsilon^{K,N}(t,\mathbf x)+D_{x_i}w_\varepsilon^{K,N}(t,\mathbf x)\right).
\]
By the preceding linear-growth and local-Lipschitz Assumptions \ref{ass:projected-classical} and \ref{ass:particle-fk}, the comparison diffusion
\[
    d\Xi_s^i=b_i(s,\mathbf\Xi_s)\,ds+dW_s^{K,i},
    \qquad \mathbf\Xi_t=\mathbf x,
\]
is nonexplosive on $[t,1]$ and satisfies the standard finite-dimensional moment bound
\[
    \mathbb{E}\sup_{s\in[t,1]}\|\mathbf\Xi_s\|^2<\infty.
\]
The common terminal condition for $v_\varepsilon^{K,N}$ and $w_\varepsilon^{K,N}$ gives the local terminal trace $e_\varepsilon^{K,N}(s,\cdot)\to0$ as $s\uparrow1$.  Together with the linear-gradient bound \Eqref{eq:vw-gradient-linear-growth}, this gives the quadratic envelope needed to pass the stopped Feynman--Kac terminal term to zero along $\mathbf\Xi$.

By Assumption~\ref{ass:particle-fk}, applying It\^o's formula to the stopped comparison diffusion and then sending the localization and terminal parameters to their limits yields
\[
    e_\varepsilon^{K,N}(t,\mathbf x)=\E\int_t^1R^{K,N}(s,\mathbf\Xi_s)\,ds.
\]
Therefore
\[
    |e_\varepsilon^{K,N}(t,\mathbf x)|
    \le
    \int_t^1\sup_{\mathbf z\in(\cH_K)^N}|R^{K,N}(s,\mathbf z)|\,ds
    \le
    \frac12L_{\varepsilon,K}(1-t).
\]
Dividing by $N$ proves the claim.
\end{proof}

\subsection{Projected regularization error}
\label{app:proof-quantitative}
\begin{ass}[Local continuity and comparator]
\label{ass:local-continuity-minimizer} Assume that the infimum $\mathcal V_0$ is attained by some $\mu^*\in\cP_2(\cH)$.  Let $X\sim\mu^*$.  Assume $G$ satisfies a local H\"older continuity condition with exponent $\alpha\in(0,1]$: there is a constant $L_G$ such that for all $\mu,\nu\in\cP_2(\cH)$,
\[
    |G(\mu)-G(\nu)|
    \le
    L_G\left(1+\mathcal{W}_2(\mu,\nu)+M_2(\mu)+M_2(\nu)\right)\mathcal{W}_2(\mu,\nu)^\alpha.
\]
For $K\ge1$, define
\[
    \Delta_K^{\mathrm{proj}}:=\E_{\mu^*}\|(I-\Pi_K)X\|^2,
\]
\[
    D_{CM,K}^2(\mu^*):=\sum_{j=1}^K\frac{\E_{\mu^*}[(X_j-x_{0,j})^2]}{\lambda_j}.
\]
\end{ass}

\begin{thm}[Projected regularization error]
\label{thm:projected-regularization-error}
Let Assumption~\ref{ass:local-continuity-minimizer} hold.  Then, for every $K\ge1$, every $\varepsilon\in(0,1]$, and every $\rho\in(0,1]$,
\begin{align}
\label{eq:projected-reg-lambda-bound}
0
\le
V_\varepsilon^K(0,\delta_{\Pi_Kx_0})-\mathcal V_0
\le{}&
C_*
\left(\rho\Tr(Q)+\Delta_K^{\mathrm{proj}}\right)^{\alpha/2}
+
\varepsilon\left[
\frac12D_{CM,K}^2(\mu^*)+\frac K2\log\frac1\rho
\right],
\end{align}
where $C_*$ is independent of $K$, $\varepsilon$, and $\rho$.  In particular, choosing $\rho=\varepsilon^{2/\alpha}$ gives
\begin{align}
\label{eq:projected-reg-eps-bound}
0
\le
V_\varepsilon^K(0,\delta_{\Pi_Kx_0})-\mathcal V_0
\le{}&
C_*
\left(\varepsilon^{2/\alpha}\Tr(Q)+\Delta_K^{\mathrm{proj}}\right)^{\alpha/2}
+
\frac\varepsilon2D_{CM,K}^2(\mu^*)
+
\frac{\varepsilon K}{\alpha}\log\frac1\varepsilon .
\end{align}
\end{thm}

\begin{proof}
Let
\[
    \gamma_{K,x_0}:=\Law_{\cH_K}(\Pi_Kx_0+W_1^K)=\mathcal N(\Pi_Kx_0,Q_K)
\]
be the uncontrolled projected terminal law.  By the finite-dimensional F\"ollmer drift representation on $\cH_K$, for every $\nu\ll\gamma_{K,x_0}$,
\begin{equation}
\label{eq:follmer-bound-projected}
    V_\varepsilon^K(0,\delta_{\Pi_Kx_0})
    \le
    G_K(\nu)+\varepsilon\mathcal H(\nu\mid\gamma_{K,x_0}),
\end{equation}
where \(\mathcal H(\nu\mid\gamma_{K,x_0})\) denotes relative entropy. Fix $\rho\in(0,1]$.  Let $X\sim\mu^*$ and let $\zeta_K\sim\mathcal N(0,Q_K)$ be independent.  Define the $\cH_K$-valued random variable
\[
    Y^{\rho,K}:=\Pi_KX+\sqrt\rho\,\zeta_K,
    \qquad
    \mu_K^\rho:=\Law_{\cH_K}(Y^{\rho,K}).
\]
Let
\[
    \widehat\mu_K^\rho:=(\iota_K)_\#\mu_K^\rho=\Law_{\cH}(Y^{\rho,K}).
\]
Conditionally on $X=x$,
\[
    Y^{\rho,K}\mid X=x\sim\mathcal N(\Pi_Kx,\rho Q_K).
\]
Since $\rho>0$ and $Q_K$ is nondegenerate on $\cH_K$, $\mu_K^\rho\ll\gamma_{K,x_0}$.

Define joint measures on $\cH\times\cH_K$ by
\[
    \eta(dx,dy):=\mu^*(dx)\,\mathcal N(\Pi_Kx,\rho Q_K)(dy),
    \qquad
    \xi(dx,dy):=\mu^*(dx)\,\mathcal N(\Pi_Kx_0,Q_K)(dy).
\]
The second marginals are $\mu_K^\rho$ and $\gamma_{K,x_0}$.  By data processing for relative entropy,
\[
    \mathcal H(\mu_K^\rho\mid\gamma_{K,x_0})
    \le
    \mathcal H(\eta\mid\xi).
\]
Since the first marginal is the same, the chain rule for relative entropy gives
\[
    \mathcal H(\eta\mid\xi)
    =
    \int_{\cH}
    \mathcal H\left(\mathcal N(\Pi_Kx,\rho Q_K)\mid\mathcal N(\Pi_Kx_0,Q_K)\right)
    \mu^*(dx).
\]
The Gaussian entropy formula gives
\[
    \mathcal H\left(\mathcal N(\Pi_Kx,\rho Q_K)\mid\mathcal N(\Pi_Kx_0,Q_K)\right)
    =
    \frac12\sum_{j=1}^K\frac{(x_j-x_{0,j})^2}{\lambda_j}
    +
    \frac K2(\rho-1-\log\rho).
\]
Since $\rho\in(0,1]$, $\rho-1-\log\rho\le\log(1/\rho)$.  Hence
\begin{equation}
\label{eq:projected-entropy-estimate}
    \mathcal H(\mu_K^\rho\mid\gamma_{K,x_0})
    \le
    \frac12D_{CM,K}^2(\mu^*)+\frac K2\log\frac1\rho.
\end{equation}

Couple $\widehat\mu_K^\rho$ with $\mu^*$ by using $(Y^{\rho,K},X)$ as $\cH$-valued random variables.  Then
\[
    \mathcal{W}_{2,\cH}^2(\widehat\mu_K^\rho,\mu^*)
    \le
    \E\|Y^{\rho,K}-X\|_{\cH}^2.
\]
Since
\[
    Y^{\rho,K}-X=-(I-\Pi_K)X+\sqrt\rho\,\zeta_K,
\]
and the two terms are orthogonal,
\begin{equation}
\label{eq:projected-coupling-estimate}
    \mathcal{W}_{2,\cH}^2(\widehat\mu_K^\rho,\mu^*)
    \le
    \Delta_K^{\mathrm{proj}}+
    \rho\Tr(Q_K)
    \le
    \Delta_K^{\mathrm{proj}}+
    \rho\Tr(Q).
\end{equation}
Also,
\[
    \int_{\cH}\|y\|^2\widehat\mu_K^\rho(dy)
    =
    \E\|Y^{\rho,K}\|^2
    \le
    2\E\|X\|^2+2\Tr(Q).
\]
Thus the local H\"older condition and optimality of $\mu^*$ imply
\begin{equation}
\label{eq:G-comparator-bound}
    G(\widehat\mu_K^\rho)-\mathcal V_0
    \le
    C_*
    \left(\rho\Tr(Q)+\Delta_K^{\mathrm{proj}}\right)^{\alpha/2}.
\end{equation}
Using \Eqref{eq:follmer-bound-projected} with $\nu=\mu_K^\rho$, and recalling that $G_K(\mu_K^\rho)=G(\widehat\mu_K^\rho)$, gives \Eqref{eq:projected-reg-lambda-bound}.  The lower bound follows from $G\ge\mathcal V_0$ and the nonnegativity of the running cost.  Setting $\rho=\varepsilon^{2/\alpha}$ gives \Eqref{eq:projected-reg-eps-bound}.
\end{proof}

\begin{thm}[Projection-first convergence of the idealized particle approximation]
\label{thm:projection-first-global}
Assume Assumptions~\ref{ass:projected-classical}, \ref{ass:particle-fk}, and~\ref{ass:local-continuity-minimizer} and $0<\varepsilon\leq 1$. Let $v_\varepsilon^{K,N}$ be a classical solution of the projected particle HJB \Eqref{eq:projected-particle-hjb}.  Then
\begin{align}
\label{eq:projection-first-final-bound}
0
\le
\frac1N v_\varepsilon^{K,N}(0,\Pi_K\mathbf x_0)-\mathcal V_0
\le{}&
\frac{L_{\varepsilon,K}}{2N}
+
C_*
\left(\varepsilon^{2/\alpha}\Tr(Q)+\Delta_K^{\mathrm{proj}}\right)^{\alpha/2}
+
\frac\varepsilon2D_{CM,K}^2(\mu^*)
+
\frac{\varepsilon K}{\alpha}\log\frac1\varepsilon .
\end{align}
Here $\Pi_K\mathbf x_0=(\Pi_Kx_0,\ldots,\Pi_Kx_0)$.
\end{thm}

\begin{proof}
Since $\mu_{\Pi_K\mathbf x_0}^N=\delta_{\Pi_Kx_0}$, the triangle inequality gives
\[
\frac1N v_\varepsilon^{K,N}(0,\Pi_K\mathbf x_0)-\mathcal V_0
\le
\left|
\frac1N v_\varepsilon^{K,N}(0,\Pi_K\mathbf x_0)
-
V_\varepsilon^K(0,\delta_{\Pi_Kx_0})
\right|
+
\left(V_\varepsilon^K(0,\delta_{\Pi_Kx_0})-\mathcal V_0\right).
\]
The first term is controlled by Theorem~\ref{thm:particle-after-projection}; the second term is controlled by Theorem~\ref{thm:projected-regularization-error}.  The lower bound follows from $G\ge\mathcal V_0$ and nonnegativity of the running cost.
\end{proof}

\begin{rmk}[Asymptotic regime]
The bound in \Eqref{eq:projection-first-final-bound} tends to zero along any joint choice $K=K_m$, $N=N_m$, and $\varepsilon=\varepsilon_m$ such that
\[
    \frac{L_{\varepsilon_m,K_m}}{N_m}\to0,
    \qquad
    \Delta_{K_m}^{\mathrm{proj}}\to0,
\]
\[
    \varepsilon_mD_{CM,K_m}^2(\mu^*)\to0,
    \qquad
    \varepsilon_mK_m\log\frac1{\varepsilon_m}\to0,\qquad \varepsilon_m\to0.
\]
\end{rmk}

\begin{rmk}
Theorem~\ref{thm:projection-first-global} is not a theorem for the fully discrete finite-candidate SCMO implementation.  It concerns the idealized continuous-time projected particle value $v_\varepsilon^{K,N}$.  Finite-candidate sampling and time-discretization errors must be controlled separately.
\end{rmk}

\subsection{Optional sufficient condition for uniform particle constants}
\label{app:uniform-L}
The main convergence theorem only needs the constants $L_{\varepsilon,K}$ in Assumption~\ref{ass:projected-classical}.  The following assumptions give one sufficient route to make these constants uniform in $\varepsilon$ and $K$.  

For $\nu\in\cP_2(\cH_K)$, write
\[
    U_\varepsilon^K(t,\nu)(x):=\partial_\mu V_\varepsilon^K(t,\nu)(x).
\]
For the projected terminal functional $G_K$, define
\[
    A_\nu^K(x):=\partial_x\partial_\mu G_K(\nu)(x),
    \qquad
    B_\nu^K(x,y):=\partial_{\mu\mu}^2G_K(\nu)(x)(y).
\]

\begin{ass}[Projected FBSDE regularity and tangent equations]
\label{ass:projected-fbsde-tangent}
Fix $\varepsilon>0$ and $K\ge1$.  Assume that $U_\varepsilon^K$ is continuous on its domain and, for $t<1$, has continuous spatial derivative $\partial_xU_\varepsilon^K=\partial_x\partial_\mu V_\varepsilon^K$ and continuous Lions derivative $\partial_\mu U_\varepsilon^K=\partial_{\mu\mu}^2V_\varepsilon^K$.  On bounded subsets of $\cP_2(\cH_K)$ and $\cH_K$, these derivatives are locally Lipschitz and have at most polynomial growth.

Assume further that, for every $(t,\nu)$ with $\xi\sim\nu$, the closed-loop SDE
\[
    dX_s^{t,\xi}=-\frac1\varepsilon Q_KU_\varepsilon^K(s,\nu_s)(X_s^{t,\xi})\,ds+dW_s^K,
    \qquad
    X_t^{t,\xi}=\xi,
    \qquad
    \nu_s=\Law(X_s^{t,\xi}),
\]
has a unique square-integrable solution, and that
\[
    Y_s^{t,\xi}:=U_\varepsilon^K(s,\nu_s)(X_s^{t,\xi}),\quad Z_s^{t,\xi}:=\partial_x U^K_\varepsilon(s,\nu_s)(X_s^{t,\xi})
\]
satisfies the decoupled FBSDE
\begin{equation}
\label{eq:projected-decoupled-fbsde}
\begin{cases}
    dX_s^{t,\xi}=-\dfrac1\varepsilon Q_KY_s^{t,\xi}\,ds+dW_s^K,\qquad X_t^{t,\xi}=\xi,\\[0.4em]
    dY_s^{t,\xi}=Z_s^{t,\xi}\,dW_s^K,\\[0.4em]
    Y_1^{t,\xi}=\partial_\mu G_K(\nu_1)(X_1^{t,\xi}).
\end{cases}
\end{equation}

Also, assume that the following $L^2$ tangent equations are valid.  If $\xi^r=\xi+r\phi(\xi)$ with $\phi\in L^2(\cH_K,\nu;\cH_K)$, then the first law variation
\[
    \delta X_s:=\left.\frac{d}{dr}X_s^{t,\xi^r}\right|_{r=0},
    \qquad
    \delta Y_s:=\left.\frac{d}{dr}Y_s^{t,\xi^r}\right|_{r=0},
    \qquad
    \delta Z_s:=\left.\frac{d}{dr}Z_s^{t,\xi^r}\right|_{r=0}
\]
exists in $L^2$ and satisfies
\begin{equation}
\label{eq:projected-first-law-tangent}
\begin{cases}
    d(\delta X_s)=-\dfrac1\varepsilon Q_K\delta Y_s\,ds,
    \qquad \delta X_t=\phi(\xi),\\[0.4em]
    d(\delta Y_s)=\delta Z_s\,dW_s^K,\\[0.4em]
    \delta Y_1=A_{\nu_1}^K(X_1^{t,\xi})\delta X_1
    +\widetilde\E\!
    \left[B_{\nu_1}^K(X_1^{t,\xi},\widetilde X_1^{t,\xi})\delta\widetilde X_1\right].
\end{cases}
\end{equation}
Finally for fixed deterministic $x\in\cH_K$, let $X^{t,x}$ be the decoupled process started from $x$ and driven by the same law flow $(\nu_s)$.  Its measure-direction variation in the direction $\phi$, denoted by $(\delta_\phi X,\delta_\phi Y,\delta_\phi Z)$, exists in $L^2$ and satisfies
\begin{equation}
\label{eq:projected-measure-tangent}
\begin{cases}
    d(\delta_\phi X_s)=-\dfrac1\varepsilon Q_K\delta_\phi Y_s\,ds,
    \qquad \delta_\phi X_t=0,\\[0.4em]
    d(\delta_\phi Y_s)=\delta_\phi Z_s\,dW_s^K,\\[0.4em]
    \delta_\phi Y_1=A_{\nu_1}^K(X_1^{t,x})\delta_\phi X_1
    +\widetilde\E\!
    \left[B_{\nu_1}^K(X_1^{t,x},\widetilde X_1^{t,\xi})\delta\widetilde X_1\right].
\end{cases}
\end{equation}
All stochastic integrals appearing above are true martingales after localization and passage to the limit.  These regularity assumptions are finite-dimensional conditions on $\cH_K$; they are stated explicitly here only to justify the FBSDE differentiations used below.
\end{ass}

\begin{ass}[Projected terminal monotonicity and coercivity]
\label{ass:projected-terminal-coercivity}
There exist constants $K_G<\infty$ and $\kappa>0$, independent of $(\varepsilon,K)$, such that
\begin{equation}
\label{eq:projected-AB-bound}
    \|A_\nu^K(x)\|_{\op}+\|B_\nu^K(x,y)\|_{\op}\le K_G,
    \qquad \nu\in\cP_2(\cH_K),\ x,y\in\cH_K,
\end{equation}
\begin{align}
\label{eq:mixed_law_coercivity}
&\mathbb E\left[
\left\langle
\delta_\phi X_1,
A^K_{\nu_1}(X_1^{t,x})\delta_\phi X_1
+
\widetilde{\mathbb E}
\left[
B^K_{\nu_1}(X_1^{t,x},\widetilde X_1^{t,\xi})\delta\widetilde X_1
\right]
\right\rangle
\right]
\nonumber\\
&\qquad\qquad\ge
\kappa\mathbb E\|\delta_\phi X_1\|^2
-
K_G\mathbb E\|\delta_\phi X_1\|\,\mathbb E\|\delta X_1\|,
\end{align}
and, for every square-integrable pair $(X,\eta)$ with $\Law(X)=\nu$,
\begin{equation}
\label{eq:projected-averaged-Q-monotonicity}
\E\left[
\frac{\1_{\{\|\eta\|>0\}}}{\|\eta\|}
\left\langle
\eta,
Q_KA_\nu^K(X)\eta
+
Q_K\widetilde\E\left[B_\nu^K(X,\widetilde X)\widetilde\eta\right]
\right\rangle
\right]
\ge0,
\end{equation}
where $(\widetilde X,\widetilde\eta)$ is an independent copy of $(X,\eta)$.
\end{ass}

\begin{prop}[Projected second Lions derivative estimate]
\label{prop:uniform-L-projected}
Assume Assumptions~\ref{ass:projected-fbsde-tangent} and~\ref{ass:projected-terminal-coercivity}.  Then
\begin{equation}
\label{eq:projected-mumu-bound}
    \|\partial_{\mu\mu}^2V_\varepsilon^K(t,\nu)(x)(y)\|_{\op}
    \le
    \frac{K_G^2}{\kappa}+K_G,
    \qquad \nu\text{-a.e. }y\in\cH_K.
\end{equation}
Consequently,
\begin{equation}
\label{eq:projected-L-uniform}
    L_{\varepsilon,K}
    \le
    \Tr(Q_K)\left(\frac{K_G^2}{\kappa}+K_G\right)
    \le
    \Tr(Q)\left(\frac{K_G^2}{\kappa}+K_G\right).
\end{equation}
The bound is independent of $\varepsilon$ and $K$.
\end{prop}

\begin{proof}
We prove only the estimate needed for $L_{\varepsilon,K}$.

\paragraph{Step 1: contraction of the first law variation.}
Let $(\delta X,\delta Y,\delta Z)$ solve \Eqref{eq:projected-first-law-tangent}.  Applying It\^o's formula to $\langle\delta X_s,Q_K\delta Y_s\rangle$ gives
\[
    d\langle\delta X_s,Q_K\delta Y_s\rangle
    =
    -\frac1\varepsilon\|Q_K\delta Y_s\|^2\,ds
    +
    \langle\delta X_s,Q_K\delta Z_s\,dW_s^K\rangle.
\]
Using the regularized norm $(\|z\|^2+\eta^2)^{1/2}$ and then sending $\eta\downarrow0$, the product rule and Cauchy--Schwarz show that
\[
    \frac{\1_{\{\|\delta X_s\|>0\}}}{\|\delta X_s\|}
    \langle\delta X_s,Q_K\delta Y_s\rangle
\]
is a supermartingale after localization.  The terminal condition and \Eqref{eq:projected-averaged-Q-monotonicity} imply that its terminal expectation is nonnegative.  Hence
\[
    \E\left[
    \frac{\1_{\{\|\delta X_s\|>0\}}}{\|\delta X_s\|}
    \langle\delta X_s,Q_K\delta Y_s\rangle
    \right]
    \ge0,
    \qquad s\in[t,1].
\]
Since
\[
    d\|\delta X_s\|
    =
    -\frac{\1_{\{\|\delta X_s\|>0\}}}{\varepsilon\|\delta X_s\|}
    \langle\delta X_s,Q_K\delta Y_s\rangle\,ds,
\]
we obtain
\begin{equation}
\label{eq:projected-first-law-contraction}
    \E\|\delta X_1\|\le \E\|\phi(\xi)\|.
\end{equation}

\paragraph{Step 2: measure-direction stability.}
Let $(\delta_\phi X,\delta_\phi Y,\delta_\phi Z)$ solve \Eqref{eq:projected-measure-tangent}.  It\^o's formula for $\langle\delta_\phi X_s,\delta_\phi Y_s\rangle$ gives
\[
    \E\langle\delta_\phi X_1,\delta_\phi Y_1\rangle
    =
    -\frac1\varepsilon\E\int_t^1\langle Q_K\delta_\phi Y_s,\delta_\phi Y_s\rangle\,ds
    \le0.
\]
Using the terminal condition in \Eqref{eq:projected-measure-tangent}, the mixed coercivity condition \Eqref{eq:mixed_law_coercivity}, and \Eqref{eq:projected-first-law-contraction},
\[
    0
    \ge
    \E\langle\delta_\phi X_1,\delta_\phi Y_1\rangle
    \ge
    \kappa\E\|\delta_\phi X_1\|^2
    -
    K_G\E\|\delta_\phi X_1\|\,\E\|\phi(\xi)\|.
\]
Therefore
\begin{equation}
\label{eq:projected-measure-stability}
    \E\|\delta_\phi X_1\|
    \le
    \frac{K_G}{\kappa}\E\|\phi(\xi)\|.
\end{equation}

\paragraph{Step 3: second Lions derivative.}
Differentiating the decoupling relation
\[
    Y_t^{t,x}=\partial_\mu V_\varepsilon^K(t,\nu)(x)
\]
in the measure direction $\phi$ gives
\[
    \left\|
    \E\left[
    \partial_{\mu\mu}^2V_\varepsilon^K(t,\nu)(x)(\xi)\phi(\xi)
    \right]
    \right\|
    =
    \|\E[\delta_\phi Y_t]\|.
\]
Since $\delta_\phi Y$ is a martingale,
\[
\begin{aligned}
    \|\E[\delta_\phi Y_t]\|
    &\le
    K_G\E\|\delta_\phi X_1\|+K_G\E\|\delta X_1\| \\
    &\le
    \left(\frac{K_G^2}{\kappa}+K_G\right)\E\|\phi(\xi)\|.
\end{aligned}
\]
The standard duality and measurable-selection argument yields
\Eqref{eq:projected-mumu-bound} for every fixed \(x\in H_K\) and
\(\nu\)-a.e. \(y\in H_K\). By the continuity of
\(\partial_{\mu\mu}^2V_\varepsilon^K\) assumed in
Assumption \ref{ass:projected-fbsde-tangent}, this estimate extends, for every fixed \(x\),
to all \(y\in\operatorname{supp}\nu\). Hence, taking \(x=y\), we obtain
\[
\left\|
\partial_{\mu\mu}^2V_\varepsilon^K(t,\nu)(y)(y)
\right\|_{\mathrm{op}}
\le
\frac{K_G^2}{\kappa}+K_G,
\qquad
y\in\operatorname{supp}\nu.
\]
Taking the \(Q_K\)-trace and integrating with respect to \(\nu\)
therefore gives \Eqref{eq:projected-L-uniform}.
\end{proof}

\subsection{Examples}

\begin{exam}[Coercive quadratic target]
Let
\[
    G_{\mathrm{quad}}(\mu)
    :=
    \frac\kappa2\int_{\cH}\|x-a\|^2\mu(dx),
    \qquad \kappa>0,
    \quad a\in\cH.
\]
Then $\mathcal V_0=0$ and the unique minimizer is $\mu^*=\delta_a$.  The derivatives are
\[
    \partial_\mu G_{\mathrm{quad}}(\mu)(x)=\kappa(x-a),
    \qquad
    \partial_x\partial_\mu G_{\mathrm{quad}}(\mu)(x)=\kappa I,
    \qquad
    \partial_{\mu\mu}^2G_{\mathrm{quad}}=0.
\]
Thus the optional uniform-$L_{\varepsilon,K}$ condition holds.  In fact, the value function is linear in the measure variable, so the sharper identity $L_{\varepsilon,K}=0$ holds.  For this example,
\[
    \Delta_K^{\mathrm{proj}}=\|(I-\Pi_K)a\|^2,
    \qquad
    D_{CM,K}^2(\delta_a)=\sum_{j=1}^K\frac{(a_j-x_{0,j})^2}{\lambda_j}.
\]
Therefore Theorem~\ref{thm:projection-first-global} gives, with $\alpha=1$ on bounded moment sets,
\[
0
\le
\frac1N v_\varepsilon^{K,N}(0,\Pi_K\mathbf x_0)
\le
C_*
\left(\varepsilon^2\Tr(Q)+\|(I-\Pi_K)a\|^2\right)^{1/2}
+
\frac\varepsilon2\sum_{j=1}^K\frac{(a_j-x_{0,j})^2}{\lambda_j}
+
\varepsilon K\log\frac1\varepsilon.
\]
\end{exam}
\section{Convergence analysis for nonmatched covariance problems}
\label{app:convergence}

\subsection{Finite-candidate consistency with nonmatched execution covariance}
\label{app:implementation-consistency}

We now connect the practical SCMO update to the matched projected particle value analyzed in Section~\ref{matched_problem}. We analyze one outer-loop rollout, consisting of $M$ updates from its current cloud. When \(L>1\), the same per-rollout analysis applies at each outer iteration; The analysis proceeds in two
stages. First, we establish consistency of the \(R\)-context,
\(S\)-candidate Monte Carlo approximation of the Cole--Hopf feedback, thereby
relating the fully finite implementation to an exact-feedback 
nonmatched-covariance chain. Second, we compare this exact-feedback chain, which is driven by the matched feedback
but executed with a covariance different from the proposal covariance,
with the matched projected particle value. The resulting covariance
contribution is retained with its sign and decomposed into a nonnegative
favorable gain and a nonnegative unfavorable contribution. In the smooth
case, its leading-order one-step effect is governed by $\frac h2
\operatorname{Tr}\left(
\left[
\mathbf Q_{K,\mathrm{dyn}}
-
\mathbf Q_{K,\mathrm{prop}}
\right]
D^2v_\varepsilon^{K,N}
\right)$. In particular, when
$\mathbf Q_{K,\mathrm{dyn}}
=
\alpha_0\mathbf Q_{K,\mathrm{prop}}$, a colder execution covariance, \(\alpha_0<1\), produces a favorable
contribution in regions of positive covariance-weighted continuation-value
curvature, whereas a hotter execution covariance, \(\alpha_0>1\), is
favorable in regions of negative weighted curvature. If the favorable
curvature regions dominate along the states visited by the chain, the
resulting gain exceeds the unfavorable covariance contribution and tightens
the matched upper bound. If this gain also dominates the Euler
discretization residual, the complete one-outer-loop bound is strictly
smaller than the ideal matched particle--projection--regularization bound.
Thus, nonmatched execution is not treated solely as an additional
error: its interaction with the curvature of the continuation value can
produce a genuine improvement in the theoretical certificate.

We recall the notation. For
\(\mathbf y=(y_1,\ldots,y_N)\in(\mathcal{H}_K)^N\), write
$
\mu_{\mathbf y}^N
=
\frac1N\sum_{i=1}^N\delta_{y_i}$, $F_{K,N}(\mathbf y)
:=
N\,G_K(\mu_{\mathbf y}^N)$. Thus \(F_{K,N}\) is the extensive terminal cost appearing in the projected
\(N\)-particle control problem. Write the full proposal covariance operator on \(\mathcal{H}\) by $Q_{\mathrm{prop}}
:=
\sigma_{\mathrm{prop}}^2 Q$, and let \(Q_{K,\mathrm{prop}}:=
\Pi_K Q_{\mathrm{prop}}\Pi_K\) be the projected covariance used in the
Cole--Hopf control problem and in the context and candidate proposals. Also write $Q_{\mathrm{dyn}}
:=
\sigma_{\mathrm{dyn}}^2 Q$ and let \(Q_{K,\mathrm{dyn}}:=\Pi_K Q_{\mathrm{dyn}}\Pi_K\) be the covariance used in the Euler
execution step. In the coefficient implementation, $Q_{K,\mathrm{prop}}
=
\sigma_{\mathrm{prop}}^2\Lambda$, $Q_{K,\mathrm{dyn}}
=
\sigma_{\mathrm{dyn}}^2\Lambda$. Since the two covariances have the same spectrum up to scale, define the
fixed variance ratio
$
\alpha_0
:=
\frac{\sigma_{\mathrm{dyn}}^2}
{\sigma_{\mathrm{prop}}^2},
\,
Q_{K,\mathrm{dyn}}
=
\alpha_0 Q_{K,\mathrm{prop}},
\,
\alpha_0\neq1$. Their \(N\)-particle block extensions are $
\mathbf Q_{K,\mathrm{prop}}
=
I_N\otimes Q_{K,\mathrm{prop}}$, $\mathbf Q_{K,\mathrm{dyn}}
=
I_N\otimes Q_{K,\mathrm{dyn}}
=
\alpha_0\mathbf Q_{K,\mathrm{prop}}$.

Throughout this subsection,
\(v_\varepsilon^{K,N}\) and \(b\) denote the matched projected particle
value and feedback associated with \(Q_{K,\mathrm{prop}}\). 

Fix \(K,N\), \(\varepsilon>0\), and \(t<1\). Let
$
\tau
=
1-t,
\,
\mathbf U
=
(U_1,\ldots,U_N)
\sim
\mathcal N(0,\tau\mathbf Q_{K,\mathrm{prop}})$, and define $L_{t,\mathbf x}(\mathbf U)
=
\exp\left(
-\frac{
F_{K,N}(\mathbf x+\mathbf U)
}{\varepsilon}
\right)$. The exact matched Cole--Hopf feedback is
\begin{equation}
\label{eq:exact-particle-feedback-fixed}
b_i(t,\mathbf x)
=
\frac1\tau
\frac{
\mathbb E\left[
L_{t,\mathbf x}(\mathbf U)U_i
\right]
}{
\mathbb E\left[
L_{t,\mathbf x}(\mathbf U)
\right]
}.
\end{equation}

\subsubsection{Multi-context finite-candidate feedback}

For each \(r=1,\ldots,R\), draw one complete context increment $\mathbf U^{(r)}
=
(U_1^{(r)},\ldots,U_N^{(r)})
\sim
\mathcal N
\left(
0,\tau\mathbf Q_{K,\mathrm{prop}}
\right)$. Conditionally on each context, draw independent candidate increments $
U_i^{(r,s)}
\sim
\mathcal N(0,\tau Q_{K,\mathrm{prop}}),\,s=1,\ldots,S$. All contexts and candidates are independent except for the shared context
within a fixed block \(r\). For the feedback of particle \(i\), use $
U_{-i}^{(r)}
=
(U_j^{(r)})_{j\neq i}$, and replace the \(i\)-th component by an independent candidate
\(U_i^{(r,s)}\).Define the complete increment
$
\mathbf U^{(i,r,s)}
=
\left(
U_1^{(r)},\ldots,U_{i-1}^{(r)},
U_i^{(r,s)},
U_{i+1}^{(r)},\ldots,U_N^{(r)}
\right)
\in(H_K)^N$ and
$
L_{i,r,s}
=
L_{t,\mathbf x}\bigl(\mathbf U^{(i,r,s)}\bigr)
=
\exp\left[
-\frac{
F_{K,N}
\left(
\mathbf x+\mathbf U^{(i,r,s)}
\right)
}{\varepsilon}
\right]$. Equivalently, \(F_{K,N}\bigl(\mathbf x+\mathbf U^{(i,r,s)}\bigr)\)
is the objective of the complete cloud obtained by perturbing every
coordinate \(j\neq i\) with the \(r\)-th context increment and replacing
the \(i\)-th coordinate by its \((r,s)\)-candidate increment. Define $\widehat b_i^{R,S}(t,\mathbf x)
=
\frac1\tau
\frac{
\sum_{r=1}^R\sum_{s=1}^S
L_{i,r,s}U_i^{(r,s)}
}{
\sum_{r=1}^R\sum_{s=1}^S
L_{i,r,s}
}$.

\begin{prop}[Consistency of the multi-context feedback estimator]
\label{prop:multicontext-consistency-fixed}
Assume
$\mathbb E\left[
L_{t,\mathbf x}(\mathbf U)
\|U_i\|
\right]
<
\infty$, $0<\mathbb E\left[
L_{t,\mathbf x}(\mathbf U)
\right]
<\infty$. Then, for every fixed \(S\geq1\),
\[
\widehat b_i^{R,S}(t,\mathbf x)
\longrightarrow
b_i(t,\mathbf x)
\qquad
\text{almost surely as }R\to\infty.
\]

For a quantitative statement, assume in addition that $\mathbb E\left[
L_{t,\mathbf x}(\mathbf U)^4
\left(
1+\|U_i\|^4
\right)
\right]
<
\infty$. Writing $B_{i,r}^S
:=
\frac1S
\sum_{s=1}^S L_{i,r,s}$,
$\overline B_i^{R,S}
:=
\frac1R
\sum_{r=1}^R B_{i,r}^S$, assume that the self-normalized denominator is \(L^2\)-stable in the
explicit sense that
\begin{equation}
\label{eq:inverse-denominator-condition-fixed}
\sup_{R,S\geq1}
\mathbb E\left[
\left(
\overline B_i^{R,S}
\right)^{-4}
\right]
<
\infty.
\end{equation}
A stronger but simpler sufficient condition is a deterministic lower bound
\(\overline B_i^{R,S}\geq c>0\) on the time--state set under
consideration. Under these assumptions,
\[
\mathbb E
\left[
\left\|
\widehat b_i^{R,S}(t,\mathbf x)
-
b_i(t,\mathbf x)
\right\|^2
\right]
\leq
C
\left(
\frac1R+\frac1{RS}
\right).
\]
The constant may depend on \(K,N,\varepsilon,t,\mathbf x\), but is
independent of \(R\) and \(S\). Consequently, under the quantitative
assumptions, convergence also holds in \(L^2\), and hence in probability,
along every sequence satisfying
\[
R_n\to\infty,
\qquad
R_nS_n\to\infty.
\]
\end{prop}

\begin{proof}
Define the context-block averages $A_{i,r}^S
=
\frac1S
\sum_{s=1}^S
L_{i,r,s}U_i^{(r,s)}$, $B_{i,r}^S
=
\frac1S
\sum_{s=1}^S
L_{i,r,s}$. Then $
\widehat b_i^{R,S}(t,\mathbf x)
=
\frac1\tau
\frac{
\overline A_i^{R,S}
}{
\overline B_i^{R,S}
},
\,
\overline A_i^{R,S}
:=
\frac1R
\sum_{r=1}^R A_{i,r}^S$. The pairs $(A_{i,r}^S,B_{i,r}^S),
\,
r=1,\ldots,R$, are independent and identically distributed. Since each complete
context--candidate pair has the same law as \(\mathbf U\),
$
a_i
:=
\mathbb E[A_{i,r}^S]
=
\mathbb E\left[
L_{t,\mathbf x}(\mathbf U)U_i
\right],
\,
\beta
:=
\mathbb E[B_{i,r}^S]
=
\mathbb E\left[
L_{t,\mathbf x}(\mathbf U)
\right]
>
0$. The strong law of large numbers gives $\overline A_i^{R,S}
\longrightarrow
a_i,
\,
\overline B_i^{R,S}
\longrightarrow
\beta$ almost surely as \(R\to\infty\), for every fixed \(S\geq1\). The
continuous-mapping theorem therefore yields $\widehat b_i^{R,S}(t,\mathbf x)
\longrightarrow
\frac1\tau\frac{a_i}{\beta}
=
b_i(t,\mathbf x)
$ almost surely.

We now prove the quantitative estimate. For an \(\mathcal{H}_K\)-valued
square-integrable random variable \(Y\), use the convention $
\operatorname{Var}(Y)
:=
\mathbb E\|Y-\mathbb EY\|^2$. Let $\mathcal G_r
:=
\sigma(U_{-i}^{(r)})$. Conditionally on \(\mathcal G_r\), the variables
$
L_{i,r,s}U_i^{(r,s)},
\,
s=1,\ldots,S$, are independent and identically distributed. The law of total variance
therefore gives
\[
\begin{aligned}
\operatorname{Var}(A_{i,r}^S)
={}&
\operatorname{Var}\left(
\mathbb E[
L_{t,\mathbf x}(\mathbf U)U_i
\mid U_{-i}]
\right)
+
\frac1S
\mathbb E\left[
\operatorname{Var}\left(
L_{t,\mathbf x}(\mathbf U)U_i
\mid U_{-i}
\right)
\right],
\end{aligned}
\]
and similarly
\[
\begin{aligned}
\operatorname{Var}(B_{i,r}^S)
={}&
\operatorname{Var}\left(
\mathbb E[
L_{t,\mathbf x}(\mathbf U)
\mid U_{-i}]
\right)
+
\frac1S
\mathbb E\left[
\operatorname{Var}\left(
L_{t,\mathbf x}(\mathbf U)
\mid U_{-i}
\right)
\right].
\end{aligned}
\]
Since the blocks are independent across \(r\),
\[
\mathbb E
\left\|
\overline A_i^{R,S}-a_i
\right\|^2
\leq
C_A
\left(
\frac1R+\frac1{RS}
\right),
\]
and
\[
\mathbb E
\left|
\overline B_i^{R,S}-\beta
\right|^2
\leq
C_B
\left(
\frac1R+\frac1{RS}
\right).
\]
The fourth-moment assumption, conditional Rosenthal inequalities within
each context, and the standard fourth-moment estimate for averages of
independent blocks also give
\[
\left(
\mathbb E
\left\|
\overline A_i^{R,S}-a_i
\right\|^4
\right)^{1/2}
\leq
C_A'
\left(
\frac1R+\frac1{RS}
\right)
\]
and
\[
\left(
\mathbb E
\left|
\overline B_i^{R,S}-\beta
\right|^4
\right)^{1/2}
\leq
C_B'
\left(
\frac1R+\frac1{RS}
\right).
\]

Using
\[
\frac{\overline A_i^{R,S}}{\overline B_i^{R,S}}
-
\frac{a_i}{\beta}
=
\frac1{\overline B_i^{R,S}}
\left[
\left(
\overline A_i^{R,S}-a_i
\right)
-
\frac{a_i}{\beta}
\left(
\overline B_i^{R,S}-\beta
\right)
\right],
\]
we obtain
\[
\begin{aligned}
&
\mathbb E
\left\|
\widehat b_i^{R,S}(t,\mathbf x)
-
b_i(t,\mathbf x)
\right\|^2
\leq
\frac{2}{\tau^2}
\mathbb E\left[
\frac{
\left\|
\overline A_i^{R,S}-a_i
\right\|^2
}{
\left(
\overline B_i^{R,S}
\right)^2
}
\right]
+
\frac{2\|a_i\|^2}{\tau^2\beta^2}
\mathbb E\left[
\frac{
\left|
\overline B_i^{R,S}-\beta
\right|^2
}{
\left(
\overline B_i^{R,S}
\right)^2
}
\right].
\end{aligned}
\]
By Cauchy--Schwarz and
~\Eqref{eq:inverse-denominator-condition-fixed},
\[
\begin{aligned}
\mathbb E\left[
\frac{
\left\|
\overline A_i^{R,S}-a_i
\right\|^2
}{
\left(
\overline B_i^{R,S}
\right)^2
}
\right]
&\leq
\left(
\mathbb E
\left\|
\overline A_i^{R,S}-a_i
\right\|^4
\right)^{1/2}
\left(
\mathbb E
\left[
\left(
\overline B_i^{R,S}
\right)^{-4}
\right]
\right)^{1/2},
\end{aligned}
\]
and the same estimate applies to the denominator fluctuation. Substituting
the fourth-moment bounds proves
\[
\mathbb E
\left[
\left\|
\widehat b_i^{R,S}(t,\mathbf x)
-
b_i(t,\mathbf x)
\right\|^2
\right]
\leq
C
\left(
\frac1R+\frac1{RS}
\right).
\]
\end{proof}

\subsubsection{exact-feedback nonmatched-covariance chain and one outer-loop rollout value bound}

Let
\[
h
=
\frac1M,
\qquad
t_m
=
mh.
\]
For a deterministic state \(\mathbf x\), define the drifted state
\[
\overline{\mathbf x}_h
=
\mathbf x+h\,b(t,\mathbf x).
\]
Define the matched and execution continuation scores by
\[
\Phi_h^{v,\mathrm{prop}}(t,\mathbf x)
=
\mathbb E\left[
v_\varepsilon^{K,N}
\left(
t+h,
\overline{\mathbf x}_h
+
\sqrt h\,
\mathbf Q_{K,\mathrm{prop}}^{1/2}Z
\right)
\right],
\]
and
\[
\Phi_h^{v,\mathrm{dyn}}(t,\mathbf x)
=
\mathbb E\left[
v_\varepsilon^{K,N}
\left(
t+h,
\overline{\mathbf x}_h
+
\sqrt h\,
\mathbf Q_{K,\mathrm{dyn}}^{1/2}Z
\right)
\right],
\qquad
Z\sim\mathcal N(0,I_{NK}).
\]
The one-step covariance effect is the signed quantity
\begin{equation}
\label{eq:fixed-covariance-one-step-effect}
c_h^{K,N,\varepsilon}(t,\mathbf x)
=
\Phi_h^{v,\mathrm{dyn}}(t,\mathbf x)
-
\Phi_h^{v,\mathrm{prop}}(t,\mathbf x).
\end{equation}
To display explicitly when the nonmatched covariance improves the matched
continuation score, define the nonnegative one-step gain and unfavorable
contribution
\begin{equation}
\label{eq:fixed-covariance-gain-loss}
\begin{aligned}
g_{h}^{K,N,\varepsilon}(t,\mathbf x)
&:=
\left(
-c_h^{K,N,\varepsilon}(t,\mathbf x)
\right)_+
=
\left(
\Phi_h^{v,\mathrm{prop}}(t,\mathbf x)
-
\Phi_h^{v,\mathrm{dyn}}(t,\mathbf x)
\right)_+,
\\
p_{h}^{K,N,\varepsilon}(t,\mathbf x)
&:=
\left(
c_h^{K,N,\varepsilon}(t,\mathbf x)
\right)_+
=
\left(
\Phi_h^{v,\mathrm{dyn}}(t,\mathbf x)
-
\Phi_h^{v,\mathrm{prop}}(t,\mathbf x)
\right)_+.
\end{aligned}
\end{equation}
Then
\begin{equation}
\label{eq:fixed-covariance-effect-decomposition}
c_h^{K,N,\varepsilon}
=
p_{h}^{K,N,\varepsilon}
-
g_{h}^{K,N,\varepsilon}.
\end{equation}
Thus \(g_{h}>0\) precisely when the execution covariance
produces a smaller one-step continuation score than matched execution, while
\(p_{h}>0\) records the opposite event.

Define the matched running cost
\begin{equation}
\label{ell}
    \ell_\varepsilon^{K,N}(t,\mathbf x)
=
\frac{\varepsilon}{2}
\left\|
b(t,\mathbf x)
\right\|_{\mathbf Q_{K,\mathrm{prop}}^{-1}}^2.
\end{equation}
The matched one-step weak Euler residual is
\begin{equation}
\label{eq:matched-euler-residual-fixed}
r_h^{K,N,\varepsilon}(t,\mathbf x)
=
\Phi_h^{v,\mathrm{prop}}(t,\mathbf x)
+
h\ell_\varepsilon^{K,N}(t,\mathbf x)
-
v_\varepsilon^{K,N}(t,\mathbf x).
\end{equation}

Define the exact-feedback nonmatched-covariance chain by
\[
\mathbf X_0
=
\Pi_K\mathbf x_0,
\qquad
\overline{\mathbf X}_{m+1}
=
\mathbf X_m
+
h\,b(t_m,\mathbf X_m),
\]
and
\begin{equation}
\label{eq:exact-fixed-covariance-chain}
\mathbf X_{m+1}
=
\overline{\mathbf X}_{m+1}
+
\sqrt h\,
\mathbf Q_{K,\mathrm{dyn}}^{1/2}Z_{m+1},
\qquad
Z_{m+1}\sim\mathcal N(0,I_{NK}),
\end{equation}
with independent increments. Set
\[
\mathcal E_h^{K,N,\varepsilon}
=
\frac1N
\sum_{m=0}^{M-1}
\mathbb E\left[
\left(
r_h^{K,N,\varepsilon}(t_m,\mathbf X_m)
\right)^+
\right],
\]
\[
\mathcal G_{h}^{K,N,\varepsilon}
=
\frac1N
\sum_{m=0}^{M-1}
\mathbb E\left[
g_{h}^{K,N,\varepsilon}(t_m,\mathbf X_m)
\right],
\]
and
\[
\mathcal P_{h}^{K,N,\varepsilon}
=
\frac1N
\sum_{m=0}^{M-1}
\mathbb E\left[
p_{h}^{K,N,\varepsilon}(t_m,\mathbf X_m)
\right].
\]
The cumulative signed covariance effect is therefore
\begin{equation}
\label{eq:cumulative-fixed-covariance-effect}
\mathcal C_{h}^{K,N,\varepsilon}
:=
\frac1N
\sum_{m=0}^{M-1}
\mathbb E\left[
c_h^{K,N,\varepsilon}(t_m,\mathbf X_m)
\right]
=
-\mathcal G_{h}^{K,N,\varepsilon}
+
\mathcal P_{h}^{K,N,\varepsilon}.
\end{equation}

For later use, write the existing particle--projection--regularization bound,
with the proposal covariance used in the matched problem, as
\begin{equation}
\label{eq:R-ideal-fixed}
\begin{aligned}
\mathcal R_{K,N,\varepsilon}^{\mathrm{prop}}
:={}&
\frac{L_{\varepsilon,K}}{2N}
+
C^\star
\left(
\varepsilon^{2/\alpha}
\operatorname{Tr}(Q_{\mathrm{prop}})
+
\Delta_K^{\mathrm{proj}}
\right)^{\alpha/2}+
\frac{\varepsilon}{2}
D_{\mathrm{CM},K,\mathrm{prop}}^2(\mu^\star)
+
\frac{\varepsilon K}{\alpha}
\log\frac1\varepsilon.
\end{aligned}
\end{equation}
Here, if
\[
Q_{\mathrm{prop}}e_j
=
\lambda_{j,\mathrm{prop}}e_j,
\]
then
\[
D_{\mathrm{CM},K,\mathrm{prop}}^2(\mu^\star)
=
\sum_{j=1}^K
\frac{
\mathbb E_{\mu^\star}
\left[
(X_j-x_{0,j})^2
\right]
}{
\lambda_{j,\mathrm{prop}}
},
\qquad
X\sim\mu^\star.
\]

\begin{thm}
\label{thm:fixed-covariance-bound}
Assume the hypotheses of
Theorem~\ref{thm:projection-first-global}, with \(Q_{\mathrm{prop}}\) as the
covariance of the matched control problem, and assume that all one-step
quantities above are integrable. Then
\begin{equation}
\label{eq:fixed-covariance-bound-with-gain}
\begin{aligned}
\mathbb E
G_K(\mu_{\mathbf X_M}^N)
-
\mathcal V_0
\leq{}&
\mathcal R_{K,N,\varepsilon}^{\mathrm{prop}}
+
\mathcal E_h^{K,N,\varepsilon}
-
\mathcal G_{h}^{K,N,\varepsilon}
+
\mathcal P_{h}^{K,N,\varepsilon}.
\end{aligned}
\end{equation}
In particular, the nonmatched covariance gives a negative covariance correction to \(\mathcal R^{\rm prop}+\mathcal E_h\) whenever
\[
\mathcal G_{h}^{K,N,\varepsilon}
>
\mathcal P_{h}^{K,N,\varepsilon}.
\]
If, more strongly,
\[
\mathcal G_{h}^{K,N,\varepsilon}
>
\mathcal E_h^{K,N,\varepsilon}
+
\mathcal P_{h}^{K,N,\varepsilon},
\]
then the complete one outer-loop rollout upper bound is smaller than the ideal matched
particle--projection--regularization bound
\(\mathcal R_{K,N,\varepsilon}^{\mathrm{prop}}\).
\end{thm}

\begin{proof}
Conditionally on \(\mathbf X_m\),
\[
\mathbb E\left[
v_\varepsilon^{K,N}
(t_{m+1},\mathbf X_{m+1})
\mid
\mathcal F_m
\right]
=
\Phi_h^{v,\mathrm{dyn}}(t_m,\mathbf X_m).
\]
By~\Eqref{eq:fixed-covariance-one-step-effect} and
\Eqref{eq:fixed-covariance-effect-decomposition},
\[
\Phi_h^{v,\mathrm{dyn}}(t_m,\mathbf X_m)
=
\Phi_h^{v,\mathrm{prop}}(t_m,\mathbf X_m)
-
g_{h}^{K,N,\varepsilon}(t_m,\mathbf X_m)
+
p_{h}^{K,N,\varepsilon}(t_m,\mathbf X_m).
\]
Using
~\Eqref{eq:matched-euler-residual-fixed},
\[
\begin{aligned}
&
\mathbb E\left[
v_\varepsilon^{K,N}
(t_{m+1},\mathbf X_{m+1})
\mid
\mathcal F_m
\right]
+
h\ell_\varepsilon^{K,N}(t_m,\mathbf X_m)
\\
&\qquad\leq
v_\varepsilon^{K,N}(t_m,\mathbf X_m)
+
\left(
r_h^{K,N,\varepsilon}(t_m,\mathbf X_m)
\right)^+
-
g_{h}^{K,N,\varepsilon}(t_m,\mathbf X_m)
+
p_{h}^{K,N,\varepsilon}(t_m,\mathbf X_m).
\end{aligned}
\]
Taking expectations and summing over \(m=0,\ldots,M-1\) yields
\[
\begin{aligned}
&
\mathbb E
v_\varepsilon^{K,N}(1,\mathbf X_M)
+
\sum_{m=0}^{M-1}
h\,
\mathbb E
\ell_\varepsilon^{K,N}(t_m,\mathbf X_m)
\\
&\qquad\leq
v_\varepsilon^{K,N}
\bigl(0,\Pi_K\mathbf x_0\bigr)
+
\sum_{m=0}^{M-1}
\mathbb E\left[
\left(
r_h^{K,N,\varepsilon}(t_m,\mathbf X_m)
\right)^+
\right]
\\
&\qquad\quad-
\sum_{m=0}^{M-1}
\mathbb E\left[
g_{h}^{K,N,\varepsilon}(t_m,\mathbf X_m)
\right]
+
\sum_{m=0}^{M-1}
\mathbb E\left[
p_{h}^{K,N,\varepsilon}(t_m,\mathbf X_m)
\right].
\end{aligned}
\]
Since
\[
v_\varepsilon^{K,N}(1,\mathbf X_M)
=
F_{K,N}(\mathbf X_M)
=
N\,G_K(\mu_{\mathbf X_M}^N),
\]
and the running cost is nonnegative, division by \(N\), subtraction of
\(\mathcal V_0\), and Theorem~\ref{thm:projection-first-global} prove the claim.
\end{proof}

\begin{prop}[A sufficient condition for a vanishing matched Euler residual]
\label{prop:vanishing-euler-residual-fixed}
For \(t<1\) and \(\mathbf x\in(\mathcal{H}_K)^N\), define the frozen-feedback
matched process
\[
\mathbf Y_s^{t,\mathbf x}
=
\mathbf x
+
s\,b(t,\mathbf x)
+
\mathbf Q_{K,\mathrm{prop}}^{1/2}\mathbf W_s,
\qquad
0\leq s\leq h.
\]
Assume that the verification theorem for \(v_\varepsilon^{K,N}\) holds.
Then
\begin{equation}
\label{eq:frozen-feedback-euler-identity-fixed}
r_h^{K,N,\varepsilon}(t,\mathbf x)
=
\frac{\varepsilon}{2}
\mathbb E
\int_0^h
\left\|
b(t,\mathbf x)
-
b(t+s,\mathbf Y_s^{t,\mathbf x})
\right\|_{\mathbf Q_{K,\mathrm{prop}}^{-1}}^2
\,ds.
\end{equation}
In particular,
\[
r_h^{K,N,\varepsilon}(t,\mathbf x)\geq0.
\]

Consider a sequence
\((K_j,N_j,\varepsilon_j,h_j)\), and write
\[
b_j=b_{\varepsilon_j}^{K_j,N_j},
\qquad
\mathbf Q_{j,\mathrm{prop}}
=
I_{N_j}\otimes Q_{K_j,\mathrm{prop}}.
\]
Assume that there exist \(\zeta\in(0,1]\) and constants \(L_j\) such that
\[
\frac1{N_j}
\left\|
b_j(s,\mathbf y)-b_j(t,\mathbf x)
\right\|_{\mathbf Q_{j,\mathrm{prop}}^{-1}}^2
\leq
L_j^2
\left(
|s-t|^{2\zeta}
+
\frac1{N_j}
\|\mathbf y-\mathbf x\|^2
\right).
\]
Define
\[
B_j
=
\frac1{N_j}
\sum_{m=0}^{M_j-1}
h_j
\mathbb E
\left[
\left\|
b_j(t_m,\mathbf X_m)
\right\|^2
\right].
\]
Then
\[
\mathcal E_{h_j}^{K_j,N_j,\varepsilon_j}
\leq
\frac{\varepsilon_jL_j^2}{2}
\left[
\frac{h_j^{2\zeta}}{2\zeta+1}
+
\frac{h_j}{2}
\operatorname{Tr}(Q_{K_j,\mathrm{prop}})
+
\frac{h_j^2}{3}B_j
\right].
\]
Consequently,
\[
\varepsilon_jL_j^2
\left[
h_j^{2\zeta}
+
h_j\operatorname{Tr}(Q_{K_j,\mathrm{prop}})
+
h_j^2B_j
\right]
\longrightarrow0
\]
implies
\[
\mathcal E_{h_j}^{K_j,N_j,\varepsilon_j}
\longrightarrow0.
\]
\end{prop}

\begin{proof}
For brevity, write
\[
v=v_\varepsilon^{K,N},
\qquad
\mathbf Q=\mathbf Q_{K,\mathrm{prop}}.
\]
The matched particle HJB is
\[
\partial_t v
+
\frac12
\operatorname{Tr}(\mathbf QD^2v)
-
\frac1{2\varepsilon}
\langle Dv,\mathbf QDv\rangle
=
0,
\]
and the matched feedback is
\[
b(t,\mathbf x)
=
-\frac1\varepsilon
\mathbf QDv(t,\mathbf x).
\]
At a generic point \((s,\mathbf y)\), for any
\(\boldsymbol\theta\in(\mathcal{H}_K)^N\),
\[
\begin{aligned}
\frac{\varepsilon}{2}
\|\boldsymbol\theta-b(s,\mathbf y)\|_{\mathbf Q^{-1}}^2
={}&
\frac{\varepsilon}{2}
\|\boldsymbol\theta\|_{\mathbf Q^{-1}}^2
+
\langle\boldsymbol\theta,Dv(s,\mathbf y)\rangle+
\frac1{2\varepsilon}
\langle
Dv(s,\mathbf y),
\mathbf QDv(s,\mathbf y)
\rangle.
\end{aligned}
\]
Using the HJB equation in the final term gives
\[
\begin{aligned}
&
\partial_tv(s,\mathbf y)
+
\langle\boldsymbol\theta,Dv(s,\mathbf y)\rangle
+
\frac12
\operatorname{Tr}
\left(
\mathbf QD^2v(s,\mathbf y)
\right)
+
\frac{\varepsilon}{2}
\|\boldsymbol\theta\|_{\mathbf Q^{-1}}^2=
\frac{\varepsilon}{2}
\|\boldsymbol\theta-b(s,\mathbf y)\|_{\mathbf Q^{-1}}^2.
\end{aligned}
\]

Fix \((t,\mathbf x)\), set
\[
\boldsymbol\theta_0=b(t,\mathbf x),
\]
and define the frozen-feedback process
\[
\mathbf Y_\rho^{t,\mathbf x}
=
\mathbf x
+
\rho\boldsymbol\theta_0
+
\mathbf Q^{1/2}\mathbf B_\rho,
\qquad
0\leq\rho\leq h,
\]
where \(\mathbf B\) is a standard Brownian motion. Applying It\^o's
formula to
\[
\rho
\longmapsto
v(t+\rho,\mathbf Y_\rho^{t,\mathbf x})
\]
and using the preceding completion-of-squares identity yields
\[
\begin{aligned}
&
\mathbb E
v(t+h,\mathbf Y_h^{t,\mathbf x})
+
\frac{\varepsilon h}{2}
\|b(t,\mathbf x)\|_{\mathbf Q^{-1}}^2
-
v(t,\mathbf x)=
\frac{\varepsilon}{2}
\mathbb E
\int_0^h
\left\|
b(t,\mathbf x)
-
b(t+\rho,\mathbf Y_\rho^{t,\mathbf x})
\right\|_{\mathbf Q^{-1}}^2
\,d\rho.
\end{aligned}
\]
Since
\[
\mathbf Y_h^{t,\mathbf x}
\overset{d}{=}
\mathbf x
+
h\,b(t,\mathbf x)
+
\sqrt h\,\mathbf Q^{1/2}Z,
\]
the left-hand side is exactly
\(r_h^{K,N,\varepsilon}(t,\mathbf x)\). This proves
~\Eqref{eq:frozen-feedback-euler-identity-fixed}.

For the joint sequence, conditionally on \(\mathbf X_m\), define
\[
\mathbf Y_{m,\rho}^{(j)}
=
\mathbf X_m
+
\rho\,b_j(t_m,\mathbf X_m)
+
\mathbf Q_{j,\mathrm{prop}}^{1/2}\mathbf B_{m,\rho},
\qquad
0\leq\rho\leq h_j,
\]
where \(\mathbf B_{m,\cdot}\) is independent of
\(\mathcal F_{t_m}\). Then
\[
\mathbf Y_{m,\rho}^{(j)}-\mathbf X_m
=
\rho\,b_j(t_m,\mathbf X_m)
+
\mathbf Q_{j,\mathrm{prop}}^{1/2}\mathbf B_{m,\rho}.
\]
The Brownian increment has conditional mean zero, so the cross term
vanishes, and
\[
\begin{aligned}
\frac1{N_j}
\mathbb E\left[
\left\|
\mathbf Y_{m,\rho}^{(j)}-\mathbf X_m
\right\|^2
\middle|
\mathbf X_m
\right]
={}&
\frac{\rho^2}{N_j}
\|b_j(t_m,\mathbf X_m)\|^2
\\
&+
\rho\operatorname{Tr}(Q_{K_j,\mathrm{prop}}),
\end{aligned}
\]
where we used
\[
\operatorname{Tr}(\mathbf Q_{j,\mathrm{prop}})
=
N_j\operatorname{Tr}(Q_{K_j,\mathrm{prop}}).
\]

The assumed H\"older--Lipschitz estimate therefore gives
\[
\begin{aligned}
&
\frac1{N_j}
\mathbb E\left[
\left\|
b_j(t_m+\rho,\mathbf Y_{m,\rho}^{(j)})
-
b_j(t_m,\mathbf X_m)
\right\|_{\mathbf Q_{j,\mathrm{prop}}^{-1}}^2
\right]\leq
L_j^2
\left[
\rho^{2\zeta}
+
\frac{\rho^2}{N_j}
\mathbb E
\|b_j(t_m,\mathbf X_m)\|^2
+
\rho\operatorname{Tr}(Q_{K_j,\mathrm{prop}})
\right].
\end{aligned}
\]
Substituting this estimate into the frozen-feedback identity and
integrating over \(\rho\in[0,h_j]\) yields
\[
\begin{aligned}
\frac1{N_j}
\mathbb E
\left[
r_{h_j}^{K_j,N_j,\varepsilon_j}
(t_m,\mathbf X_m)
\right]
\leq{}&
\frac{\varepsilon_jL_j^2}{2}
\left[
\frac{h_j^{2\zeta+1}}{2\zeta+1}
+
\frac{h_j^3}{3N_j}
\mathbb E
\|b_j(t_m,\mathbf X_m)\|^2+
\frac{h_j^2}{2}
\operatorname{Tr}(Q_{K_j,\mathrm{prop}})
\right].
\end{aligned}
\]
Finally, summing over \(m=0,\ldots,M_j-1\), using
\(M_jh_j=1\), and recalling
\[
B_j
=
\frac1{N_j}
\sum_{m=0}^{M_j-1}
h_j\,
\mathbb E
\|b_j(t_m,\mathbf X_m)\|^2,
\]
gives
\[
\mathcal E_{h_j}^{K_j,N_j,\varepsilon_j}
\leq
\frac{\varepsilon_jL_j^2}{2}
\left[
\frac{h_j^{2\zeta}}{2\zeta+1}
+
\frac{h_j}{2}
\operatorname{Tr}(Q_{K_j,\mathrm{prop}})
+
\frac{h_j^2}{3}B_j
\right].
\]
\end{proof}
\begin{rmk}[Smooth interpretation and magnitude of the covariance effect]
\label{rem:smooth-fixed-covariance-effect}
Suppose \(v_\varepsilon^{K,N}\) is sufficiently smooth for the following
second-order Gaussian expansion to hold uniformly along the states visited by
the chain. Define
\[
\chi_v(t,\mathbf x)
:=
\operatorname{Tr}\left(
\left[
\mathbf Q_{K,\mathrm{dyn}}
-
\mathbf Q_{K,\mathrm{prop}}
\right]
D^2v_\varepsilon^{K,N}(t,\mathbf x)
\right).
\]
Then
\[
c_h^{K,N,\varepsilon}(t,\mathbf x)
=
\frac h2
\chi_v(t+h,\overline{\mathbf x}_h)
+
o(h).
\]
Since the positive-part map is Lipschitz,
\[
\begin{aligned}
g_{h}^{K,N,\varepsilon}(t,\mathbf x)
&=
\frac h2
\left(
-\chi_v(t+h,\overline{\mathbf x}_h)
\right)_+
+
o(h),
\\
p_{h}^{K,N,\varepsilon}(t,\mathbf x)
&=
\frac h2
\left(
\chi_v(t+h,\overline{\mathbf x}_h)
\right)_+
+
o(h).
\end{aligned}
\]
Consequently,
\[
\begin{aligned}
\mathcal G_{h}^{K,N,\varepsilon}
&=
\frac1{2N}
\sum_{m=0}^{M-1}
h\,
\mathbb E\left[
\left(
-\chi_v(t_{m+1},\overline{\mathbf X}_{m+1})
\right)_+
\right]
+
o(1),
\\
\mathcal P_{h}^{K,N,\varepsilon}
&=
\frac1{2N}
\sum_{m=0}^{M-1}
h\,
\mathbb E\left[
\left(
\chi_v(t_{m+1},\overline{\mathbf X}_{m+1})
\right)_+
\right]
+
o(1).
\end{aligned}
\]
If each one-step remainder is \(O(h^2)\), then the accumulated remainder is
\(O(h)\).

For proportional covariances, let
\[
\kappa_v(t,\mathbf x)
=
\operatorname{Tr}\left(
\mathbf Q_{K,\mathrm{prop}}
D^2v_\varepsilon^{K,N}(t,\mathbf x)
\right).
\]
Then
\[
\chi_v
=
(\alpha_0-1)\kappa_v.
\]
If the execution covariance is colder, \(\alpha_0<1\), then
\[
\begin{aligned}
g_{h}^{K,N,\varepsilon}(t,\mathbf x)
&=
\frac h2
(1-\alpha_0)
\left(
\kappa_v(t+h,\overline{\mathbf x}_h)
\right)_+
+
o(h),
\\
p_{h}^{K,N,\varepsilon}(t,\mathbf x)
&=
\frac h2
(1-\alpha_0)
\left(
-\kappa_v(t+h,\overline{\mathbf x}_h)
\right)_+
+
o(h).
\end{aligned}
\]
Thus positive continuation-value curvature contributes to the negative gain
term, while negative curvature contributes to the unfavorable term. If the
execution covariance is hotter, \(\alpha_0>1\), the roles of the positive and
negative curvature regions are reversed:
\[
\begin{aligned}
g_{h}^{K,N,\varepsilon}(t,\mathbf x)
&=
\frac h2
(\alpha_0-1)
\left(
-\kappa_v(t+h,\overline{\mathbf x}_h)
\right)_+
+
o(h),
\\
p_{h}^{K,N,\varepsilon}(t,\mathbf x)
&=
\frac h2
(\alpha_0-1)
\left(
\kappa_v(t+h,\overline{\mathbf x}_h)
\right)_+
+
o(h).
\end{aligned}
\]
Therefore a fixed cold covariance improves the approximation whenever
positive weighted curvature dominates negative weighted curvature along the
visited path; a fixed hot covariance improves it under the opposite
curvature balance. 
\end{rmk}

\begin{prop}[Continuous-time covariance identity]
\label{prop:continuous-fixed-covariance-identity}
Assume that \(v_\varepsilon^{K,N}\) is \(C^{1,2}\) and that the matched
closed-loop feedback
\[
b(t,\mathbf x)
=
-\frac1\varepsilon
\mathbf Q_{K,\mathrm{prop}}
Dv_\varepsilon^{K,N}(t,\mathbf x)
\]
is admissible. Consider
\[
d\mathbf X_t^{\mathrm{dyn}}
=
b(t,\mathbf X_t^{\mathrm{dyn}})\,dt
+
\mathbf Q_{K,\mathrm{dyn}}^{1/2}d\mathbf W_t,
\qquad
\mathbf X_0^{\mathrm{dyn}}
=
\Pi_K\mathbf x_0.
\]
Assume that the covariance SDE above admits a nonexplosive
strong solution and that
\[
\mathbb E
\int_0^1
\left[
\ell_\varepsilon^{K,N}
(t,X_t^{\mathrm{dyn}})
+
|\chi_v(t,X_t^{\mathrm{dyn}})|
+
\left\|
Q_{K,\mathrm{dyn}}^{1/2}
Dv_\varepsilon^{K,N}
(t,X_t^{\mathrm{dyn}})
\right\|^2
\right]dt
<
\infty.
\]
Set
\[
\chi_t
:=
\operatorname{Tr}\left(
\left[
\mathbf Q_{K,\mathrm{dyn}}
-
\mathbf Q_{K,\mathrm{prop}}
\right]
D^2v_\varepsilon^{K,N}
(t,\mathbf X_t^{\mathrm{dyn}})
\right).
\]
Then
\[
\begin{aligned}
&
\mathbb E
F_{K,N}(\mathbf X_1^{\mathrm{dyn}})
+
\mathbb E
\int_0^1
\ell_\varepsilon^{K,N}
(t,\mathbf X_t^{\mathrm{dyn}})\,dt =
v_\varepsilon^{K,N}
\bigl(0,\Pi_K\mathbf x_0\bigr)
+
\frac12
\mathbb E
\int_0^1
\chi_t\,dt.
\end{aligned}
\]
Define the nonnegative continuous-time gain and unfavorable contribution
\[
\mathfrak G_{}^{K,N,\varepsilon}
:=
\frac1{2N}
\mathbb E
\int_0^1
(-\chi_t)_+\,dt,
\qquad
\mathfrak P_{}^{K,N,\varepsilon}
:=
\frac1{2N}
\mathbb E
\int_0^1
(\chi_t)_+\,dt.
\]
Then
\begin{equation}
\label{eq:continuous-fixed-covariance-residual}
\mathfrak R_{}^{K,N,\varepsilon}
:=
\frac1{2N}
\mathbb E
\int_0^1
\chi_t\,dt
=
-\mathfrak G_{}^{K,N,\varepsilon}
+
\mathfrak P_{}^{K,N,\varepsilon},
\end{equation}
and consequently
\begin{equation}
\label{eq:continuous-fixed-covariance-bound-with-gain}
\mathbb E
G_K(\mu_{\mathbf X_1^{\mathrm{dyn}}}^N)
-
\mathcal V_0
\leq
\mathcal R_{K,N,\varepsilon}^{\mathrm{prop}}
-
\mathfrak G_{}^{K,N,\varepsilon}
+
\mathfrak P_{}^{K,N,\varepsilon}.
\end{equation}
In the proportional case,
\[
\chi_t
=
(\alpha_0-1)
\operatorname{Tr}\left(
\mathbf Q_{K,\mathrm{prop}}
D^2v_\varepsilon^{K,N}
(t,\mathbf X_t^{\mathrm{dyn}})
\right).
\]
\end{prop}

\begin{proof}
The matched particle HJB can be written in closed-loop form as
\[
\partial_tv_\varepsilon^{K,N}
+
\left\langle
b,Dv_\varepsilon^{K,N}
\right\rangle
+
\frac12
\operatorname{Tr}\left(
\mathbf Q_{K,\mathrm{prop}}
D^2v_\varepsilon^{K,N}
\right)
+
\ell_\varepsilon^{K,N}
=
0.
\]
Applying It\^o's formula to
\(v_\varepsilon^{K,N}(t,\mathbf X_t^{\mathrm{dyn}})\) gives
\[
\begin{aligned}
d
v_\varepsilon^{K,N}
(t,\mathbf X_t^{\mathrm{dyn}})
=
\left[
-\ell_\varepsilon^{K,N}
(t,\mathbf X_t^{\mathrm{dyn}})
+
\frac12\chi_t
\right]dt
+
\left\langle
Dv_\varepsilon^{K,N},
\mathbf Q_{K,\mathrm{dyn}}^{1/2}d\mathbf W_t
\right\rangle.
\end{aligned}
\]
Integrate, take expectations, and use the terminal condition. Dividing by
\(N\), dropping the nonnegative running cost, and applying
Theorem~\ref{thm:projection-first-global} yield
~\Eqref{eq:continuous-fixed-covariance-bound-with-gain}. The
decomposition follows from
\(\chi_t=(\chi_t)_+-(-\chi_t)_+\).
\end{proof}

\begin{rmk}[Interpretation and vanishing conditions for the covariance terms]
\label{rem:fixed-covariance-residual}
The negative term $-\mathfrak G_{}^{K,N,\varepsilon}$ is the part of the covariance nonmatch that lowers the matched
continuation-value certificate, while
$
\mathfrak P_{}^{K,N,\varepsilon}$ is the unfavorable part. Hence nonmatched execution improves the
matched continuous-time upper bound whenever $\mathfrak G_{}^{K,N,\varepsilon}
>
\mathfrak P_{}^{K,N,\varepsilon}$.

For a fixed cold covariance \(\alpha_0<1\), writing
\[
\kappa_v(t,\mathbf x)
=
\operatorname{Tr}\left(
\mathbf Q_{K,\mathrm{prop}}
D^2v_\varepsilon^{K,N}(t,\mathbf x)
\right),
\]
one has
\[
\begin{aligned}
\mathfrak G_{}^{K,N,\varepsilon}
&=
\frac{1-\alpha_0}{2N}
\mathbb E
\int_0^1
\left(
\kappa_v(t,\mathbf X_t^{\mathrm{dyn}})
\right)_+dt,
\\
\mathfrak P_{}^{K,N,\varepsilon}
&=
\frac{1-\alpha_0}{2N}
\mathbb E
\int_0^1
\left(
-\kappa_v(t,\mathbf X_t^{\mathrm{dyn}})
\right)_+dt.
\end{aligned}
\]
Thus positive weighted curvature contributes directly to the negative gain
term. For a hot covariance, the roles of positive and negative
curvature are reversed.

Moreover,
\[
\mathfrak G_{}^{K,N,\varepsilon}
+
\mathfrak P_{}^{K,N,\varepsilon}
\leq
\frac{|\alpha_0-1|}{2N}
\mathbb E
\int_0^1
\left|
\kappa_v(t,\mathbf X_t^{\mathrm{dyn}})
\right|dt.
\]
Hence both terms vanish along a joint sequence whenever the total weighted
particle curvature is \(o(N)\), after accounting for any dependence on
\(K\) and \(\varepsilon\). Vanishing is not required for improvement: a
strictly positive excess of gain over unfavorable curvature yields a
strictly better upper bound.
\end{rmk}

\subsubsection{Consistency of the fully finite one outer-loop rollout implementation}

For $\mathbf x=(x_1,\ldots,x_N)\in(\mathcal{H}_K)^N$, write $\mathcal J_{K,N}(\mathbf x)
:=
G_K(\mu_{\mathbf x}^N)
=
\frac1N F_{K,N}(\mathbf x)$. For each Euler step \(m\), let
\(\Xi_m^{R,S}\) denote the complete collection of fresh context and
candidate variables used to construct the multi-context estimator, i.e.,
\[
\Xi_m^{R,S}
:=
\left(
\left(Z_{m,p}^{(r)}\right)_{
\substack{1\leq r\leq R\\1\leq p\leq N}},
\left(\xi_{m,i,r,s}\right)_{
\substack{1\leq i\leq N\\1\leq r\leq R\\1\leq s\leq S}}
\right),
\]
where all $Z_{m,p}^{(r)},\ \xi_{m,i,r,s}
\sim
\mathcal N(0,I_K)$ are mutually independent.

Write $\widehat b_m^{R,S}(\mathbf x)
=
\widehat b^{R,S}(t_m,\mathbf x;\Xi_m^{R,S})
\in(\mathcal{H}_K)^N$. The variables $\Xi_0^{R,S},\ldots,\Xi_{M-1}^{R,S}, Z_1^{\mathrm{ex}},\ldots,Z_M^{\mathrm{ex}}$ are mutually independent, and $Z_m^{\mathrm{ex}}
\sim
\mathcal N(0,I_{NK})$.

Define the finite-candidate chain by
\begin{equation}
\label{eq:finite-fixed-covariance-chain}
\left\{
\begin{aligned}
\mathbf C_0^{R,S,h}
&=
\Pi_K\mathbf x_0,
\\
\mathbf C_{m+1}^{R,S,h}
&=
\mathbf C_m^{R,S,h}
+
h\widehat b_m^{R,S}
\left(
\mathbf C_m^{R,S,h}
\right)
+
\sqrt h\,
\mathbf Q_{K,\mathrm{dyn}}^{1/2}
Z_{m+1}^{\mathrm{ex}}.
\end{aligned}
\right.
\end{equation}

On the same probability space, using the same execution noises
\(Z_{m+1}^{\mathrm{ex}}\), define the exact-feedback nonmatched-covariance
chain by
\begin{equation}
\label{eq:coupled-exact-fixed-covariance-chain}
\left\{
\begin{aligned}
\mathbf X_0^h
&=
\Pi_K\mathbf x_0,
\\
\mathbf X_{m+1}^h
&=
\mathbf X_m^h
+
h b(t_m,\mathbf X_m^h)
+
\sqrt h\,
\mathbf Q_{K,\mathrm{dyn}}^{1/2}
Z_{m+1}^{\mathrm{ex}}.
\end{aligned}
\right.
\end{equation}

Set
\[
\delta_{R,S}
:=
\frac1R+\frac1{RS}.
\]

\begin{ass}[Quantitative stability of the finite implementation]
\label{ass:finite-implementation-fixed}
Fix \(K,N,\varepsilon>0\) and \(h=1/M>0\). Assume the following.

\begin{enumerate}
\item
There exist constants \(L_b,B_b<\infty\) such that, for every
\(m=0,\ldots,M-1\) and every
\(\mathbf x,\mathbf y\in(\mathcal{H}_K)^N\),
\[
\left\|
b(t_m,\mathbf x)-b(t_m,\mathbf y)
\right\|
\leq
L_b\|\mathbf x-\mathbf y\|,
\]
and
\[
\|b(t_m,\mathbf x)\|^2
\leq
B_b
\left(
1+\|\mathbf x\|^2
\right).
\]

\item
There exists \(C_{\mathrm{MC}}<\infty\), independent of
\(m,R,S\), and \(\mathbf x\), such that
\begin{equation}
\label{eq:uniform-multicontext-mse}
\mathbb E_{\Xi_m^{R,S}}
\left[
\left\|
\widehat b_m^{R,S}(\mathbf x)
-
b(t_m,\mathbf x)
\right\|^2
\right]
\leq
C_{\mathrm{MC}}
\delta_{R,S}
\left(
1+\|\mathbf x\|^2
\right).
\end{equation}
Here the expectation is only over the fresh context and candidate
variables at step \(m\).

\item
The terminal cloud functional is globally Lipschitz: there exists
\(L_{\mathcal J}<\infty\) such that
\begin{equation}
\label{eq:terminal-cloud-lipschitz}
\left|
\mathcal J_{K,N}(\mathbf x)
-
\mathcal J_{K,N}(\mathbf y)
\right|
\leq
L_{\mathcal J}
\|\mathbf x-\mathbf y\|,
\qquad
\mathbf x,\mathbf y\in(\mathcal{H}_K)^N.
\end{equation}
\end{enumerate}
\end{ass}

\begin{thm}[Strong consistency of finite-candidate SCMO]
\label{thm:finite-scmo-consistency-fixed}
Under
Assumption~\ref{ass:finite-implementation-fixed}, there exists a
constant
\[
C_{\mathrm{stab}}
=
C_{\mathrm{stab}}
\left(
K,N,\varepsilon,h,L_b,B_b,C_{\mathrm{MC}},
\mathbf Q_{K,\mathrm{dyn}},\mathbf x_0
\right)
<
\infty
\]
which is independent of \(R\) and \(S\), such that
\begin{equation}
\label{eq:strong-finite-chain-rate}
\max_{0\leq m\leq M}
\mathbb E
\left[
\left\|
\mathbf C_m^{R,S,h}
-
\mathbf X_m^h
\right\|^2
\right]
\leq
C_{\mathrm{stab}}
\left(
\frac1R+\frac1{RS}
\right).
\end{equation}
Consequently,
\[
\mathbf C_m^{R,S,h}
\longrightarrow
\mathbf X_m^h
\]
in \(L^2\), and hence in probability, for every \(m\leq M\), whenever
\[
\delta_{R,S}\longrightarrow0.
\]

Moreover,
\begin{equation}
\label{eq:terminal-objective-mc-rate}
\left|
\mathbb E
G_K
\left(
\mu_{\mathbf C_M^{R,S,h}}^N
\right)
-
\mathbb E
G_K
\left(
\mu_{\mathbf X_M^h}^N
\right)
\right|
\leq
L_{\mathcal J}
\sqrt{
C_{\mathrm{stab}}\delta_{R,S}
}.
\end{equation}
It follows that
\[
\begin{aligned}
\limsup_{\delta_{R,S}\to0}
\left[
\mathbb E
G_K
\left(
\mu_{\mathbf C_M^{R,S,h}}^N
\right)
-
\mathcal V_0
\right]
\leq{}&
\mathcal R_{K,N,\varepsilon}^{\mathrm{prop}}
+
\mathcal E_h^{K,N,\varepsilon}
-
\mathcal G_{h}^{K,N,\varepsilon}
+
\mathcal P_{h}^{K,N,\varepsilon}.
\end{aligned}
\]
\end{thm}

\begin{proof}
We first establish a moment bound for the finite chain. By~\Eqref{eq:uniform-multicontext-mse} and the linear-growth
bound on \(b\),
\[
\begin{aligned}
\mathbb E_{\Xi_m^{R,S}}
\left[
\left\|
\widehat b_m^{R,S}(\mathbf x)
\right\|^2
\right]
&\leq
2\|b(t_m,\mathbf x)\|^2+
2\mathbb E_{\Xi_m^{R,S}}
\left[
\left\|
\widehat b_m^{R,S}(\mathbf x)
-
b(t_m,\mathbf x)
\right\|^2
\right]
\leq
B_{\widehat b}
\left(
1+\|\mathbf x\|^2
\right),
\end{aligned}
\]
where one may take
\[
B_{\widehat b}
=
2B_b+4C_{\mathrm{MC}},
\]
because
\[
\delta_{R,S}
=
\frac1R+\frac1{RS}
\leq2.
\]

For arbitrary vectors \(x,y\) and \(h\in(0,1]\),
\[
\|x+hy\|^2
\leq
(1+h)\|x\|^2
+
(h+h^2)\|y\|^2.
\]
Using the independence and zero mean of the execution noise gives
\[
\begin{aligned}
\mathbb E
\left[
\left\|
\mathbf C_{m+1}^{R,S,h}
\right\|^2
\right]
\leq{}&
(1+h)
\mathbb E
\left[
\left\|
\mathbf C_m^{R,S,h}
\right\|^2
\right]
+
(h+h^2)
\mathbb E
\left[
\left\|
\widehat b_m^{R,S}
\left(
\mathbf C_m^{R,S,h}
\right)
\right\|^2
\right]
+
h\operatorname{Tr}
\left(
\mathbf Q_{K,\mathrm{dyn}}
\right).
\end{aligned}
\]
Hence there is \(c_0<\infty\), independent of \(R,S\), such that
\[
\mathbb E
\left[
\left\|
\mathbf C_{m+1}^{R,S,h}
\right\|^2
\right]
\leq
(1+c_0h)
\mathbb E
\left[
\left\|
\mathbf C_m^{R,S,h}
\right\|^2
\right]
+
c_0h.
\]
The discrete Gronwall inequality yields
\begin{equation}
\label{eq:uniform-finite-chain-moment}
\sup_{R,S\geq1}
\max_{0\leq m\leq M}
\mathbb E
\left[
\left\|
\mathbf C_m^{R,S,h}
\right\|^2
\right]
\leq
K_C
\end{equation}
for some \(K_C<\infty\).

Set
\[
D_m^{R,S}
=
\mathbf C_m^{R,S,h}
-
\mathbf X_m^h
\]
and
\[
e_m^{R,S}
=
\widehat b_m^{R,S}
\left(
\mathbf C_m^{R,S,h}
\right)
-
b
\left(
t_m,\mathbf C_m^{R,S,h}
\right).
\]
Because the exact and finite chains use the same execution noise,
the noise terms cancel and
\[
\begin{aligned}
D_{m+1}^{R,S}
={}&
D_m^{R,S}+
h
\left[
b
\left(
t_m,\mathbf C_m^{R,S,h}
\right)
-
b(t_m,\mathbf X_m^h)
\right]
+
h e_m^{R,S}.
\end{aligned}
\]
Define
\[
A_m^{R,S}
=
D_m^{R,S}
+
h
\left[
b
\left(
t_m,\mathbf C_m^{R,S,h}
\right)
-
b(t_m,\mathbf X_m^h)
\right].
\]
The Lipschitz condition gives
\[
\|A_m^{R,S}\|
\leq
(1+hL_b)
\|D_m^{R,S}\|.
\]
Using
\[
\|a+he\|^2
\leq
(1+h)\|a\|^2
+
(h+h^2)\|e\|^2,
\]
we obtain
\[
\begin{aligned}
\mathbb E
\left[
\left\|
D_{m+1}^{R,S}
\right\|^2
\right]
\leq{}&
\rho_h
\mathbb E
\left[
\left\|
D_m^{R,S}
\right\|^2
\right]
+
(h+h^2)
\mathbb E
\left[
\left\|
e_m^{R,S}
\right\|^2
\right],
\end{aligned}
\]
where
\[
\rho_h
=
(1+h)(1+hL_b)^2.
\]

Conditioning on
\(\mathbf C_m^{R,S,h}\), using the independence of the fresh
context and candidate samples, and applying
~\Eqref{eq:uniform-multicontext-mse}, we get
\[
\begin{aligned}
\mathbb E
\left[
\left\|
e_m^{R,S}
\right\|^2
\right]
&=
\mathbb E
\left[
\mathbb E
\left[
\left.
\left\|
e_m^{R,S}
\right\|^2
\right|
\mathbf C_m^{R,S,h}
\right]
\right]
\leq
C_{\mathrm{MC}}
\delta_{R,S}
\left[
1+
\mathbb E
\left\|
\mathbf C_m^{R,S,h}
\right\|^2
\right]
\leq
C_{\mathrm{MC}}
(1+K_C)
\delta_{R,S}.
\end{aligned}
\]
Therefore,
\[
\mathbb E
\left[
\left\|
D_{m+1}^{R,S}
\right\|^2
\right]
\leq
\rho_h
\mathbb E
\left[
\left\|
D_m^{R,S}
\right\|^2
\right]
+
(h+h^2)
C_{\mathrm{MC}}
(1+K_C)
\delta_{R,S}.
\]
Since $D_0^{R,S}=0$, iteration gives
\[
\mathbb E
\left[
\left\|
D_m^{R,S}
\right\|^2
\right]
\leq
(h+h^2)
C_{\mathrm{MC}}
(1+K_C)
\left(
\sum_{q=0}^{m-1}\rho_h^q
\right)
\delta_{R,S}.
\]
Thus ~\Eqref{eq:strong-finite-chain-rate} holds with
\[
C_{\mathrm{stab}}
=
(h+h^2)
C_{\mathrm{MC}}
(1+K_C)
\sum_{q=0}^{M-1}\rho_h^q.
\]
Finally, the Lipschitz condition on \(\mathcal J_{K,N}\) gives
\[
\begin{aligned}
&
\left|
\mathbb E
\mathcal J_{K,N}
\left(
\mathbf C_M^{R,S,h}
\right)
-
\mathbb E
\mathcal J_{K,N}
\left(
\mathbf X_M^h
\right)
\right|
\leq
L_{\mathcal J}
\mathbb E
\left[
\left\|
\mathbf C_M^{R,S,h}
-
\mathbf X_M^h
\right\|
\right]
\leq
L_{\mathcal J}
\sqrt{
C_{\mathrm{stab}}\delta_{R,S}
}.
\end{aligned}
\]
Combining this convergence with
Theorem~\ref{thm:fixed-covariance-bound}
proves the final assertion.
\end{proof}

\begin{rmk}[A directly checkable smooth sufficient condition]
\label{rem:checkable-finite-implementation-condition}
Suppose
\[
F_{K,N}\in C_b^2((\mathcal{H}_K)^N),
\]
and write
\[
M_1
=
\sup_{\mathbf x}
\|DF_{K,N}(\mathbf x)\|,
\qquad
M_2
=
\sup_{\mathbf x}
\|D^2F_{K,N}(\mathbf x)\|_{\mathrm{op}}.
\]
Assume also that
\[
F_{\min}
\leq
F_{K,N}(\mathbf x)
\leq
F_{\max},
\qquad
\mathbf x\in(H_K)^N.
\]

Let
\[
\pi_{t,\mathbf x}(d\mathbf u)
=
\frac{
\exp(
-F_{K,N}(\mathbf x+\mathbf u)/\varepsilon
)
}{
\mathbb E[
\exp(
-F_{K,N}(\mathbf x+\mathbf U)/\varepsilon
)
]
}
\,
\mathcal N
\left(
0,(1-t)\mathbf Q_{K,\mathrm{prop}}
\right)(d\mathbf u).
\]
Gaussian integration by parts gives
\[
b(t,\mathbf x)
=
-\frac1\varepsilon
\mathbf Q_{K,\mathrm{prop}}
\mathbb E_{\pi_{t,\mathbf x}}
\left[
DF_{K,N}(\mathbf x+\mathbf U)
\right].
\]
Consequently,
\[
\|b(t,\mathbf x)\|
\leq
\frac{
\|\mathbf Q_{K,\mathrm{prop}}\|_{\mathrm{op}}M_1
}{
\varepsilon
},
\]
and
\[
\|D_{\mathbf x}b(t,\mathbf x)\|_{\mathrm{op}}
\leq
\|\mathbf Q_{K,\mathrm{prop}}\|_{\mathrm{op}}
\left(
\frac{M_2}{\varepsilon}
+
\frac{M_1^2}{\varepsilon^2}
\right).
\]
Thus the feedback is globally Lipschitz and bounded.

Moreover,
\[
e^{-F_{\max}/\varepsilon}
\leq
L_{t,\mathbf x}(\mathbf U)
\leq
e^{-F_{\min}/\varepsilon}.
\]
Hence every empirical Gibbs denominator has the deterministic lower
bound
\[
\overline B_i^{R,S}
\geq
e^{-F_{\max}/\varepsilon},
\]
and the inverse-denominator condition holds automatically. Since the
candidate increments are Gaussian, all required numerator moments are
finite, and the uniform estimate
\[
\mathbb E
\left[
\left\|
\widehat b_m^{R,S}(\mathbf x)-b(t_m,\mathbf x)
\right\|^2
\right]
\leq
C_{\mathrm{MC}}
\left(
\frac1R+\frac1{RS}
\right)
\]
holds with a constant independent of \(\mathbf x,m,R,S\).

Finally,
\[
\mathcal J_{K,N}
=
\frac1N F_{K,N}
\]
is globally Lipschitz with
\[
L_{\mathcal J}
\leq
\frac{M_1}{N}.
\]
Therefore
Assumption~\ref{ass:finite-implementation-fixed}
holds.
\end{rmk}

\begin{rmk}[Localized polynomial-growth conditions and interacting objectives]
\label{rem:polynomial-growth-finite-implementation}

The globally bounded assumptions in
Remark~\ref{rem:checkable-finite-implementation-condition}
are convenient but not necessary. They can be replaced by local
regularity, polynomial growth, denominator stability, and moment bounds
for the finite and exact chains.

Fix \(K,N,\varepsilon>0\) and \(h=1/M>0\), and write
\[
E_{K,N}
=
(\mathcal{H}_K)^N.
\]
For \(m=0,\ldots,M-1\), let
\[
\mathbf U_m
\sim
\mathcal N
\left(
0,\tau_m\mathbf Q_{K,\mathrm{prop}}
\right),
\qquad
\tau_m=1-t_m.
\]
For the finite-candidate estimator, write
\[
\overline B_{m,i}^{R,S}(\mathbf x)
=
\frac1R
\sum_{r=1}^R
\left[
\frac1S
\sum_{s=1}^S
L_{m,i,r,s}(\mathbf x)
\right]
\]
for its empirical Gibbs denominator, and define the terminal cloud
functional
\[
\mathcal J_{K,N}(\mathbf x)
=
G_K(\mu_{\mathbf x}^N).
\]

Assume the following conditions.

\medskip
\noindent
\textit{(i) Structural regularity and polynomial growth.}
Suppose that
\[
F_{K,N}
\in
C^2(E_{K,N})
\]
is bounded from below and satisfies, for some \(q\geq0\),
\begin{equation}
\label{eq:polynomial-F-growth}
\left|
F_{K,N}(\mathbf x)
\right|
\leq
C_F
\left(
1+\|\mathbf x\|^2
\right),
\qquad
\mathbf x\in E_{K,N},
\end{equation}
and
\begin{equation}
\label{eq:polynomial-F-derivative-growth}
\left\|
DF_{K,N}(\mathbf x)
\right\|
+
\left\|
D^2F_{K,N}(\mathbf x)
\right\|_{\mathrm{op}}
\leq
C_F
\left(
1+\|\mathbf x\|^q
\right),
\qquad
\mathbf x\in E_{K,N}.
\end{equation}
The quadratic bound in~\Eqref{eq:polynomial-F-growth} is the
particle-level counterpart of the at-most-quadratic growth assumption
on \(G_K\).

\medskip
\noindent
\textit{(ii) Local inverse-denominator stability.}
Assume that, for every \(R_0<\infty\),
\begin{equation}
\label{eq:local-inverse-denominator}
\sup_{\substack{
0\leq m\leq M-1\\
1\leq i\leq N\\
\|\mathbf x\|\leq R_0\\
R,S\geq1
}}
\mathbb E
\left[
\left(
\overline B_{m,i}^{R,S}(\mathbf x)
\right)^{-4}
\right]
<
\infty.
\end{equation}

\medskip
\noindent
\textit{(iii) Uniform chain moments.}
Assume that, for some \(p>0\),
\begin{equation}
\label{eq:uniform-p-moments-localized}
\sup_{R,S\geq1}
\max_{0\leq m\leq M}
\mathbb E
\left[
\left\|
\mathbf C_m^{R,S,h}
\right\|^p
\right]
+
\max_{0\leq m\leq M}
\mathbb E
\left[
\left\|
\mathbf X_m^h
\right\|^p
\right]
<
\infty.
\end{equation}

\medskip
\noindent
\textit{(iv) Continuity and polynomial growth of the terminal functional.}
Assume that \(\mathcal J_{K,N}\) is continuous and that, for some
\(r<p\),
\begin{equation}
\label{eq:J-polynomial-growth}
\left|
\mathcal J_{K,N}(\mathbf x)
\right|
\leq
C_{\mathcal J}
\left(
1+\|\mathbf x\|^r
\right),
\qquad
\mathbf x\in E_{K,N}.
\end{equation}

We now derive the consequences of these assumptions. Define the
desirability denominator
\[
\beta_m(\mathbf x)
=
\mathbb E
\left[
\exp\left(
-\frac{
F_{K,N}(\mathbf x+\mathbf U_m)
}{\varepsilon}
\right)
\right].
\]
Since \(F_{K,N}\) is bounded from below, the Gibbs factor is bounded
from above. Dominated convergence therefore implies that
\(\beta_m\) is continuous. Moreover,
\[
\beta_m(\mathbf x)>0
\]
for every \(\mathbf x\). Since the set of Euler times is finite, for
every localization radius \(R_0<\infty\),
\begin{equation}
\label{eq:compact-desirability-lower-bound}
\underline\beta_{R_0}
:=
\min_{\substack{
0\leq m\leq M-1\\
\|\mathbf x\|\leq R_0
}}
\beta_m(\mathbf x)
>
0.
\end{equation}

Let
\[
\pi_{m,\mathbf x}(d\mathbf u)
=
\frac{
\exp\left(
-F_{K,N}(\mathbf x+\mathbf u)/\varepsilon
\right)
}{
\beta_m(\mathbf x)
}
\,
\mathcal N
\left(
0,\tau_m\mathbf Q_{K,\mathrm{prop}}
\right)(d\mathbf u)
\]
be the corresponding Gibbs-tilted proposal law. Gaussian integration
by parts gives
\begin{equation}
\label{eq:feedback-gradient-polynomial}
b(t_m,\mathbf x)
=
-\frac1\varepsilon
\mathbf Q_{K,\mathrm{prop}}
\mathbb E_{\pi_{m,\mathbf x}}
\left[
DF_{K,N}(\mathbf x+\mathbf U_m)
\right].
\end{equation}
Differentiating this representation gives
\begin{equation}
\label{eq:feedback-derivative-polynomial}
\begin{aligned}
D_{\mathbf x}b(t_m,\mathbf x)
={}&
-\frac1\varepsilon
\mathbf Q_{K,\mathrm{prop}}
\mathbb E_{\pi_{m,\mathbf x}}
\left[
D^2F_{K,N}(\mathbf x+\mathbf U_m)
\right]
+
\frac1{\varepsilon^2}
\mathbf Q_{K,\mathrm{prop}}
\operatorname{Cov}_{\pi_{m,\mathbf x}}
\left(
DF_{K,N}(\mathbf x+\mathbf U_m)
\right).
\end{aligned}
\end{equation}
\Eqref{eq:polynomial-F-derivative-growth} and
\Eqref{eq:compact-desirability-lower-bound}, together with Gaussian
polynomial moments, imply that, for every \(R_0<\infty\), there are
constants
\[
B_{b,R_0},
\qquad
L_{b,R_0}
<
\infty
\]
such that
\[
\sup_{\substack{
0\leq m\leq M-1\\
\|\mathbf x\|\leq R_0
}}
\|b(t_m,\mathbf x)\|
\leq
B_{b,R_0},
\]
and
\[
\|b(t_m,\mathbf x)-b(t_m,\mathbf y)\|
\leq
L_{b,R_0}
\|\mathbf x-\mathbf y\|
\]
whenever
\[
\|\mathbf x\|\vee\|\mathbf y\|
\leq
R_0.
\]
Thus the exact-feedback is locally bounded and locally Lipschitz,
although it need not be globally bounded or globally Lipschitz.

Because \(F_{K,N}\) is bounded from below, the required local numerator
moments are finite automatically:
\[
\sup_{\substack{
0\leq m\leq M-1\\
1\leq i\leq N\\
\|\mathbf x\|\leq R_0
}}
\mathbb E
\left[
L_{t_m,\mathbf x}(\mathbf U_m)^4
\left(
1+\|\mathbf U_{m,i}\|^4
\right)
\right]
<
\infty.
\]
Combining these numerator bounds with
\Eqref{eq:local-inverse-denominator}, the block variance argument in
Proposition~\ref{prop:multicontext-consistency-fixed} gives, for every
\(R_0<\infty\) and every \(R,S\geq1\),
\begin{equation}
\label{eq:local-multicontext-mse-polynomial}
\sup_{\substack{
0\leq m\leq M-1\\
\|\mathbf x\|\leq R_0
}}
\mathbb E_{\Xi_m^{R,S}}
\left[
\left\|
\widehat b_m^{R,S}(\mathbf x)
-
b(t_m,\mathbf x)
\right\|^2
\right]
\leq
C_{\mathrm{MC},R_0}
\left(
\frac1R+\frac1{RS}
\right).
\end{equation}
If only convergence in probability is required, the inverse-moment
condition can be replaced by compact-uniform convergence in probability
of the numerator and denominator block averages. The inverse-moment
condition is used to obtain the explicit \(L^2\) rate.

Define the joint exit time
\[
\tau_{R_0}
=
\inf
\left\{
m\leq M:
\left\|
\mathbf C_m^{R,S,h}
\right\|
\vee
\left\|
\mathbf X_m^h
\right\|
>
R_0
\right\}.
\]
Up to \(\tau_{R_0}\), the local Lipschitz and local Monte Carlo
constants are deterministic. The coupling proof of
Theorem~\ref{thm:finite-scmo-consistency-fixed} therefore gives
\[
\max_{0\leq m\leq M}
\mathbb E
\left[
\left\|
\mathbf C_{m\wedge\tau_{R_0}}^{R,S,h}
-
\mathbf X_{m\wedge\tau_{R_0}}^h
\right\|^2
\right]
\leq
C_{\mathrm{stab},R_0}
\left(
\frac1R+\frac1{RS}
\right).
\]
Consequently, for every \(\eta>0\),
\begin{equation}
\label{eq:localized-probability-bound}
\begin{aligned}
&
\mathbb P
\left(
\max_{0\leq m\leq M}
\left\|
\mathbf C_m^{R,S,h}
-
\mathbf X_m^h
\right\|
>
\eta
\right)
\\
&\qquad\leq
\frac{
(M+1)C_{\mathrm{stab},R_0}
}{
\eta^2
}
\left(
\frac1R+\frac1{RS}
\right)
+
\mathbb P(\tau_{R_0}\leq M).
\end{aligned}
\end{equation}
By the union bound, Markov's inequality, and
\Eqref{eq:uniform-p-moments-localized},
\[
\mathbb P(\tau_{R_0}\leq M)
\leq
\frac{(M+1)C_p}{R_0^p},
\]
where \(C_p\) is independent of \(R,S\). First sending
\[
R\to\infty,
\qquad
RS\to\infty,
\]
and then \(R_0\to\infty\), proves
\[
\max_{0\leq m\leq M}
\left\|
\mathbf C_m^{R,S,h}
-
\mathbf X_m^h
\right\|
\longrightarrow0
\]
in probability.

Finally, the moment condition
\Eqref{eq:uniform-p-moments-localized} and the growth bound
\Eqref{eq:J-polynomial-growth}, with \(r<p\), imply uniform
integrability of $\{
\mathcal J_{K,N}
(
\mathbf C_M^{R,S,h}
)\}_{R,S}$.

Continuity of \(\mathcal J_{K,N}\), convergence in probability, and
uniform integrability therefore yield
\[
\mathbb E
\mathcal J_{K,N}
\left(
\mathbf C_M^{R,S,h}
\right)
\longrightarrow
\mathbb E
\mathcal J_{K,N}
\left(
\mathbf X_M^h
\right).
\]
Thus the finite-candidate consistency result remains valid under
polynomial growth, without requiring \(F_{K,N}\) or its derivatives to
be globally bounded.
\end{rmk}

\begin{exam}[Quadratic-growth interacting measure objectives]
\label{exam:quadratic-growth-interacting-class}
Let \(\mathcal{H}_K\) be finite-dimensional, let \(a\in \mathcal{H}_K\), and consider
\[
\begin{aligned}
G_K(\mu)
={}&
\frac{\kappa_1}{2}
\int_{H_K}
\|x-a\|^2\,\mu(dx)
+
\frac{\kappa_2}{2}
\left\|
\int_{H_K}x\,\mu(dx)-a
\right\|^2
+
\lambda
\operatorname{MMD}_k^2(\mu,\nu),
\end{aligned}
\]
where
\[
\kappa_1>0,
\qquad
\kappa_2,\lambda\geq0,
\]
\(\nu\in\mathcal P_2(H_K)\), and \(k\) is a bounded symmetric positive-definite \(C^2\) kernel whose
first and second derivatives are bounded.

The first term is an additive quadratic confinement, the second term is
a nonlinear mean interaction, and the third term is a pairwise
distributional interaction. Since
\[
\left\|
\int x\,\mu(dx)-a
\right\|^2
\leq
2M_2(\mu)+2\|a\|^2
\]
and the bounded-kernel MMD term is bounded, there is a constant \(C\)
such that
\[
0
\leq
G_K(\mu)
\leq
C
\left(
1+M_2(\mu)
\right).
\]
Thus \(G_K\) is bounded from below, continuous, and has at most
quadratic growth.

For $
\mathbf x=(x_1,\ldots,x_N)$, $\overline x
=
\frac1N\sum_{i=1}^Nx_i$, the extensive terminal cost is
\[
\begin{aligned}
F_{K,N}(\mathbf x)
={}&
\frac{\kappa_1}{2}
\sum_{i=1}^N
\|x_i-a\|^2
+
\frac{N\kappa_2}{2}
\|\overline x-a\|^2
+
\lambda N
\operatorname{MMD}_k^2
\left(
\mu_{\mathbf x}^N,\nu
\right).
\end{aligned}
\]
Let
\[
k_\nu(x)
=
\int_{H_K}
k(x,y)\,\nu(dy).
\]
Using symmetry of \(k\), differentiation gives
\[
\begin{aligned}
D_{x_i}F_{K,N}(\mathbf x)
={}&
\kappa_1(x_i-a)
+
\kappa_2(\overline x-a)
+
\frac{2\lambda}{N}
\sum_{j=1}^N
D_1k(x_i,x_j)
-
2\lambda Dk_\nu(x_i).
\end{aligned}
\]
The bounded derivative assumptions on \(k\) imply
\[
\left\|
DF_{K,N}(\mathbf x)
\right\|
\leq
C_{K,N}
\left(
1+\|\mathbf x\|
\right),
\]
while the block Hessian satisfies
\[
\left\|
D^2F_{K,N}(\mathbf x)
\right\|_{\mathrm{op}}
\leq
C_{K,N}.
\]
In particular,
~\Eqref{eq:polynomial-F-growth} and
\Eqref{eq:polynomial-F-derivative-growth} hold with linear gradient
growth and bounded Hessian.

The coercive quadratic example is recovered by setting
\[
\kappa_2=\lambda=0.
\]
If \(\kappa_2>0\), the objective couples all particles through their
empirical mean. Indeed,
\[
D_{x_ix_j}^2
\left[
\frac{N\kappa_2}{2}
\|\overline x-a\|^2
\right]
=
\frac{\kappa_2}{N}I,
\qquad
1\leq i,j\leq N.
\]
Thus the objective is no longer an additive sum of independent
single-particle costs. If \(\lambda>0\), the MMD term adds a second,
kernel-mediated interaction between every pair of particles. Thus this class contains genuinely interacting empirical-measure
objectives while satisfying the structural growth and regularity
conditions~\Eqref{eq:polynomial-F-growth}--~\Eqref{eq:polynomial-F-derivative-growth}.
Moreover, since the MMD term is bounded and the quadratic part of
\(F_{K,N}\) has Hessian bounded by $\kappa_1+\kappa_2$, a sufficient condition for the inverse-denominator condition~\Eqref{eq:local-inverse-denominator} is
\[
4\tau_{\max}
(\kappa_1+\kappa_2)
\|Q_{K,\mathrm{prop}}\|_{\mathrm{op}}
<
\varepsilon,
\qquad
\tau_{\max}:=\max_{0\leq m\leq M-1}\tau_m.
\]
Indeed, under this condition, for every \(R_0<\infty\),
\[
\sup_{\substack{
0\leq m\leq M-1\\
\|\mathbf x\|\leq R_0
}}
\mathbb E
\left[
\exp\left(
\frac{4F_{K,N}(\mathbf x+\mathbf U_m)}{\varepsilon}
\right)
\right]
<\infty.
\]
Since \(z\mapsto z^{-4}\) is convex, Jensen's inequality gives
\[
\left(
\overline B_{m,i}^{R,S}(\mathbf x)
\right)^{-4}
\leq
\frac1{RS}
\sum_{r=1}^R\sum_{s=1}^S
L_{m,i,r,s}(\mathbf x)^{-4},
\]
and hence~\Eqref{eq:local-inverse-denominator} holds uniformly in
\(R,S\). 

It remains to verify the chain moment condition
\Eqref{eq:uniform-p-moments-localized}. Since
\(\kappa_1>0\),
\[
F_{K,N}(\mathbf y)
\ge
\frac{\kappa_1}{2}
\sum_{i=1}^N\|y_i-a\|^2.
\]
Moreover, if \(E_j\ge0\) are any finite collection of energies and
\(w_j\propto e^{-E_j/\varepsilon}\), then for every nondecreasing
\(\phi\),
\[
\sum_j w_j\phi(E_j)
\le
\frac1n\sum_{j=1}^n\phi(E_j),
\]
because \(\phi(E_j)\) and \(e^{-E_j/\varepsilon}\) are oppositely
ordered. Taking \(\phi(z)=z^{p/2}\), using the preceding coercivity and
the quadratic upper bound on \(F_{K,N}\), gives, for every \(p\ge2\),
\[
\sup_{\substack{R,S\ge1\\0\le m\le M-1}}
\mathbb E_{\Xi_m^{R,S}}
\left[
\|\widehat b_m^{R,S}(\mathbf x)\|^p
\right]
\le
C_p(1+\|\mathbf x\|^p).
\]
The same argument for the population Gibbs ratio gives
\[
\|b(t_m,\mathbf x)\|^p
\le
C_p(1+\|\mathbf x\|^p).
\]
Since the execution noise is finite-dimensional Gaussian, the exact
and finite Euler recursions therefore imply, by induction over the
finite number of time steps,
\[
\sup_{R,S\ge1}\max_{0\le m\le M}
\mathbb E\|C_m^{R,S,h}\|^p
+
\max_{0\le m\le M}
\mathbb E\|X_m^h\|^p
<\infty
\]
whenever the initial cloud has finite \(p\)-th moment. In particular,
this holds for every \(p<\infty\) under the deterministic collapsed
initialization used above. Hence, under the displayed
inverse-denominator condition, all hypotheses of
Remark~\ref{rem:polynomial-growth-finite-implementation} are satisfied;
one may take any \(p>2\), since the terminal growth exponent is \(r=2\).
\end{exam}
\section{Experimental details and additional figures}
\label{app:experiments}
\subsection{Three controlled toy problems}
\subsubsection{Four-mode MMD sanity check}
Let $\nu^\star=\frac14\sum_{s_1,s_2\in\{-1,1\}}\delta_{(s_1,s_2)}$ be the uniform law on the four corners of the square. We optimize $G(\mu):=\operatorname{MMD}_k^2(\mu,\nu^\star)$,
where
$\operatorname{MMD}_k^2(\mu,\nu)
:=
\iint k(x,x')\,\mu(dx)\mu(dx')
-
2\iint k(x,y)\,\mu(dx)\nu(dy)
+
\iint k(y,y')\,\nu(dy)\nu(dy')$. We use a Gaussian kernel \(k\), which is characteristic, so the unique minimizer is
\(\mu^\star=\nu^\star\). We compare SCMO with Adam point and Particle Adam. Adam point optimizes a
single state and is evaluated as a degenerate point measure, whereas
Particle Adam applies Adam directly to the same \(N=128\) particles using
gradients of the empirical MMD objective. SCMO uses \(N=128\), \(R=1\), \(S=256\), \(M=256\), \(L=1\), \(T=1\), \(\varepsilon=10^{-10}\), proposal scale \(\sigma_{\mathrm{prop}}=1.0\), and execution scale \(\sigma_{\mathrm{dyn}}=0.15\). The Gaussian MMD kernel bandwidth is \(\ell=0.7\). 

Table~\ref{tab:toy-four-mode-mmd} reports results over \(50\) independent runs. Reported metrics
are the final \(\operatorname{MMD}^2\), the number of occupied modes, the minimum mass assigned
to any target mode, and the \(L^1\) error of the four-mode mass vector. Specifically, let \(z_1,\ldots,z_4\) denote the four target modes and define $B_\rho(z_m):=\{x\in\mathbb R^2:\|x-z_m\|\le \rho\}, \,\,p_m:=\mu^N(B_\rho(z_m))$, $m = 1,\dots,4$, and $\mu^N:=\frac1N\sum_{i=1}^N\delta_{x_i}$ denote the final empirical measure. Thus \(p_m\) is the empirical mass assigned to a radius-\(\rho\) neighborhood of mode \(z_m\).
We count mode \(m\) as occupied if \(p_m\ge0.05\), define the minimum mode mass as
\(\min_m p_m\), and define the \(L^1\) mass error by $\sum_{m=1}^4\left|p_m-\frac14\right|$. Adam point is included as a degenerate point-measure baseline, while
Particle Adam serves as a gradient-oracle reference because the MMD
objective is smooth. As expected, Particle Adam achieves the lowest MMD
value. SCMO uses only objective evaluations and still
recovers all four modes in every run, with minimum mode mass \(0.229\pm0.010\) and \(L^1\) mass
error \(0.056\pm0.027\), indicating a nearly balanced four-mode empirical law.

\begin{table}[H]
\centering
\small
\setlength{\tabcolsep}{4pt}
\caption{
Four-mode MMD sanity check over \(50\) independent runs. Values are mean \(\pm\) standard
deviation.
}
\label{tab:toy-four-mode-mmd}
\begin{tabular}{lcccc}
\toprule
Method
&
\makecell{\(\mathrm{MMD}^2\)\\}
&
\makecell{Modes\\}
&
\makecell{Min. mode\\mass}
&
\makecell{\(L^1\) mass\\error}
\\
\midrule
Adam point \(\delta_x\)
&
\(7.50{\times}10^{-1}\pm0.00\)
&
--
&
--
&
--
\\
Particle Adam
&
\(1.32{\times}10^{-3}\pm8.21{\times}10^{-4}\)
&
\(4.0\pm0.0\)
&
\(0.222\pm0.012\)
&
\(0.073\pm0.025\)
\\
SCMO
&
\(1.49{\times}10^{-2}\pm2.49{\times}10^{-3}\)
&
\(4.0\pm0.0\)
&
\(0.229\pm0.010\)
&
\(0.056\pm0.027\)
\\
\bottomrule
\end{tabular}
\end{table}
\begin{figure}[H]
    \centering
    \includegraphics[width=0.5\linewidth]{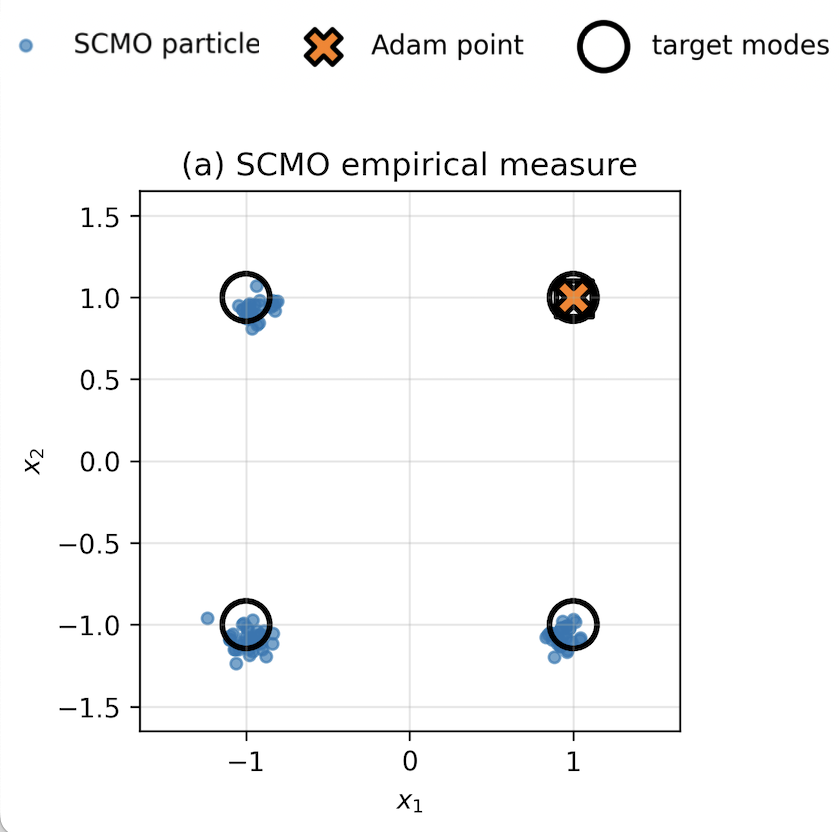}
    \caption{
    Representative SCMO empirical measure for the four-mode MMD sanity check.
    The target law is the uniform measure over the four marked modes \((\pm1,\pm1)\).
    Quantitative results over \(50\) runs are reported in Table~\ref{tab:toy-four-mode-mmd}.
    }
    \label{fig:toy-four-mode-mmd-cloud}
\end{figure}
\subsubsection{Asymmetric two-well escape}
We consider the two-dimensional asymmetric
two-well objective $
J(x)
=
c
-
\alpha_{\mathrm{near}}
\exp\left(-\frac{\|x-x_{\mathrm{near}}\|^2}{2\sigma^2}\right)
-
\alpha_{\mathrm{far}}
\exp\left(-\frac{\|x-x_{\mathrm{far}}\|^2}{2\sigma^2}\right)$, where
$
x_{\mathrm{near}}=(-1,0), x_{\mathrm{far}}=(1.8,0),
$
with \(c=0.30\), \(\alpha_{\mathrm{near}}=1.00\), \(\alpha_{\mathrm{far}}=1.35\), and
\(\sigma=0.22\). The nearby left well is easier to reach from the origin but is shallower, while the
farther right well has strictly lower objective value.

We run \(50\) independent trials. Adam is initialized from a small Gaussian perturbation around the
origin, with initialization scale \(0.05\), learning rate \(10^{-2}\), and \(2000\) steps. SCMO is
initialized from a collapsed particle cloud at the origin and uses \(N=128\), \(R=1\), \(S=128\), \(M=128\),
\(L=3\), horizon \(T=3\), temperature \(\varepsilon=10^{-10}\), proposal scale
\(\sigma_{\mathrm{prop}}=1.0\), and execution scale \(\sigma_{\mathrm{dyn}}=0.15\).

Table~\ref{tab:two-well-escape-results} shows that Adam consistently converges to the nearby
shallow well and never reaches the deeper basin. In contrast, SCMO transports almost all of its
empirical mass to the farther deep well: the final deep-well mass is \(0.973\pm0.014\), while the
shallow-well mass is only \(0.027\pm0.014\). The SCMO empirical law attains objective
\(-1.020\pm4.89\times10^{-3}\), and the best SCMO atom reaches \(-1.050\pm1.42\times10^{-4}\),
close to the bottom of the deeper well. This toy experiment illustrates long-range distributional
exploration: SCMO moves the particle law toward the globally better basin, whereas Adam commits
to the locally attractive shallow well.

\begin{table}[H]
\centering
\small
\setlength{\tabcolsep}{5pt}
\caption{
Asymmetric two-well escape benchmark over \(50\) independent runs.
Values are mean \(\pm\) standard deviation. Lower objective is better.
}
\label{tab:two-well-escape-results}
\begin{tabular}{lccc}
\toprule
Output
&
\makecell{Objective\\\(\downarrow\)}
&
\makecell{Deep-well\\mass/rate \(\uparrow\)}
&
\makecell{Shallow-well\\mass \(\downarrow\)}
\\
\midrule
Adam point
&
\(-0.700\pm5.86{\times}10^{-6}\)
&
\(0.000\pm0.000\)
&
\(1.000\pm0.000\)
\\
SCMO empirical law
&
\(-1.020\pm4.89{\times}10^{-3}\)
&
\(0.973\pm0.014\)
&
\(0.027\pm0.014\)
\\
\bottomrule
\end{tabular}
\end{table}

\begin{figure}[H]
    \centering
    \includegraphics[width=0.42\linewidth]{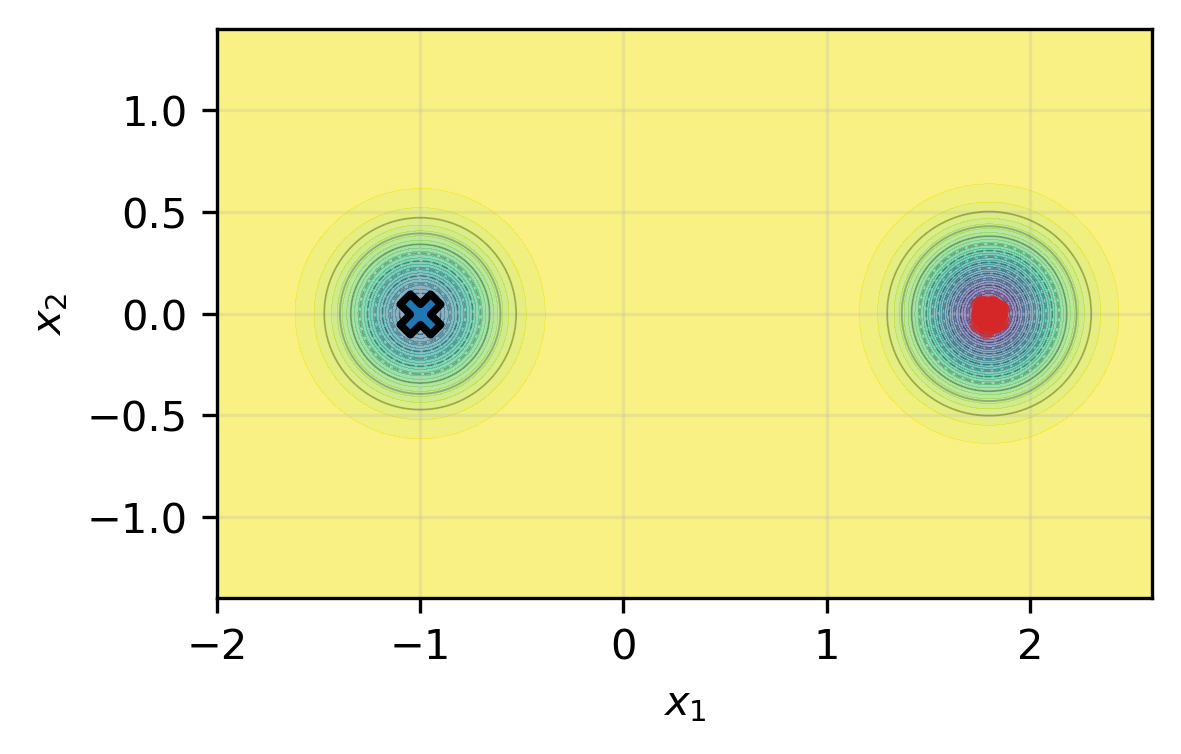}
    \caption{
    Representative final configuration for the asymmetric two-well escape toy. The background contours show the objective landscape, with darker regions indicating lower objective values. The left basin is closer to the initialization but shallower, while the right basin is farther but deeper. Red points are final SCMO particles, and the blue cross is the final Adam point. Adam remains in the nearby shallow well, whereas SCMO transports most of its empirical measure to the deeper well. Quantitative results over \(50\) runs are reported in Table~\ref{tab:two-well-escape-results}.
    }
    \label{fig:toy-escape-particles}
\end{figure}
\subsubsection{Nonsmooth plateau escape}

We use a simple two-dimensional nonsmooth objective to test whether SCMO can escape a flat region
where local first-order information is uninformative. The objective is
$
J(x)=1_{\{\|x\|\le r\}}
+
1_{\{\|x\|>r\}}
\min\left\{
\frac{\|x-c_1\|^2}{\beta},
\frac{\|x-c_2\|^2}{\beta}
\right\}$, with \(r=0.65\), \(\beta=6\), \(c_1=(-1.2,0.95)\), and \(c_2=(1.2,-0.95)\). We compare the performance of SCMO with Adam. The origin lies inside
the flat plateau, so Adam initialized at the origin receives no useful local gradient. SCMO is also
initialized from a collapsed cloud at the origin, but uses candidate sampling and exponential
reweighting to move particles toward the low-objective basins.

We run \(50\) independent trials. Adam uses learning rate \(10^{-2}\) and \(2000\) steps. SCMO uses
\(N=128\) particles, \(R=1\), \(S=128\) candidate replacements per particle, \(M=128\) inner steps,
\(L=3\) outer iterations, horizon \(T=3\), temperature \(\varepsilon=10^{-10}\), proposal scale
\(\sigma_{\mathrm{prop}}=1.0\), and execution scale \(\sigma_{\mathrm{dyn}}=0.15\).

Table~\ref{tab:nonsmooth-plateau-results} shows that Adam remains at the plateau value in every
run, while SCMO moves all particle mass outside the plateau and obtains a near-zero empirical
objective. The final SCMO cloud also splits approximately evenly between the two low-objective
basins. 
\begin{table}[H]
\centering
\small
\setlength{\tabcolsep}{5pt}
\caption{
Nonsmooth plateau escape over \(50\) independent runs. Values are mean \(\pm\) standard deviation.
For SCMO, ``mass outside plateau'' is the final empirical mass outside \(\{x:\|x\|\le r\}\).
}
\label{tab:nonsmooth-plateau-results}
\begin{tabular}{lccc}
\toprule
Output
&
\makecell{Objective\\}
&
\makecell{Escape / mass outside\\plateau}
&
\makecell{Mass in two\\wells}
\\
\midrule
Adam point
&
\(1.00\pm0.00\)
&
\(0.00\pm0.00\)
&
--
\\
SCMO empirical law
&
\(2.45{\times}10^{-4}\pm2.12{\times}10^{-5}\)
&
\(1.00\pm0.00\)
&
\(0.500\pm0.050\;/\;0.500\pm0.050\)
\\
\bottomrule
\end{tabular}
\end{table}

\begin{figure}[H]
    \centering
    \includegraphics[width=0.95\linewidth]{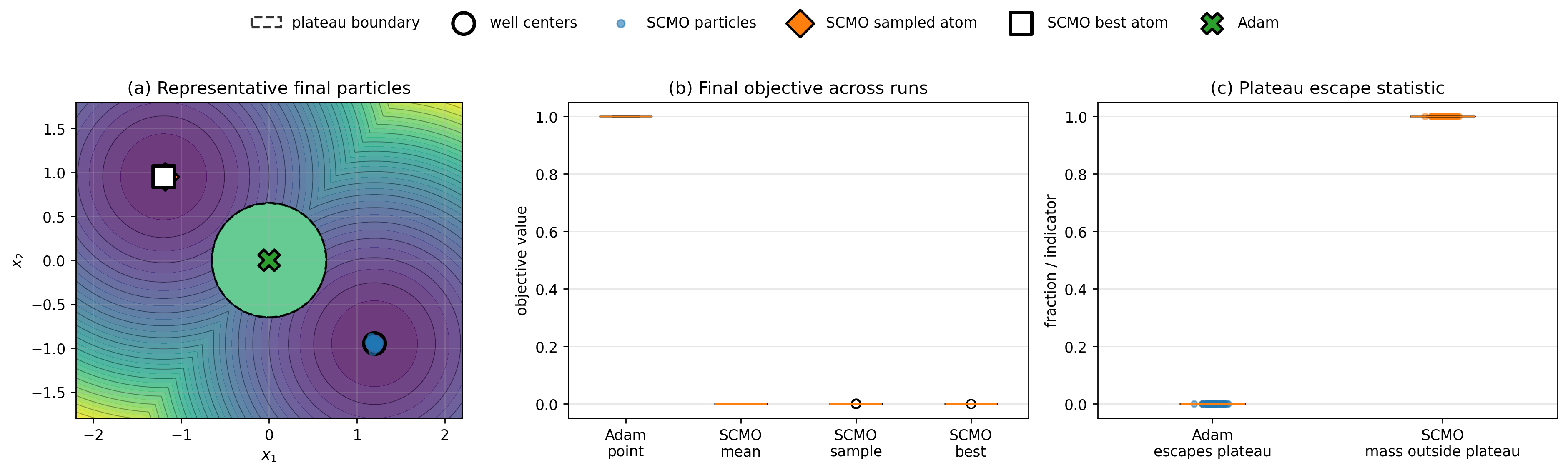}
    \caption{
    Nonsmooth plateau escape sanity check. 
    (a) Representative final configuration. Adam remains at the origin inside the flat plateau,
    while SCMO transports its particles into the low-objective basins.
    (b) Final objective values over \(50\) runs.
    (c) Plateau escape statistic: Adam is evaluated by whether the final point leaves the plateau,
    while SCMO is evaluated by final empirical mass outside the plateau.
    }
    \label{fig:nonsmooth-plateau}
\end{figure}
\subsection{Function-Space Optimization with PDE Energies}
\paragraph{PDE problem setting.}
We next evaluate SCMO on five function-space optimization problems derived from standard phase-field and variational energies. The suite is designed to test recovery of structured function-valued laws under smooth, nonsmooth, discontinuous, and metastable objectives.

Specifically, let \(H=L^2([0,1])\), and represent each candidate probability measure on \(H\) by an \(N\)-particle empirical law. Each particle is expanded using \(K=32\) elements of the orthonormal Neumann cosine basis
\[
e_0(x)=1,
\qquad
e_k(x)=\sqrt{2}\cos(k\pi x),
\quad k=1,\ldots,K-1.
\]
Thus, for the \(i\)-th particle,
\begin{align}
\label{u_c}
u_i(x)
=
\sum_{k=0}^{K-1}c_{ik}e_k(x)
=
c_{i0}
+
\sqrt{2}\sum_{k=1}^{K-1}c_{ik}\cos(k\pi x).
\end{align}
This representation automatically satisfies the homogeneous Neumann boundary conditions $u_i'(0)=u_i'(1)=0$, and its constant coefficient equals the spatial mean, $c_{i0}=\int_0^1u_i(x)\,\mathrm dx$. We evaluate the functions and their energies on the uniform grid $x_g=\frac{g}{127},
\,g=0,\ldots,127$, using trapezoidal quadrature.

The five problems isolate complementary optimization challenges. P1 is a smooth Allen--Cahn reference problem testing recovery of a balanced two-phase law supported near \(u\equiv-1\) and \(u\equiv+1\). P2 tests optimization when both the particlewise total-variation energy and the law-level absolute-moment penalty are nonsmooth. P3 tests simultaneous recovery of three separated phases under a discontinuous and interacting occupancy objective with target masses \((0.25,0.50,0.25)\). P4 tests escape from the prescribed shallow local phase \(u\equiv-1\) toward the unique global phase \(u\equiv+1\) in a smooth nonconvex landscape. Finally, P5 tests both escape from the metastable central phase \(u\equiv0\) and distributional splitting between the two outer global phases \(u\equiv\pm1\), with target masses \((0.5,0,0.5)\).

We now detail the explicit objective functional for each problem. Recall~\Eqref{u_c}, for a coefficient cloud $c=(c_1,\ldots,c_N)\in(\mathbb R^K)^N$, define its empirical function-valued law and constant-mode moments by
$\mu_c^N=\frac1N\sum_{i=1}^N\delta_{u_i},
\,m_q(c)=\frac1N\sum_{i=1}^N c_{i0}^{\,q}.
$ For each problem \(p\), the implemented optimization problem is
\[
\min_{c\in(\mathbb R^K)^N}G_p(\mu_c^N),
\]
where
\[
G_p(\mu_c^N)
=
\frac1N\sum_{i=1}^N E_p(u_i)+\Phi_p(c).
\]
Here, \(E_p\) is the particlewise function-space energy and \(\Phi_p\) is a law-level interaction term.

The five implemented objectives are as follows:
\begin{align*}
E_1(u)
&=
\frac{\eta_1}{2}\int_0^1 |u'(x)|^2\,\mathrm dx
+\frac{1}{4\eta_1}\int_0^1 \bigl(u(x)^2-1\bigr)^2\,\mathrm dx,
\\
\Phi_1(c)
&=5m_1(c)^2,
\qquad
\eta_1=0.05;
\\[1mm]
E_2(u)
&=
\eta_2\int_0^1 |u'(x)|\,\mathrm dx
+\frac{1}{4\eta_2}\int_0^1 \bigl(u(x)^2-1\bigr)^2\,\mathrm dx,
\\
\Phi_2(c)
&=5|m_1(c)|,
\qquad
\eta_2=0.04;
\\[1mm]
E_3(u)
&=
\frac{\eta_3}{2}\int_0^1 |u'(x)|^2\,\mathrm dx
+\frac{4}{\eta_3}\int_0^1
\bigl(u(x)+1\bigr)^2u(x)^2\bigl(u(x)-1\bigr)^2\,\mathrm dx,
\\
\Phi_3(c)
&\text{ defined below},
\qquad
\eta_3=0.06;
\\[1mm]
E_4(u)
&=
\frac{\eta_4}{2}\int_0^1 |u'(x)|^2\,\mathrm dx+
\frac{1}{\eta_4}\int_0^1
\left[
\frac14\bigl(u(x)^2-1\bigr)^2
-\beta_4\left(\frac32u(x)-\frac12u(x)^3\right)
+\beta_4
\right]\mathrm dx,
\\
\Phi_4(c)
&=0,
\qquad
(\eta_4,\beta_4)=(0.05,0.12);
\\[1mm]
E_5(u)
&=
\frac{\eta_5}{2}\int_0^1 |u'(x)|^2\,\mathrm dx
+
\frac{1}{\eta_5}\int_0^1
\left[
u(x)^2\bigl(u(x)^2-1\bigr)^2
+\beta_5\bigl(1-u(x)^2\bigr)^2
\right]\mathrm dx,
\\
\Phi_5(c)
&=
5\left\{
m_1(c)^2+\bigl[m_2(c)-1\bigr]^2
\right\},
\qquad
(\eta_5,\beta_5)=(0.05,0.02).
\end{align*}

For P3, let
\[
z=(-1,0,1),
\qquad
w=(0.25,0.50,0.25)
\]
denote the three target phases and their prescribed masses. Define the radius-gated empirical mass around phase \(z_j\) by
\[
q_j(c)
=
\frac1N\sum_{i=1}^N
\mathbf 1\!\left\{
|c_{i0}-z_j|\le 0.28
\right\},
\qquad j=1,2,3.
\]
Because the phase centers are separated by one and the gating radius is \(0.28\), the three neighborhoods are disjoint. The law-level interaction for P3 is
\[
\Phi_3(c)
=
2\sum_{j=1}^3|q_j(c)-w_j|
-
4\,\mathbf 1\!\left\{
q_1(c)\ge0.15,\;
q_2(c)\ge0.30,\;
q_3(c)\ge0.15
\right\}.
\]
Its first term penalizes deviations from the target phase masses, whereas the discontinuous reward tests whether an optimizer can satisfy all three occupancy requirements simultaneously.

\paragraph{Evaluation criteria.}
The primary performance metric is the implemented objective \(G_p\). We additionally report the structural success for P3, P4, P5. P3 is considered successful when all three occupancy requirements $q_1(c)\ge0.15,\,q_2(c)\ge0.30,\,q_3(c)\ge0.15$ are satisfied. P4 success requires at least \(0.90\) nearest-phase mass at the global phase \(+1\). P5 success requires at least \(0.90\) combined nearest-phase mass in the two outer phases and an \(\ell_1\) mass error of at most \(0.20\) relative to the target \((0.5,0,0.5)\).

\paragraph{Baselines.}
We compare SCMO with three complementary controls. Particle Adam applies Adam~\citep{kingma2015adam} jointly to all \(NK\) coefficients of the empirical cloud, providing a gradient-oracle baseline without restricting the particles to a parametric family. CEM cloud~\citep{deboer2005crossentropy} is a derivative-free whole-cloud optimizer: each candidate is a complete \(N\)-particle coefficient cloud sampled from a coordinatewise Gaussian proposal, whose mean and standard deviation are updated from the elite candidates. The diagonal-Gaussian baseline restricts the particle law to $c=m+\exp(\ell)\odot\varepsilon,\,\varepsilon\sim\mathcal N(0,I_K)$, and optimizes \(m\) and \(\ell\) using reparameterized Monte Carlo gradients and Adam~\citep{kingma2014autoencoding}. Fresh samples are used for gradient estimation, while a fixed \(N\)-sample common-random-number cloud is used to evaluate and select iterates; this cloud is constructed to reproduce the shared initial empirical law exactly. Particle Adam and the Gaussian family use the gradients supplied by the implemented objective; in particular, the hard occupancy indicators in P3 have zero gradient almost everywhere. These baselines test SCMO against direct empirical-cloud gradient optimization, parametric-law optimization, and derivative-free joint-cloud search.

\paragraph{Hyperparameters and budget accounting.}
All methods use \(K=32\) cosine coefficients, a 128-point spatial grid, \(N=256\) particles in each reported empirical cloud, and the same initial cloud for each task and random seed. SCMO additionally uses \(R=2\) context clouds, \(S=128\) replacement candidates per particle and context, \(M=128\) inner steps, time horizon \(T=2\), and Gibbs softmin temperature \(\varepsilon=10^{-10}\), giving an inner step size \(\Delta t=T/M=1/64\).

The task-specific settings were selected on development runs for each method's best performance and frozen before evaluation. The resulting settings are:
\begin{center}
\small
\begin{tabular}[H]{crrrrcr}
\toprule
Task
& \(L\)
& \(\sigma_{\rm prop}\)
& \(\sigma_{\rm dyn}\)
& Particle Adam lr
& CEM \(B/\rho/\sigma_0\)
& Gaussian lr\\
\midrule
P1 & 30  & 1.0 & 0.20 & 0.05 & \(32/.05/.08\)  & .08\\
P2 & 500 & 0.1 & 0.01 & 0.10 & \(32/.05/.08\)  & .02\\
P3 & 30  & 1.0 & 0.20 & 0.10 & \(256/.20/.35\) & .08\\
P4 & 50  & 0.4 & 0.01 & 0.10 & \(32/.05/.08\)  & .08\\
P5 & 30  & 1.0 & 0.20 & 0.10 & \(32/.05/.08\)  & .02\\
\bottomrule
\end{tabular}
\end{center}
Here, \(B\) is the number of complete \(N\)-particle clouds sampled in each CEM iteration, \(\rho\) is the elite fraction, and \(\sigma_0\) is the coordinatewise initial proposal standard deviation. CEM uses a minimum standard deviation of \(0.005\) and smoothing coefficient \(0.25\), where \(0.25\) is the weight assigned to the new elite estimate. Particle Adam and the diagonal-Gaussian family use at most \(3{,}000\) optimizer updates. Each Gaussian update uses \(512\) reparameterized Monte Carlo training samples, and its initial standard deviation is floored at \(0.1\); this floor is not imposed after optimization begins.

We exactly match the configured derivative-free search budgets of SCMO and CEM using particlewise energy evaluations as the common work unit. For task \(p\), one SCMO inner update evaluates \(RN\) context particles and \(RNS\) proposed replacements. Its configured search budget is therefore $\mathcal B_p^{\rm SCMO}
=
L_pMNR(S+1)$. If CEM samples \(B_p\) complete \(N\)-particle clouds per iteration, we set $I_p^{\rm CEM}
=
\left\lceil
\frac{L_pMR(S+1)}{B_p}
\right\rceil$. For every setting in the table, \(B_p\) divides \(L_pMR(S+1)\) exactly, and hence
$\mathcal B_p^{\rm CEM}
=
I_p^{\rm CEM}B_pN
=
\mathcal B_p^{\rm SCMO}$. The resulting CEM iteration counts for P1--P5 are, respectively,
$
30{,}960,\,
516{,}000,\,
3{,}870,\,
51{,}600,\,
30{,}960$. The corresponding matched configured search-particle budgets are
$
253{,}624{,}320,\,
4{,}227{,}072{,}000,\,
253{,}624{,}320,\,
422{,}707{,}200,\,
253{,}624{,}320$. These are task-specific configured maxima; in particular, P2 receives a substantially larger maximum because it uses \(L=500\). Particle Adam and the diagonal-Gaussian family use gradient oracles and are therefore not equated with the derivative-free particle-evaluation count. Both retain their native cap of \(3{,}000\) optimizer updates, and their objective histories had plateaued by their final updates across all tasks.

Plateau stopping can trigger only at or after \(60\%\) of each method's configured native budget and only when no meaningful improvement has occurred over \(20\%\) of that budget. An improvement is considered meaningful when it is at least
$
\max\!\left\{
10^{-4},
10^{-4}\max\bigl(1,|G_{\rm anchor}|\bigr)
\right\}$, where \(G_{\rm anchor}\) is the objective at the most recent meaningful improvement. 

\paragraph{Results.} 
The results are summarized in Table~\ref{tab:pde-five-regime}. SCMO has the lowest mean objective on P1--P3 and the second-lowest on P4--P5.  Its
descriptive mean within-seed rank is 1.45, and it is among the top two methods
in all 100 task--seed comparisons.  On P3, SCMO encounters and returns a cloud
satisfying all three quotas in 20/20 trials, whereas every baseline does so in
0/20.  On P4, SCMO and the Gaussian family both reach the global phase in
20/20 trials, with the Gaussian family attaining the lower objective.  On P5,
SCMO alone satisfies the fixed balanced-global criterion in 20/20
trials, although Particle Adam attains the lower objective. Together, these results demonstrate SCMO's strong performance across smooth,
nonsmooth nonconvex, discontinuous interacting, and metastable function-space
objectives, particularly when success requires recovering a prescribed
multiphase law rather than merely attaining a low scalar objective.
Figure~\ref{fig:pde-representative-runs} shows deterministically selected
median-objective SCMO clouds for all five tasks.
\begin{table}[H]
\label{pde_result}
\centering
\footnotesize
\setlength{\tabcolsep}{4.7pt}
\caption{Five function-space tasks, with 20 paired seeds per task. Panel A reports the objective (mean $\pm$ standard deviation; lower is better); bold and underline mark the best and second-best means. Panel B reports strict structural success.}
\label{tab:pde-five-regime}
\begin{tabular}{lrrrr}
\toprule
\multicolumn{5}{l}{\textbf{A. Objective value}}\\
Task & SCMO & Particle Adam & CEM cloud & Gaussian family\\
\midrule
P1 Smooth Allen--Cahn & \bm{$0.128\pm0.003$} & \underline{$0.184\pm0.017$} & $1.734\pm0.042$ & $0.996\pm0.211$\\
P2 TV--Allen--Cahn & \bm{$0.984\pm0.028$} & $1.747\pm0.031$ & \underline{$1.069\pm0.020$} & $1.118\pm0.279$\\
P3 Hard three-phase quotas & \bm{$-2.541\pm0.635$} & $2.174\pm0.034$ & $2.835\pm0.083$ & \underline{$2.009\pm0.001$}\\
P4 Tilted local/global & \underline{$0.00883\pm0.00008$} & $3.846\pm0.157$ & $5.654\pm0.007$ & \bm{$0.00109\pm0.00010$}\\
P5 Metastable balanced law & \underline{$0.342\pm0.004$} & \bm{$0.264\pm0.041$} & $5.263\pm0.131$ & $5.271\pm0.098$\\
\midrule
Mean within-seed rank $\downarrow$ & 1.45 & 2.60 & 3.57 & 2.38\\
\midrule
\multicolumn{5}{l}{\textbf{B. Strict structural success}}\\
P3 quota success & 20/20 & 0/20 & 0/20 & 0/20\\
P4 global-phase success & 20/20 & 0/20 & 0/20 & 20/20\\
P5 balanced-global success & 20/20 & 0/20 & 0/20 & 0/20\\
\bottomrule
\end{tabular}
\end{table}
\begin{figure}[H]
    \caption{Representative runs for function-space optimization with PDE energies}
    \centering
    \includegraphics[width=1\linewidth]{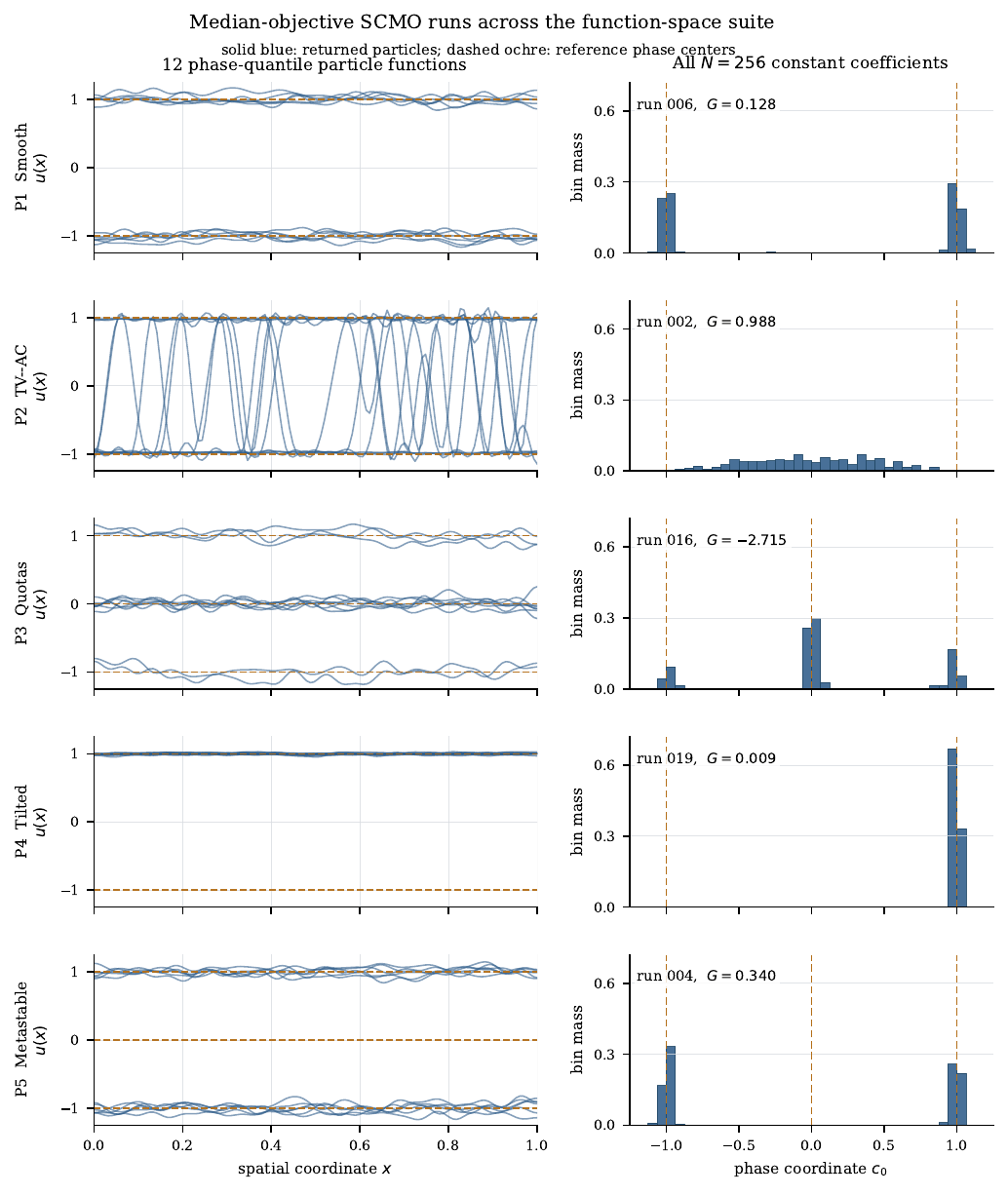}
        \label{fig:pde-representative-runs}

\end{figure}
\subsection{Trajectory-Law Optimization Through Narrow Passages}
\label{sec:trajectory-law-experiments}

\paragraph{Trajectory problem setting.}
We next consider open-loop trajectory optimization in planar environments with
obstacles.  A particle is a sequence of $H=42$ two-dimensional controls,
$a_i=(a_{i,0},\ldots,a_{i,H-1})\in(\mathbb R^2)^H$, and induces the discrete
trajectory
\[
    x_{i,0}=x_{\mathrm{start}},\qquad
    x_{i,t+1}=x_{i,t}+\Delta t\,a_{i,t},
    \qquad \Delta t=0.15.
\]
The optimization variable is the empirical law
\(\mu_a^N=N^{-1}\sum_{i=1}^N\delta_{a_i}\), represented by \(N=128\)
trajectories. Tasks T1--T4 start at \((-4.8,0)\), whereas T5 starts at
\((-4.8,0.75)\). In every task, the trajectories must pass through one or two
openings in a rectangular wall before reaching a circular goal.
Table~\ref{tab:trajectory-geometry} records the task geometry; ``strip''
denotes an additional rectangular pre-wall obstacle.

\begin{table}[htbp]
\centering
\small
\setlength{\tabcolsep}{4.2pt}
\caption{Geometry of the five trajectory-law tasks. Each interval in the gap
column denotes a vertical opening in the wall. The start is \((-4.8,0)\) for
T1--T4 and \((-4.8,0.75)\) for T5.}
\label{tab:trajectory-geometry}
\begin{tabular}{llclll}
\toprule
Task & Goal (tolerance) & Wall \(x\)-range & Wall \(y\)-range
& Gap interval(s) & Pre-wall strip \\
\midrule
T1 & \((4.6,0)\) \((0.50)\) & \([-0.55,0.55]\) & \([-3.2,3.2]\)
   & \([-0.80,0.80]\) & none \\
T2 & \((4.6,0.55)\) \((0.34)\) & \([-0.90,0.90]\) & \([-3.2,3.2]\)
   & \([1.16,1.39]\) & \([-2.55,-1.00]\times[-0.40,0.40]\) \\
T3 & \((4.6,0)\) \((0.38)\) & \([-0.75,0.75]\) & \([-3.2,3.2]\)
   & \([1.05,1.45]\), \([-1.45,-1.05]\) & none \\
T4 & \((4.6,0)\) \((0.38)\) & \([0.12,1.48]\) & \([-3.5,3.5]\)
   & \([1.38,1.70]\), \([-1.70,-1.38]\) & none \\
T5 & \((4.6,0.75)\) \((0.50)\) & \([-0.40,0.40]\) & \([-3.4,3.4]\)
   & \([0.25,1.10]\), \([-0.30,0.15]\) & none \\
\bottomrule
\end{tabular}
\end{table}

The five tasks isolate distinct optimization regimes. T1 is a wide single-passage reference problem with a feasible initialization. T2 narrows the opening and tests precise local feasibility. T3 has reflection-symmetric upper and lower passages together with continuous-in-mass law-level terms that reward successful mass, progress along both routes, and route balance; it therefore tests multimodal-law recovery without a discontinuous route-coverage bonus. T4 retains the symmetric two-route geometry but introduces discontinuous success and route-quota rewards, thereby testing an interacting, discontinuous objective. T5 tests escape from a feasible but inferior lower-route basin.

\paragraph{Trajectory objective.}
For task \(p\), let \([x_p^-,x_p^+]\times[y_p^-,y_p^+]\) be the
rectangular wall and let \(\mathcal I_p\) be the union of its enabled vertical
gap intervals, as specified in Table~\ref{tab:trajectory-geometry}.  The solid
part of the wall is
\[
\mathcal W_p
=
\left\{
(x,y):
x\in[x_p^-,x_p^+],\
y\in[y_p^-,y_p^+],\
y\notin\mathcal I_p
\right\}.
\]
Let \(\mathcal S_p\) denote the additional pre-wall rectangular strip on T2
and set \(\mathcal S_p=\varnothing\) on T1 and T3--T5.  The complete obstacle
set is
\[
\mathcal O_p=\mathcal W_p\cup\mathcal S_p.
\]
Let \(g_p\) and \(r_p\) denote the goal center and tolerance in task \(p\).
For particle \(i\), define
\begin{align*}
c_i
&=
\mathbf 1\!\left\{
x_{i,t}\in\mathcal O_p
\text{ for some }t\in\{1,\ldots,H\}
\right\},\\
d_i
&=
\left\lVert x_{i,H}-g_p\right\rVert_2^2,\\
e_i
&=
\sum_{t=0}^{H-1}\left\lVert a_{i,t}\right\rVert_2^2,\\
v_i
&=
\sum_{t=0}^{H-2}
\left\lVert a_{i,t+1}-a_{i,t}\right\rVert_2^2,\\
s_i
&=
\mathbf 1\!\left\{
c_i=0,\
\left\lVert x_{i,H}-g_p\right\rVert_2\le r_p
\right\}.
\end{align*}

For any scalar particle statistic \(z_i\), write
\(\bar z=N^{-1}\sum_i z_i\). On the two-route tasks, each successful
trajectory is assigned an upper or lower label according to its interpolated
height at the wall midpoint. The unconditional successful-route masses are
\begin{align*}
p_u
&=
\frac1N\sum_{i=1}^N
\mathbf 1\!\left\{
s_i=1,\ i\text{ uses the upper gap}
\right\},\\
p_l
&=
\frac1N\sum_{i=1}^N
\mathbf 1\!\left\{
s_i=1,\ i\text{ uses the lower gap}
\right\}.
\end{align*}
Thus, a failed or unclassified trajectory contributes to neither route mass.
For T3--T4, \(h_i\) is the mean squared vertical distance from the trajectory
to the interior of its nearest enabled passage, evaluated at five equally
spaced cross-sections of the wall. Each side of a passage is inset by \(20\%\)
of its width before this distance is computed. This shaping term is invariant
under vertical reflection and contains no preferred route sign.

The implemented task objective is
\begin{align}
G_p(\mu_a^N)
={}&
\lambda_{\mathrm{col},p}\bar c
+\lambda_{\mathrm{g},p}\bar d
+\lambda_{\mathrm{e},p}\bar e
+0.012\bar v
+w_{h,p}\bar h
-w_{s,p}\bar s
\nonumber\\
&-w_{0,p}\mathbf 1\!\left\{\bar s\ge c_{0,p}\right\}
\nonumber\\
&-w_{c,p}\mathbf 1\!\left\{
p_u\ge q_{u,p},\ p_l\ge q_{l,p}
\right\}
\nonumber\\
&-w_{j,p}\min(p_u,q_{u,p})
-w_{j,p}\min(p_l,q_{l,p})
+w_{b,p}(p_u-p_l)^2,
\label{eq:trajectory-cloud-objective}
\end{align}
where terms with zero weights are omitted. The task-specific values are
\begin{center}
\small
\setlength{\tabcolsep}{3.2pt}
\begin{tabular}{crrrrrrrrrr}
\toprule
Task
& \(\lambda_{\rm col}\)
& \((\lambda_{\rm g},\lambda_{\rm e})\)
& \(c_0\)
& \(w_0\)
& \(w_s\)
& \(w_c\)
& \(w_j\)
& \(w_b\)
& \(w_h\)
& \((q_u,q_l)\) \\
\midrule
T1 & 6 & \((1.1,.02)\) & .20 & 1 & 0   & 0    & 0 & 0   & 0 & \((0,0)\) \\
T2 & 6 & \((1.1,.02)\) & .25 & 1 & 0   & 0    & 0 & 0   & 0 & \((0,0)\) \\
T3 & 6 & \((1.1,.02)\) & --  & 0 & 1.5 & 0    & 1 & .75 & 1 & \((.20,.20)\) \\
T4 & 6 & \((1.1,.02)\) & .45 & 1 & 0   & 1.25 & 1 & .35 & 1 & \((.15,.15)\) \\
T5 & 7 & \((5.0,.04)\) & .30 & 1 & 0   & 0    & 0 & 0   & 0 & \((0,0)\) \\
\bottomrule
\end{tabular}
\end{center}
In particular, T5 increases the terminal-error and control-effort weights to
\(5.0\) and \(0.040\), respectively, while setting every route-specific
objective weight to zero. Its objective therefore does not explicitly favor
the upper or lower route. T5 initializes every method around the same feasible but inferior lower-route
template, whose wall-crossing height is \(y=-0.075\), while the start and goal
at \(y=0.75\) are aligned with the more direct upper passage; it therefore
tests escape from the initial route basin under route-neutral proposals.

\paragraph{Evaluation criteria.}
The primary metric is the implemented objective \(G_p(\mu_a^N)\), for which
lower is better. We additionally report successful mass \(\bar s\),
all-particle success, and the following task-specific structural criteria:
\(\bar s\ge .20\) on T1; \(\bar s\ge .25\) on T2;
\(\bar s=1\) and \(p_u,p_l\ge .20\) on T3;
\(\bar s\ge .45\) and \(p_u,p_l\ge .15\) on T4; and
\(\bar s\ge .30\) and \(p_u\ge .90\) on T5. The T3 requirement deliberately
combines full success with two-route mass, so merely placing particles near
both gaps is insufficient. The T5 upper-route requirement is an evaluation
criterion rather than an optimized route reward: \(G_5\) contains no term
depending on \(p_u\) or \(p_l\). It therefore measures whether the optimizer
escapes the initial lower-route basin and discovers the geometrically
preferred upper passage without being directly rewarded for its route label.

\paragraph{Baselines.}
We compare SCMO with three derivative-free controls, all initialized from the same empirical cloud within each task and seed and all using route-neutral proposals. Greedy whole-cloud random search~\citep{solis1981randomsearch} samples batches of complete \(N\)-trajectory perturbation clouds around the incumbent and accepts the lowest-objective proposal only when it improves \(G_p\). Whole-cloud CEM~\citep{deboer2005crossentropy} maintains a coordinatewise Gaussian proposal over the complete empirical cloud, evaluates \(G_p\) for every sampled cloud, and updates its mean and standard deviation from the elite clouds. The two-component MPPI-style optimizer~\citep{williams2017mppi} maintains two Gaussian distributions over open-loop control sequences. Each component is updated using Gibbs weights computed from point-trajectory costs; equal numbers of trajectories are then sampled from the two components to form an \(N\)-particle cloud, which is evaluated and selected using \(G_p\). This method is an open-loop MPPI-style law optimizer rather than a canonical receding-horizon controller. These controls test SCMO against greedy whole-cloud search, elite-based joint-cloud adaptation, and Gibbs-weighted parametric trajectory sampling.

\paragraph{Hyperparameters and budget accounting.}
All SCMO runs use $N=128$, $R=4$, $S=96$, $M=90$, $L=13$, algorithmic horizon
$T_{\rm alg}=1$, $\varepsilon=10^{-4}$, and covariance-coordinate decay $0.035$.
The task-specific proposal and diffusion scales are $(1,.20)$ on T1--T2,
$(.60,.035)$ on T3--T4, and $(.30,.01)$ on T5. The task-specific settings were selected on development runs for each method's best performance and frozen before evaluation. The resulting settings are reported in Table~\ref{tab:trajectory-hyperparameters}.
\begin{table}[H]
\centering
\scriptsize
\setlength{\tabcolsep}{3.1pt}
\caption{Task-specific trajectory optimizer parameters.  CEM entries give
initial $(x,y)$ standard deviations, elite fraction, and smoothing.  MPPI
entries give $(x,y)$ noise standard deviations, temperature, and mean
smoothing.}
\label{tab:trajectory-hyperparameters}
\begin{tabular}{cccccccc}
\toprule
Task & SCMO $(\sigma_{\rm prop},\sigma_{\rm dyn})$
& RS batch & RS scale
& CEM $(\sigma_x,\sigma_y)$/elite/smooth
& MPPI $(\sigma_x,\sigma_y)$/temp/smooth
 \\
\midrule
T1 & $(1,.20)$     & 12 & .0025 & $(.10,.10)$/.20/.70 & $(.03,.03)$/.01/.25 \\
T2 & $(1,.20)$     & 12 & .0100 & $(.10,.10)$/.20/.50 & $(.15,.80)$/.08/.25\\
T3 & $(.60,.035)$  & 12 & .0250 & $(.05,.80)$/.10/.50 & $(.10,.20)$/.02/.25\\
T4 & $(.60,.035)$  & 12 & .0250 & $(.05,.80)$/.10/.50 & $(.10,.40)$/.04/.25\\
T5 & $(.30,.010)$  & 53 & .0050 & $(.03,.12)$/.10/.50 & $(.03,.15)$/.02/.25\\
\bottomrule
\end{tabular}
\end{table}

Whole-cloud CEM uses population 192, re-inflates every 25 iterations by $1.25$,
and has a minimum standard deviation of $.03$ on T1--T2 and $.02$ on T3--T5.
The two-component MPPI-style optimizer uses 384 sampled trajectories per
component, standard-deviation smoothing $.15$ on T1--T4 and $.10$ on T5,
re-inflation every 20 iterations by $1.15$, and minimum $(x,y)$ standard
deviations $(.03,.05)$ on T1--T2, $(.02,.02)$ on T3--T4, and $(.02,.03)$ on
T5. Random search uses 37,928 updates on T1--T4 and 8,588 on T5.

We match the maximum number of simulated trajectories rather than nominal
iterations.  SCMO's exact configured cap is $B_{\rm traj}
=N+LMN\{R(S+1)+1\}
=58{,}256{,}768$ trajectory rollouts, and the baseline iteration ceilings above are chosen so
that every method has this same cap; the last atomic batch is shortened when
necessary.  Budget-aware plateau stopping is shared across methods.  It can
activate only after $40\%$ of the configured cap and requires no significant
incumbent improvement for $15\%$ of the cap, with improvement threshold
$\max\{10^{-4},10^{-4}\max(1,|G_{\rm anchor}|)\}$.  

\paragraph{Results.}
The results are summarized in Table~\ref{tab:trajectory-results}. SCMO attains
the lowest mean objective on T2--T5, including the explicitly law-valued
two-route tasks T3 and T4. It satisfies the strict T3 and T4 structural
criteria in all \(20\) runs; Mixture MPPI satisfies the T4 criterion in
\(14/20\) runs, and no baseline satisfies the full T3 criterion. On the narrow
T2 task, SCMO reaches mean successful mass \(.996\pm.006\), compared with
\(.516\pm.053\) for Mixture MPPI and zero for Random Search and CEM. On T5,
SCMO attains the lowest mean objective, \(2.7152\pm0.0001\); SCMO and Mixture
MPPI satisfy the strict upper-route criterion in all \(20\) runs, whereas
Random Search and CEM satisfy it in none. The descriptive mean within-seed
ranks are \(1.50\) for SCMO, \(2.20\) for Mixture MPPI, \(2.90\) for CEM, and
\(3.40\) for Random Search. Together, these results show that SCMO is reliable
on the easy reference problem and performs strongly on narrow-feasibility,
multimodal-law, and discontinuous interacting objectives, while also
discovering the geometrically better route from an inferior initialization
under route-neutral proposals.

\begin{table}[htbp]
\centering
\footnotesize
\setlength{\tabcolsep}{4.6pt}
\caption{Five trajectory-law tasks with 20 paired seeds per task.  Panel A
reports the objective (mean $\pm$ sample standard deviation; lower is better);
bold and underline mark the best and second-best means.  Panel B reports the
strict task-specific structural criterion.}
\label{tab:trajectory-results}
\begin{tabular}{lrrrr}
\toprule
\multicolumn{5}{l}{\textbf{A. Objective value}}\\
Task & SCMO & Random search & CEM cloud & Mixture MPPI \\
\midrule
T1 Easy passage
& $0.886\pm0.002$ & $\bm{0.846\pm0.000}$ & $0.886\pm0.002$
& $\underline{0.848\pm0.000}$ \\
T2 Narrow passage
& $\bm{1.097\pm0.028}$ & $7.900\pm0.002$ & $7.853\pm0.000$
& $\underline{3.188\pm0.416}$ \\
T3 Two-route multimodal
& $\bm{0.080\pm0.006}$ & $8.005\pm0.252$ & $3.930\pm0.322$
& $\underline{1.360\pm0.392}$ \\
T4 Hard route balance
& $\bm{-0.482\pm0.008}$ & $8.448\pm0.128$ & $6.628\pm0.149$
& $\underline{1.073\pm0.777}$ \\
T5 Local/global routes
& $\bm{2.7152\pm0.0001}$
& $2.8569\pm0.0010$
& $\underline{2.8293\pm0.0008}$
& $2.8403\pm0.0038$ \\
\midrule
Mean within-seed rank $\downarrow$ & 1.50 & 3.40 & 2.90 & 2.20 \\
\midrule
\multicolumn{5}{l}{\textbf{B. Strict structural success}}\\
T1: $\bar s\ge .20$ & 20/20 & 20/20 & 20/20 & 20/20 \\
T2: $\bar s\ge .25$ & 20/20 & 0/20 & 0/20 & 20/20 \\
T3: $\bar s=1$, $p_u,p_l\ge .20$ & 20/20 & 0/20 & 0/20 & 0/20 \\
T4: $\bar s\ge .45$, $p_u,p_l\ge .15$ & 20/20 & 0/20 & 0/20 & 14/20 \\
T5: $\bar s\ge .30$, $p_u\ge .90$
& 20/20 & 0/20 & 0/20 & 20/20 \\
\bottomrule
\end{tabular}
\end{table}

Figure~\ref{fig:trajectory-representative-scmo} displays the SCMO cloud whose
objective is closest to the median SCMO objective for each task, with ties
resolved deterministically.  Each panel plots 40 trajectories selected at
evenly spaced wall-crossing order statistics from the returned 128-particle
cloud.  In particular, the T3 and T4 panels show both routes without any
endpoint-preserving route proposal. 

\begin{figure}[htbp]
\centering
\includegraphics[width=\textwidth]{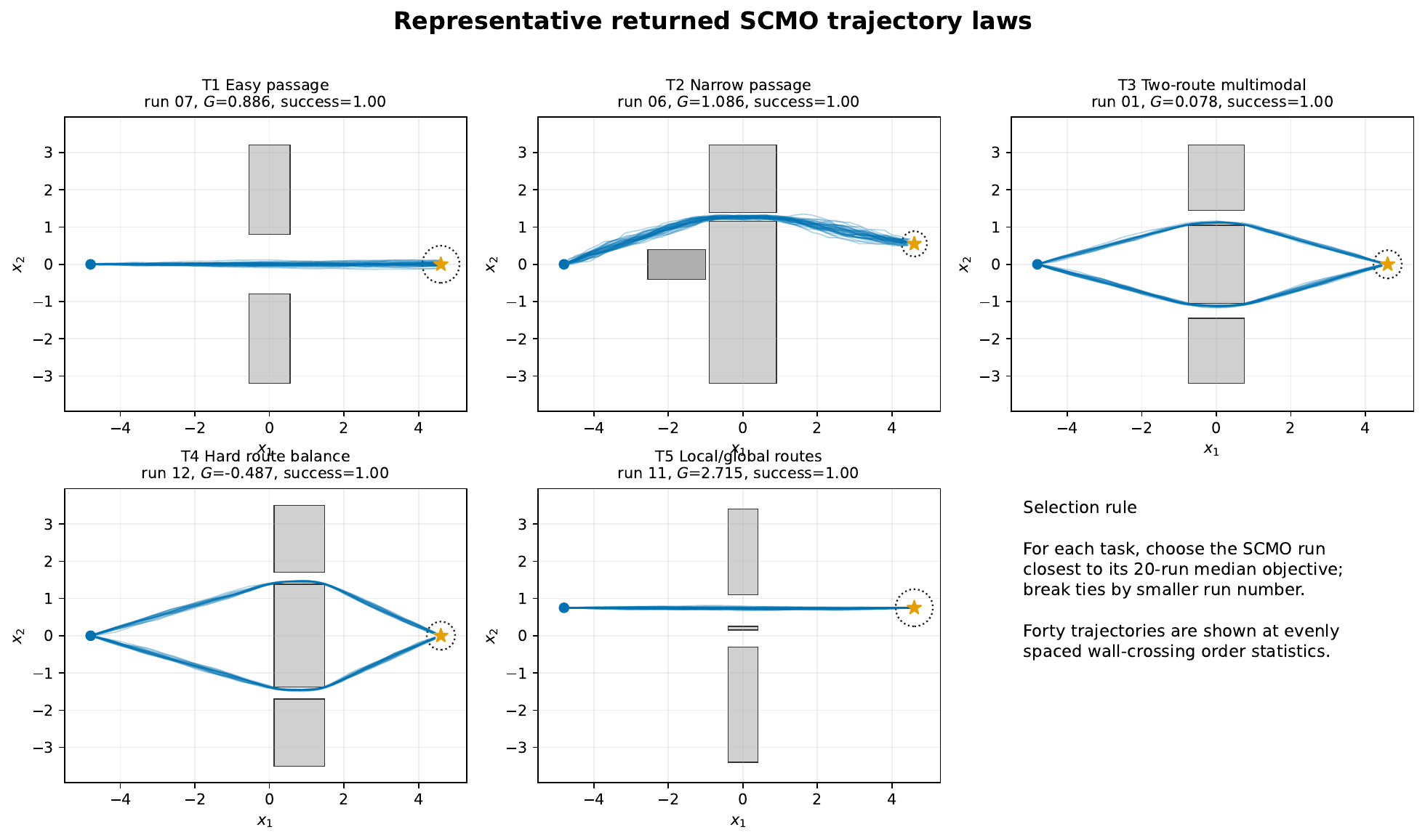}
\caption{Deterministically selected median-objective SCMO clouds for T1--T5.
Each panel shows 40 trajectories selected at evenly spaced wall-crossing order
statistics from the returned $N=128$ cloud, together with the wall, enabled
gaps, and goal.}
\label{fig:trajectory-representative-scmo}
\end{figure}

\subsection{Push-T as a Contact-Rich External Benchmark}
\label{sec:pusht-law-experiment}

\paragraph{Push-T problem setting.}
To complement the synthetic diagnostics, we evaluate on the contact-rich
Push-T manipulation environment used in Diffusion Policy
\citep{chi2023diffusionpolicy}.  In Push-T, a circular agent must push a
T-shaped block into a fixed T-shaped target region centered at \((256,256)\);
a rollout succeeds when the block covers more than \(95\%\) of the target
region.  We study fixed-state, open-loop optimization of a law over complete
pushing plans, rather than the canonical policy-learning protocol or a
learned-policy evaluation. For particle \(i\), reshape the optimization variable
\(z_i\in\mathbb R^{32}\) into 16 two-dimensional vectors
\(z_{i,0},\ldots,z_{i,15}\in\mathbb R^2\). At the equally spaced times $\tau_j=\frac{299j}{15},\, j=0,\ldots,15$, these vectors define bounded target-position knots
\[
q_{i,j}
=
5+502\,\operatorname{sigmoid}(z_{i,j})
\in[5,507]^2,
\qquad
\operatorname{sigmoid}(r)=\frac{1}{1+e^{-r}},
\]
where the sigmoid is applied coordinatewise. A shape-preserving piecewise
cubic Hermite interpolant is then fitted separately to the \(x\)- and
\(y\)-coordinates of the 16 knots and evaluated at
\(t=0,\ldots,299\). This produces the open-loop action sequence $a_i=(a_{i,0},\ldots,a_{i,299}),
\,
a_{i,t}\in[5,507]^2$, where \(a_{i,t}\) is the absolute target position sent to the circular pushing
agent at step \(t\). The interpolation does not overshoot the range of adjacent
knots, so every action remains within the environment bounds without
subsequent clipping. This mapping is fixed and contains no learned
parameters. The optimized empirical law is
\[
\mu_z^N=\frac1N\sum_{i=1}^N\delta_{z_i},
\qquad N=8,
\]
and therefore represents a distribution over eight open-loop pushing plans.

The simulator state consists of the center position
\(p_{\rm push}=(x_{\rm push},y_{\rm push})\) of the circular pusher and the
planar pose \((p_T,\theta_T)\) of the movable T-shaped block, where
\(p_T=(x_T,y_T)\) is the block center and \(\theta_T\) is its orientation in
radians. Coordinates are measured in pixels in the \(512\times512\) workspace.
We use the exactly serialized reset state
\[
\bigl(x_{\rm push},y_{\rm push},x_T,y_T,\theta_T\bigr)
=
(450,62,340,172,\pi/4).
\]
Thus, the pusher initially has center \((450,62)\), while the T-shaped block
has center \((340,172)\) and orientation \(\pi/4\), and goal center being $(256,256)$.  Five deterministic conditions are obtained by
scaling the agent and block positions about the goal by
$\gamma\in\{.85,.90,1.00,1.05,1.10\}$; the initial action center is scaled in
the same way while a frozen residual cloud is preserved.  These are controlled
symmetric perturbations of one state, not independent random task draws. Within each fixed condition, every
method receives the same serialized reset state and the same serialized
$N=8$ initial latent cloud.

\paragraph{Push-T objective.}
For a decoded action sequence $a=(a_0,\ldots,a_{299})$, let $r_{\max}(a)$ be
the maximum Push-T reward attained during its rollout and define
\begin{align*}
 \mathcal L(a)&=\sum_{t=1}^{299}\|a_t-a_{t-1}\|_2,\\
 \mathcal A(a)&=\sum_{t=2}^{299}
 \|a_t-2a_{t-1}+a_{t-2}\|_2.
\end{align*}
Let $I(a)$ indicate an invalid rollout summary.  The implemented particle loss
is
\begin{equation}
 \ell(a)=1-r_{\max}(a)
 +0.02\frac{\mathcal L(a)}{512}
 +0.01\frac{\mathcal A(a)}{512}
 +10I(a).
 \label{eq:pusht-particle-loss}
\end{equation}
Here \(I(a)=1\) if the rollout summary contains a non-finite reward or coverage, a non-finite or negative \(\mathcal L(a)\) or \(\mathcal A(a)\), or an unparseable success flag, and \(I(a)=0\) otherwise; invalid numerical fields are set to zero when computing \(\ell(a)\). We use the official \texttt{gym-pusht==0.1.6} success criterion. Goal
coverage is the fraction of the fixed T-shaped goal region overlapped by the
movable T-shaped block, and a rollout succeeds when this coverage strictly
exceeds \(0.95\). Each rollout ends upon success or after the 300-step horizon.
To classify the pushing route, we consider the first pusher--block contact
located more than 10 pixels to either side of the directed line from the
initial block center to the goal center. Contacts within 10 pixels of this
line are treated as ambiguous, and classification continues until a later
contact leaves this central band. A qualifying contact on the positive side
is labelled left, whereas one on the negative side is labelled right. Let \(s_i\in\{0,1\}\) indicate rollout success and let \(\rho_i\)
denote its left, right, or unclassified route label.
Define the unconditional masses by
\begin{align*}
 p_s&=\frac1N\sum_{i=1}^N s_i,\\
 p_L&=\frac1N\sum_{i=1}^N \mathbf 1\{s_i=1,\rho_i=L\},\\
 p_R&=\frac1N\sum_{i=1}^N \mathbf 1\{s_i=1,\rho_i=R\}.
\end{align*}
Failures contribute to neither route mass.  Successful rollouts whose contact
remains inside the dead zone count toward $p_s$ but toward neither $p_L$ nor
$p_R$.  The empirical-law objective is
\begin{align}
 G_{\rm PushT}(\mu_z^N)
 ={}&\frac1N\sum_{i=1}^N\ell(a_i)
 +(p_L-.5)^2+(p_R-.5)^2
 \nonumber\\
 &+2\bigl[(.25-p_L)_+^2+(.25-p_R)_+^2\bigr].
 \label{eq:pusht-cloud-objective}
\end{align}

\paragraph{Evaluation criteria.}
The primary metric is $G_{\rm PushT}$, and lower is better.  We additionally report successful mass,
the both-quota event $p_L,p_R\ge .25$, the both-route event
$p_L,p_R\ge .125$, and all-particle success $p_s=1$.  Because $N=8$, the
both-route threshold requires at least one successful trajectory on each
classified route.  Route balance and quotas are explicit terms of~\Eqref{eq:pusht-cloud-objective}; they are therefore target-aligned
diagnostics rather than independent downstream outcomes.

\paragraph{Baselines.}
We compare SCMO with four derivative-free whole-cloud baselines.  For every
baseline, one candidate is a complete \(N\times32\) latent cloud and is scored
using the same empirical-law objective \(G_{\rm PushT}\).  Greedy whole-cloud
random search~\citep{solis1981randomsearch} samples isotropic Gaussian
perturbations around the incumbent cloud and retains the best proposal only
when it improves the objective.  Whole-cloud CEM
\citep{deboer2005crossentropy} maintains a coordinatewise Gaussian proposal
over complete clouds and updates its mean and standard deviation using the
elite candidates.  Our MPPI-style baseline~\citep{williams2017mppi} samples
complete-cloud perturbations and updates the nominal cloud using their
Gibbs-weighted average; it is an open-loop empirical-law optimizer rather than
a canonical receding-horizon controller.  Full CMA-ES
\citep{hansen2001cma} treats the flattened cloud as one \(256\)-dimensional
individual and applies positive-weight full-covariance and cumulative
step-size adaptation.  Its population is fixed at ten for exact rollout-budget
matching rather than using the dimension-dependent default. Within every condition
and seed, all methods start from the same serialized cloud and return the
lowest-objective cloud visited.

\paragraph{Hyperparameters and budget accounting.}
The task-specific settings were selected for each method's best performance. Table~\ref{tab:pusht-hyperparameters} gives the frozen method parameters. 

\begin{table}[H]
\centering
\small
\setlength{\tabcolsep}{5.0pt}
\caption{Frozen Push-T optimizer parameters.  Every method performs 50 updates
and receives exactly 4,408 optimization simulator calls.}
\label{tab:pusht-hyperparameters}
\begin{tabular}{lll}
\toprule
Method & Population/update structure & Selected parameters \\
\midrule
SCMO & $R=2$, $S=4$, $M=5$, $L=10$
& proposal $.075$, execution $.01$, temperature $.3$, $T_{\rm alg}=1$ \\
CEM cloud & 10 complete clouds
& elite fraction $.3$, init. std. $.075$, floor $.01$, smoothing $.5$ \\
Random-search cloud & 11 complete clouds
& proposal scale $.04$ \\
MPPI-style cloud & 10 complete clouds
& proposal $.075$, step $.6$, temperature $.1$ \\
Full CMA-ES cloud & 10 complete clouds
& parent fraction $.3$, initial scale $.075$, eigenvalue floor $10^{-12}$ \\
\bottomrule
\end{tabular}
\end{table}

The budget is exact at the simulator-call level.  The initial cloud costs eight
rollouts.  Each of SCMO's $LM=50$ updates costs
$N\{R+RS+1\}=8(2+8+1)=88$ rollouts, so its total is
$8+50\times88=4{,}408$.  CEM, MPPI, and full-covariance CMA-ES each evaluate
ten complete clouds plus one monitor cloud per update, also costing
$8(10+1)=88$ rollouts.  Random search evaluates 11 complete clouds, again
costing $11\times8=88$. No early stopping is used, and every
method returns its best visited eligible incumbent.

\paragraph{Results.}
The results are summarized in Table~\ref{tab:pusht-results}.  SCMO has the
lowest selected incumbent objective and the highest successful mass.  Its
seed-macro objective is
$0.0207\pm0.0036$, compared with $0.0619\pm0.0141$ for the strongest baseline,
CEM; the SCMO-minus-CEM difference is $-0.0412$ with a $95\%$ paired-\(t\) interval
$[-0.0512,-0.0312]$.  SCMO also has the lowest mean selected objective in each
of the five fixed conditions.  Both SCMO and CEM meet both route quotas in all
50 state--seed runs and in all five conditions for every one of the ten seeds.
SCMO makes all eight particles successful in 36/50 runs, compared with 27/50
for CEM, while Random Search, MPPI, and full CMA-ES do so in none. Figure~\ref{fig:pusht-representative-rollouts} replays
all eight selected particles for every method at the unit-scale condition.

\begin{table}[H]
\centering
\footnotesize
\setlength{\tabcolsep}{4.4pt}
\caption{Push-T final evaluation.  Continuous entries are mean $\pm$ sample
standard deviation over ten seed-level macro-averages across the five fixed
conditions.  Event counts pool the 50 state--seed runs.  Lower
objective and higher successful mass are better; bold and underline mark the
best and second-best continuous means.}
\label{tab:pusht-results}
\begin{tabular}{lrrrrrr}
\toprule
Method & Objective $\downarrow$ & Success mass $\uparrow$
& Both routes & Both quotas & All $8$ succeed \\
\midrule
SCMO
& $\bm{0.021\pm0.004}$ & $\bm{0.965\pm0.027}$ & 50/50 & 50/50 & 36/50 \\
CEM cloud
& $\underline{0.062\pm0.014}$ & $\underline{0.920\pm0.059}$ & 50/50 & 50/50 & 27/50 \\
Random-search cloud
& $0.172\pm0.043$ & $0.585\pm0.058$ & 50/50 & 45/50 & 0/50 \\
MPPI-style cloud
& $0.227\pm0.026$ & $0.520\pm0.047$ & 50/50 & 28/50 & 0/50 \\
Full CMA-ES cloud
& $0.318\pm0.038$ & $0.420\pm0.037$ & 49/50 & 6/50 & 0/50 \\
\bottomrule
\end{tabular}
\end{table} 
\begin{figure}[H]
\centering
\includegraphics[width=\textwidth]{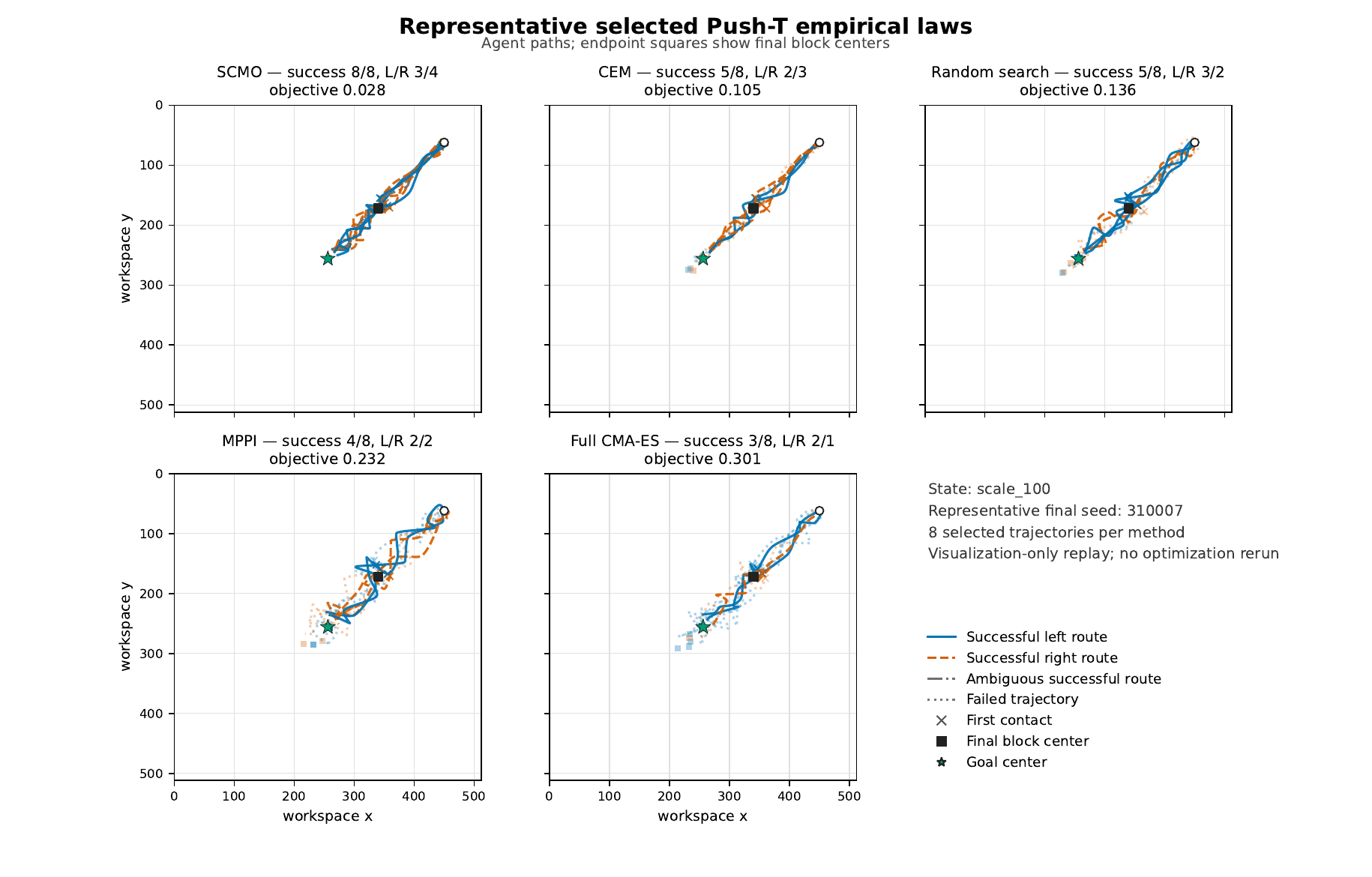}
\caption{Visualization-only replays of the eight selected particles for each
method under the deterministically selected representative seed 310007 at
scale $1.00$.}
\label{fig:pusht-representative-rollouts}
\end{figure}
\section{Additional illustration: latent distribution matching}
\label{app:latent-matching}

We include an additional learned-representation experiment to illustrate that SCMO can optimize an
empirical law in a latent space. This is not intended as a generative-modeling benchmark. We use a
pretrained convolutional VAE on MNIST and keep both the encoder and decoder fixed. Training
images are encoded into \(d=16\)-dimensional latent vectors, and the resulting empirical latent law is
used as the target \(\nu_{\mathrm{MNIST}}\). SCMO evolves $\mu_t^N=\frac1N\sum_{i=1}^N\delta_{z_t^i}$ by minimizing $G(\mu_t^N)=\operatorname{MMD}^2(\mu_t^N,\nu_{\mathrm{MNIST}})$ with a Gaussian kernel. The decoder is used only for visualization after the latent particles have
been optimized.

For the direct SCMO run, we use \(12000\) target latents obtained by encoding MNIST training images, kernel bandwidth \(1.0\), \(N=64\), \(R=1\), \(S=8\), \(M=64\), \(L=8\), horizon \(T=1\), and \(\varepsilon=10^{-8}\). The latent covariance is
diagonal with decay parameter \(0.12\), proposal scale \(\sigma_{\mathrm{prop}}=1.0\), and execution scale \(\sigma_{\mathrm{dyn}}=0.08\).

We also test whether the expensive candidate-evaluation step can be amortized. We collect \(32{,}000\) cloud transitions from \(100\) independent
teacher SCMO rollouts, each using \(N=128\), \(R=1\), \(S=128\),
\(M=64\), and \(L=5\). For each step, we
store the time feature $\phi_k=\left(\frac{\tau_k}{T},\frac{\ell}{L-1}\right)$, the current cloud \(Z_k=(z_k^1,\ldots,z_k^N)\), and the SCMO displacement $D_k=(d_k^1,\ldots,d_k^N),
    \qquad
    d_k^i=\Delta t\,\widehat\theta_k^i$. The amortized model is a DeepSets-type permutation-equivariant network. It computes a set summary $h(Z)=\frac1N\sum_{j=1}^N\psi(z^j)$, and predicts the displacement of particle \(i\) by $\widehat d_\eta^i
    =
    \rho_\eta\!\left(\phi_k,z^i,h(Z),\overline z,\operatorname{Var}(Z)\right)$. The model is trained to imitate the law-level SCMO transition using
\[
\mathcal L(\eta)
=
\lambda_{\mathrm{law}}\,
\operatorname{MMD}^2\!\left(
\frac1N\sum_{i=1}^N\delta_{z_k^i+\widehat d_\eta^i},
\frac1N\sum_{i=1}^N\delta_{z_k^i+d_k^i}
\right)
+
\lambda_{\mathrm{disp}}\,
\frac1N\sum_{i=1}^N
\|\widehat d_\eta^i-d_k^i\|^2,
\]
with \(\lambda_{\mathrm{law}}=1.0\), \(\lambda_{\mathrm{disp}}=0.5\), and MMD bandwidth \(1.0\). The
network is trained for \(150\) epochs with Adam, learning rate \(10^{-3}\), and batch size \(128\).

At inference time, the learned network replaces the SCMO candidate pool and replacement-score
evaluation. Starting from a fresh Gaussian latent cloud, we iterate the learned measure-flow map and
then decode the final latent particles. Figure~\ref{fig:mnist_amortized_scmo} shows decoded
samples from the true SCMO teacher and the amortized network. Both rows contain recognizable
samples across all ten digit classes, suggesting that the learned permutation-equivariant map captures
the law-level SCMO evolution well enough to preserve broad latent-mode coverage.
\begin{figure}[H]
    \centering
    \includegraphics[width=0.88\linewidth]{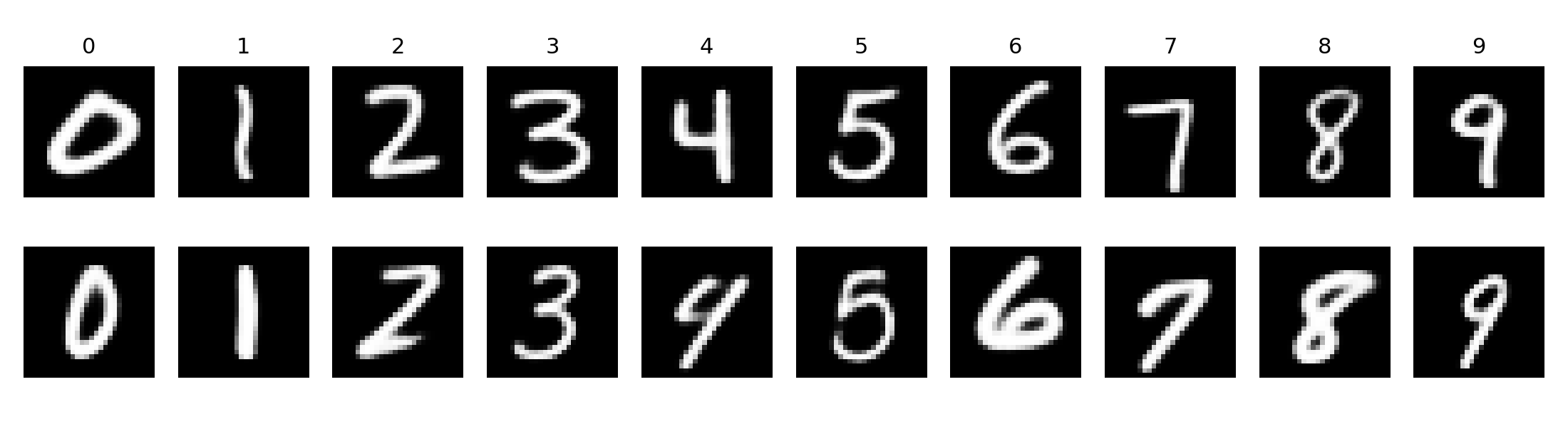}
    \caption{
    Additional latent distribution-matching illustration on MNIST. A pretrained VAE is kept fixed,
    and SCMO optimizes an empirical law over \(16\)-dimensional latent vectors toward the encoded
    MNIST latent distribution using an MMD objective. The first row shows decoded samples from
    a true SCMO rollout. The second row shows decoded samples from the learned DeepSets
    measure-flow network, which imitates SCMO cloud transitions and avoids candidate replacement
    scoring at inference time. Columns are grouped by classifier-predicted digit class
    \(0,\ldots,9\); the classifier is used only for visualization.
    }
    \label{fig:mnist_amortized_scmo}
\end{figure}
\end{document}